\pdfoutput=1 
\documentclass[hidelinks,onefignum,onetabnum]{siamart220329}

\usepackage{filecontents}
\usepackage{biblatex}
\usepackage{xr-hyper}
\usepackage{hyperref}

\usepackage{lipsum}
\usepackage{amsfonts}
\usepackage{graphicx}
\usepackage{epstopdf}
\usepackage{algorithmic}
\ifpdf
  \DeclareGraphicsExtensions{.eps,.pdf,.png,.jpg}
\else
  \DeclareGraphicsExtensions{.eps}
\fi

\usepackage[utf8]{inputenc}
\usepackage{textcomp}

\usepackage{amsmath,amssymb}
\usepackage{mathtools}
\usepackage{altvars2}
\usepackage{tabularx}

\usepackage{multirow}
\usepackage{caption} 
\usepackage{xcolor}
\definecolor{heteroppca}{RGB}{ 53, 96,240}
\definecolor{homoppca}{RGB}{255,160,253}
\definecolor{group1}  {RGB}{ 73,227,176}
\definecolor{group2}  {RGB}{255,166, 23}
\definecolor{inv}  {RGB}{ 10,112, 56}
\definecolor{sqinv}{RGB}{247,  5,  5}

\usepackage{graphicx}
\usepackage{bmpsize}
\usepackage{tikz}
\usepackage{subfig}

\usepackage{xparse}
\usepackage{xspace}

\usepackage{todonotes}
\presetkeys{todonotes}{inline,prepend,caption={todo}}{}

\usepackage[normalem]{ulem}
\usepackage{bm}

\usepackage{algorithm}
\usepackage{algorithmic}

\usepackage{multicol}
\usepackage{wrapfig}

\newcommand{\new}[1] {{\color{blue}#1}}

\newcommand{\edit}[1]{\textcolor{black}{#1}}
\newcommand{\visedit}[1] {{\color{black}#1}}
\newcommand{\revise}[1] {{\color{black}#1}}
\long\def\new#1{\bgroup\color{black}#1\egroup} 

\newcommand{\argmin}{\operatorname*{argmin}}
\newcommand{\argmax}{\operatorname*{argmax}}

\newcommand{\tr}{\operatorname{tr}}
\newcommand{\I}{\mathbf{I}}

\newcommand{\R}{\mathbb{R}}

\newcommand{\Y}{\mathbf{Y}}

\newcommand{\X}{\mathbf{X}}
\newcommand{\A}{\mathbf{A}}
\newcommand{\B}{\mathbf{B}}
\newcommand{\D}{\mathbf{D}}
\newcommand{\M}{\blmath{M}}

\newcommand{\K}{{\cal K}}
\newcommand{\Cc}{{\cal C}}
\newcommand{\rank}{{\mathrm{rank}}}

\newtheorem{assumption}{Assumption}[theorem]

\newcommand{\xmath}[1] {\ensuremath{#1}\xspace}
\newcommand{\blmath}[1] {\xmath{\mathbf{#1}}}

\newcommand{\sym}[1] {\xmath{\mathrm{sym}\left( #1 \right)}}

\newcommand{\E}{\blmath{E}}

\newcommand{\N}{\blmath{N}}
\newcommand{\Ss}{\blmath{S}}
\newcommand{\Cb}{\overline{\bmC}}
\newcommand{\Qb}{\overline{\bmQ}}

\newcommand{\Ub}{\overline{\bmU}}

\newcommand{\ub}{\overline{\bmu}}
\newcommand{\Lambdab}{\overline{\bmLambda}}

\newcommand{\Xb}{\overline{\bmX}}
\newcommand{\Zb}{\overline{\bmZ}}
\newcommand{\Yb}{\overline{\bmY}}
\newcommand{\nub}{\overline{\bmnu}}
\newcommand{\nubi}{\overline{\nu}}
\newcommand{\Mb}{\overline{\M}}
\newcommand{\bb}{\overline{\bmb}}
\newcommand{\cb}{\overline{\bmc}}

\makeatletter
\newcommand*\bigcdot{\mathpalette\bigcdot@{.5}}
\newcommand*\bigcdot@[2]{\mathbin{\vcenter{\hbox{\scalebox{#2}{$\m@th#1\bullet$}}}}}
\makeatother

\DeclareMathOperator{\Opt}{Opt}
\DeclareMathOperator{\Feas}{Feas}

\DeclareMathOperator{\Relint}{Relint}

\newcommand{\Lc}{{\cal L}}

\DeclareMathOperator{\Range}{Range}

\DeclareMathOperator{\Ext}{Ext}

\newcommand{\SSS}{{\mathbb{S}}}

\newsiamremark{remark}{Remark}
\newsiamremark{hypothesis}{Hypothesis}
\crefname{hypothesis}{Hypothesis}{Hypotheses}
\newsiamthm{claim}{Claim}

\headers{SDP for SHQF on the Stiefel}{K. Gilman, S. Burer, and L. Balzano}

\title{A Semidefinite Relaxation for Sums of Heterogeneous Quadratic Forms on the Stiefel Manifold\thanks{
\funding{K. Gilman and L. Balzano were supported in part by ARO YIP award W911NF1910027, AFOSR YIP award FA9550-19-1-0026, and NSF BIGDATA award IIS-1838179. L. Balzano was additionally supported by NSF award 2331590.}}}

\author{Kyle Gilman\thanks{Department of Electrical and Computer Engineering, University of Michigan, Ann Arbor, MI 
  (\email{kgilman@umich.edu}).}
\and Samuel Burer\thanks{Department of Business Analytics, University of Iowa, Iowa City, IA 
  (\email{samuel-burer@uiowa.edu}).}
\and Laura Balzano\footnotemark[2]}

\usepackage{amsopn}
\DeclareMathOperator{\diag}{diag}

\ifpdf
\hypersetup{
  pdftitle={A Semidefinite Relaxation for Sums of Heterogeneous Quadratic Forms on the Stiefel Manifold},
  pdfauthor={K. Gilman, S. Burer, and L. Balzano}
}
\fi

\makeatletter
\newcommand*{\addFileDependency}[1]{
  \typeout{(#1)}
  \@addtofilelist{#1}
  \IfFileExists{#1}{}{\typeout{No file #1.}}
}
\makeatother

\begin{document}


\maketitle

\begin{abstract}
We study the maximization of sums of heterogeneous quadratic forms over the Stiefel manifold, a nonconvex problem that arises in several modern signal processing and machine learning applications such as heteroscedastic probabilistic principal component analysis (HPPCA). In this work, we derive a novel semidefinite program (SDP) relaxation of the original problem and study a few of its theoretical properties. We prove a global optimality certificate for the original nonconvex problem via a dual certificate, which leads to a simple feasibility problem to certify global optimality of a candidate solution on the Stiefel manifold. In addition, our relaxation reduces to an assignment linear program  for jointly diagonalizable problems and is therefore known to be tight in that case. We generalize this result to show that it is also tight for close-to jointly diagonalizable problems, and we show that the HPPCA problem has this characteristic. Numerical results validate our global optimality certificate and sufficient conditions for when the SDP is tight in various problem settings.
\end{abstract}



\section{Introduction}
\label{s:intro}

This paper studies the problem known in the literature as {\em the maximization of sums of heterogeneous quadratic forms over the Stiefel manifold\/}\footnote{We note here that ``heterogeneous" refers to the fact that the $\bmM_i$ are distinct and the problem is not separable in each $\bmu_i$. Indeed, the objective in \cref{eq:sum_heterogeneous_quadratics} is a homogeneous polynomial in the entries of $\bmU$ since all terms are degree 2.} \cite{brockett1989least, bolla:98, rapcsak2002minimization, Berezovskyi2008}. Specifically, given $d \times d$ symmetric positive semidefinite (PSD) matrices $\bmM_1,\hdots,\bmM_k \succeq 0$ for $k < d$, we wish to maximize the convex objective function $\sum_{i=1}^k  \bmu_i' \bmM_i \bmu_i$ over the nonconvex constraint that $\bmU = [\bmu_1 \cdots \bmu_k] \in \bbR^{d\times k}$ has orthonormal columns: 
\begin{align}
\label{eq:sum_heterogeneous_quadratics}
    \max_{\bmU \in \text{St}(k,d) }\sum_{i=1}^k  \bmu_i' \bmM_i \bmu_i,
\end{align}
where $\text{St}(k,d) = \{\bmU \in \bbR^{d \times k} : \bmU' \bmU = \I_k \}$ denotes the Stiefel manifold. This problem arises in modern signal processing and machine learning applications like heteroscedastic probabilistic principal component analysis (HPPCA) \cite{hong2021heppcat}, generalized PPCA \cite{gu2020generalized}, heterogeneous
clutter in radar sensing \cite{SunBreloy2016}, and robust sparse PCA \cite{breloyMMStiefel2021}. Each of these applications involves learning a signal subspace for data possessing heterogeneous statistics.

In particular, HPPCA \cite{hong2021heppcat} models data collected from sources of varying quality with different additive noise variances, and estimates the best approximating low-dimensional subspace by maximizing the likelihood, providing superior estimation compared to standard PCA. Specifically, we are given $L$ data groups $(\bmY_1,\hdots,\bmY_L)$ \edit{where each $\bmY_\ell \in \bbR^{d \times n_\ell}$ represents a matrix of $n_\ell$ samples of a $d$-dimensional signal plus additive Gaussian noise with variance $v_\ell$}. Using second-order statistics $\bmA_\ell: = \frac{1}{v_\ell}\bmY_\ell \bmY_\ell' \succeq 0$ for $\ell \in [L]$ and known positive weights $w_{\ell,i}$ for $(\ell, i) \in [L] \times [k]$, a subproblem of HPPCA involves optimizing the sum of Brockett cost functions \cite[Section 4.8]{AbsMahSep2008} with respect to a $k$-dimensional orthonormal basis $\bmU$, and can be equivalently recast in the form \cref{eq:sum_heterogeneous_quadratics} as follows:
\begin{align}
\label{eq:origobj}
    \max_{\bmU: \bmU' \bmU = \I} \sum_{\ell=1}^L \tr(\bmU' \bmA_\ell \bmU \bmW_\ell ) = \max_{\bmU: \bmU' \bmU = \I} \sum_{\ell=1}^L \sum_{i=1}^k  w_{\ell,i} \bmu_i'  \bmA_\ell \bmu_i =\max_{\bmU: \bmU' \bmU = \I}\sum_{i=1}^k  \bmu_i' \bmM_i \bmu_i,
\end{align}
where $\bmW_\ell := \text{diag}(\{w_{\ell,i}\}_{i=1}^k)$ for all $\ell$ and $\bmM_i := \sum_{\ell=1}^L w_{\ell,i} \bmA_\ell$ for all $i$. Other sensing problems such as independent component analysis (ICA) \cite{theis_ica_jd} and approximate joint diagonalization (AJD) \cite{pham:hal-00371941} also model data with heterogeneous statistics and optimize objective functions of a similar form, as we discuss in Section \ref{s:related_work}.

For (\ref{eq:origobj}), the case of a single Brockett cost function ($L=1$) has a known analytical solution obtained by the SVD or eigendecomposition  \cite[Section 4.8]{AbsMahSep2008}, whereas analytical solutions are not known for $L \ge 2$. Indeed, for $L \ge 2$ and general $\bmA_\ell$, few, if any, guarantees for optimal recovery exist except in special cases, such as when the constructed $\bmM_i$ commute \cite{bolla:98}. Generally speaking, existing theory only gives restrictive sufficient conditions for global optimality that are typically  difficult to check in practice. Given that (\ref{eq:sum_heterogeneous_quadratics}) is nontrivial and challenging in several ways---nonconvex due to the Stiefel manifold constraint, non-separable because of the weighted sum of objectives, and not readily solved by singular value or eigenvalue decomposition---many works apply iterative local solvers to (\ref{eq:sum_heterogeneous_quadratics}). 

However, given the nonconvexity of \cref{eq:sum_heterogeneous_quadratics}, these local approaches do not find a global maximum in general. An alternative approach is to relax problems such as (\ref{eq:sum_heterogeneous_quadratics}) to a semidefinite program (SDP), allowing the use of standard convex solvers. While the SDP has stronger optimality guarantees, the challenge is then to derive conditions under which the SDP is tight, i.e., returns an optimal solution to the original nonconvex problem. SDP relaxations such as the ``Fantope" \cite{MR308159,MR1146651} exist for solving PCA-like problems, but to the best of our knowledge, no previous convex methods exist to solve (\ref{eq:sum_heterogeneous_quadratics}). 
\setlength\parfillskip{0pt plus .75\textwidth}
\setlength\emergencystretch{1pt}

The main contribution of this paper is a novel convex SDP relaxation of (\ref{eq:sum_heterogeneous_quadratics}), whose constraint set is related to the Fantope but distinct.
By studying this SDP and its optimality criteria, we derive sufficient conditions to certify the global optimality of
\edit{any candidate solution obtained from any iterative solver for the nonconvex problem.}
We then propose a straightforward method to certify global optimality by solving a much smaller SDP feasibility problem that scales favorably with the problem dimension. Our work also generalizes existing results for \cref{eq:sum_heterogeneous_quadratics} with commuting matrices to the case with ``almost commuting'' matrices, showing that as long as the data matrices are within an open neighborhood of a commuting tuple of data matrices (to be defined precisely in Section \ref{s:theory:ss:continuity}), the SDP is tight and provably recovers an optimal solution of (\ref{eq:sum_heterogeneous_quadratics}).

\paragraph{Notation} We use boldface, upper case letters $\bmA$ to denote matrices, boldface, lower case letters $\bmv$ to denote vectors, and italic, lowercase letters $c$ for scalars. We denote the cone of $d \times d$ symmetric positive semidefinite matrices as $\mathbb{S}^d_+$, and use $\bmA \succeq 0$ to denote an element $\bmA \in \mathbb{S}^d_+$. We denote the Hermitian transpose of a matrix as $\bmA'$, the trace of a matrix as $\tr(\bmA)$, and the inner product of matrices (with identical dimensions) $\langle \bmA, \bmB \rangle := \tr(\bmA'\bmB)$. \edit{We also make use of the notation $[\bmA, \bmB] = 0$ for commuting square matrices $\bmA$ and $\bmB$ of the same sizes, which is equivalent to $\bmA \bmB - \bmB \bmA = 0$ where here $0$ is the zero matrix.} The spectral norm of a matrix is denoted by $\|\bmA\|$, the Frobenius norm by $\|\bmA\|_{\mathrm{F}}$, and the trace norm by $\|\bmA\|_{\mathrm{tr}} := \sqrt{\frac{1}{d} \sum_{i,j = 1}^d |\mathbf{A}_{i,j}|^2} = \frac{1}{\sqrt{d}}\|\bmA\|_\mathrm{F}$. The identity matrix of size $d$ is denoted as $\I_d$. \visedit{Finally, we denote} $[k] := \{1,\hdots,k\}$.

\section{Semidefinite program relaxation}
\label{s:sdp}

By relaxing the considered nonconvex problem \cref{eq:sum_heterogeneous_quadratics} to a convex one, the well-established principles of convex optimization permit us to study when an optimal solution of the SDP relaxation recovers a global maximum of \cref{eq:sum_heterogeneous_quadratics} and importantly, when a given local stationary point is a global maximum. After re-expressing the original problem using equivalent constraints, we lift the variables into the cone of PSD matrices, relax the nonconvex constraints to convex surrogates, and obtain an SDP.

First, we begin by slightly rewriting \cref{eq:sum_heterogeneous_quadratics} and the Stiefel manifold constraints as
\begin{align}
    \max_{\bmu_1, \hdots, \bmu_k} & \tr \left( \sum_{i=1}^k \bmM_i \bmu_i \bmu_i'  \right) \quad \text{s.t. } \tr \left( \bmu_i\bmu_i'\right) = 1 \quad \forall i \in [k], \quad \tr\left(\bmu_j\bmu_i'\right) = 0 \quad \forall i\neq j.
\end{align}

\noindent Letting $\bmX_i = \bmu_i\bmu_i' \in \bbR^{d \times d}$
\edit{and using the eigenvalue structure of the rank-$k$ projection matrix $\sum_{i=1}^k \bmu_i\bmu_i'$},
this is equivalent to the lifted problem:
\begin{align}
\begin{split}
    \max_{\bmX_1,\hdots,\bmX_k}  \tr \left( \sum_{i=1}^k \bmM_i \bmX_i \right) \quad
    \text{s.t.}&~  \lambda_j\left(\sum_{i=1}^k \bmX_i\right) \in \{ 0, 1\} \quad \forall j\in[d] \\
    & \tr(\bmX_i) = 1 , \quad {\rank(\bmX_i) = 1}, \quad \bmX_i \succcurlyeq 0 \quad \forall i\in [k],
    \label{eq:ncvxprimal}
\end{split}
\end{align}
where $\lambda_j( \cdot )$ indicates the $j$-th eigenvalue of its argument. Note that this
problem is nonconvex due to {the rank constraint and} the constraint that  the eigenvalues are binary.
Similar to the relaxations in \cite{vu2013fantope, luo2010sdp}, we relax the eigenvalue constraint in \eqref{eq:ncvxprimal}
to $0 \preccurlyeq \sum_{i=1}^k \bmX_i \preccurlyeq \I\;$ {and remove the rank constraint}, which yields the SDP relaxation we consider throughout the remainder of this work:

\begin{align}
    \begin{split}
    \revise{p^* = \max_{\bmX_1,\hdots,\bmX_k}}& \quad \tr\left(\sum_{i=1}^k \bmM_i \bmX_i\right) \label{eq:primal_problem}\quad \\
    &\text{s.t.}~\sum_{i=1}^k \bmX_i \preccurlyeq \I, \quad \tr(\bmX_i) = 1, \quad \bmX_i \succcurlyeq 0 \quad i=1,\dots,k. 
    \end{split}
    \tag{SDP-P}
\end{align}
Note that $0 \preccurlyeq \sum_{i=1}^k \bmX_i$ can be omitted since it is already satisfied when $\bmX_i \succcurlyeq 0$ for all $i$. 

The feasible set of (\ref{eq:primal_problem}) is closely related to the convex set found in \cite{vu2013fantope, luo2010sdp, garber2022efficient} called the {\em Fantope}. The Fantope is the convex hull of all matrices $\bmU \bmU' \in \mathbb{R}^{d \times d}$ such that $\bmU \in \mathbb{R}^{d \times k}$ and $\bmU' \bmU = \I$ \cite{MR308159,MR1146651}.
Indeed, our relaxation can be viewed as providing a decomposition of the Fantope variable $\bmX = \bmU \bmU'$ into the sum $\bmX_1 + \cdots + \bmX_k$ such that each $\bmX_i$ satisfies $\tr(\bmX_i) = 1$ and $0 \preceq \bmX_i \preceq \I$. This decomposition allows (\ref{eq:primal_problem}) to capture the exact form of the objective function, which sums the individual terms $\tr(\bmM_i \bmX_i)$. 

\sloppypar{Precisely, the feasible set of (\ref{eq:primal_problem}) is a convex relaxation of the set $\left\{ (\bmu_1 \bmu_1',\ldots,\bmu_k \bmu_k'): \bmU' \bmU = \I \right\}$. Naturally, one wonders whether our relaxation always solves the original nonconvex problem. We show in Appendix \ref{sec:counterexample} that it does not, \edit{using a counter example that demonstrates our relaxation does not exactly capture the convex hull, which is a necessary condition for the relaxation to be tight for all objectives.} Our work therefore studies this SDP in two ways: first, we provide a global optimality certificate; second, we study a class of ``close-to jointly diagonalizable" problem instances, which includes the heteroscedastic PCA problem, and show that the SDP is tight for this class.}

For dual variables $\bmZ_i \in \mathbb{S}_+^{d}$ for $i \in [k]$, $\bmY \in \mathbb{S}_+^d$, $\bmnu \in \bbR^k$, the dual of \cref{eq:primal_problem}, 
which will play a central role in the theory
of this paper, is
\begin{align}
\label{eq:dual_problem} \tag{SDP-D}
    d^* = \min_{\bmY, \bmZ_i, \bmnu}& \tr(\bmY) + \sum_{i=1}^k \nu_i     
    \quad \text{s.t. }
    \bmY \succcurlyeq 0, \quad \bmY = \bmM_i + \bmZ_i - \nu_i\I, \quad \bmZ_i \succcurlyeq 0 \quad \forall i\in[k].
\end{align}
The derivation of \cref{eq:dual_problem} in \cref{appendix:dual:derivation} follows by standard analysis of the Lagrangian. However, a short proof of weak duality also verifies that \cref{eq:dual_problem} upper bounds \cref{eq:primal_problem}:
\begin{align*}
    \sum_{i=1}^k \tr(\M_i \bmX_i) &= \sum_{i=1}^k \tr((\bmY - \bmZ_i + \nu_i \bmI) \bmX_i) \\
    &= \tr\left( \bmY  \sum_{i=1}^k \bmX_i \right) - \sum_{i=1}^k \tr(\bmZ_i \bmX_i) + \sum_{i=1}^k \nu_i \tr(\bmX_i) \\
    &\leq \tr(\bmY) + \sum_{i=1}^k \nu_i, 
\end{align*}
where the inequality follows from the constraints in \cref{eq:primal_problem} and \cref{eq:dual_problem}. Therefore $p^* \leq d^*$.
Since the constraint set of \cref{eq:primal_problem} is closed and bounded with non-empty interior, and strong duality holds by \revise{the following lemma, then there exists an optimal primal solution to \cref{eq:primal_problem} and optimal dual solution to \cref{eq:dual_problem}.
\begin{lemma} \label{lem:strongdualityholds}
    If $k < d$, strong duality holds for the SDP relaxation with primal \eqref{eq:primal_problem} and dual \eqref{eq:dual_problem}.
\end{lemma}
The proof of this lemma follows from Slater's condition and can be found in \cref{appendix:sdp_relaxation_pfs}.

We now define the ``rank-one property'' of a feasible solution of
\eqref{eq:primal_problem}, which allows us to characterize the
relationship between optimal solutions of 
\eqref{eq:primal_problem} and optimal solutions of the original
nonconvex problem.}

\begin{definition}[Rank-one property (ROP)] 
    A feasible solution to \cref{eq:primal_problem} is said to have the rank-one property if $\bmX_1,\hdots,\bmX_k$ are all rank-one.
\end{definition}
We note that if a feasible solution has the rank-one property, the first singular vectors of the $\bmX_i$ are necessarily mutually orthogonal, and $\sum_i \bmX_i$ is a rank-$k$ projection matrix, due to the constraint $\sum_i \bmX_i \preccurlyeq \I$.  \revise{The following lemma establishes the relationship between the properties of the optimal solutions of (SDP-P) to those of the original nonconvex problem.}
\revise{
\begin{lemma}
\label{lem:projtight}
    \revise{An} optimal solution $\bmX^* := (\bmX^*_1,\hdots,\bmX^*_k)$ to the SDP relaxation in \eqref{eq:primal_problem} is \revise{an} optimal solution to the original nonconvex problem in \eqref{eq:sum_heterogeneous_quadratics} (equivalently \eqref{eq:ncvxprimal}) if and only if $\bmX^*$ has the rank-one property.
\end{lemma}
The proof of this lemma can be
found in \cref{appendix:sdp_relaxation_pfs}.
The next lemma now relates the properties of the optimal solutions to (SDP-D) to optimal solutions of (SDP-P) with the ROP.
}

\revise{\begin{lemma}
\label{lem:Zrank_orthoXi}
    If the optimal dual variables $\bmZ_i^*$ for $i=1,\hdots,k$ each have rank $d-1$, the optimal solution 
    $\bmX^* := (\bmX^*_1,\hdots,\bmX^*_k)$ has the rank-one property.
\end{lemma}
The proof of this result is also in \cref{appendix:sdp_relaxation_pfs}, and it follows directly from complementary slackness. This key result, through careful analysis of the dual problem, will later allow us to characterize problem instances with ROP solutions, which by \cref{lem:projtight}, are optimal solutions to the nonconvex problem.
}

\section{Related work}
\label{s:related_work}

There are a few important related works on the objective in \cref{eq:sum_heterogeneous_quadratics}, as well as many more than can be reviewed here, including ones on eigenvalue/eigenvector problems and their variations, low-rank SDPs, and nonconvex quadratics where $\bmM_i$ are not PSD. For the curious reader, Section \ref{appendix:related} in the supplement provides a more extensive related work section. Here, we focus on the works most directly related to \cref{eq:sum_heterogeneous_quadratics}.

The papers \cite{bolla:98,rapcsak2002minimization,Berezovskyi2008} previously investigated the sum of heterogeneous quadratic forms in \cref{eq:sum_heterogeneous_quadratics}. The work in \cite{bolla:98} only studied the structure of this problem when the matrices $\bmM_i$ were commuting. The work in \cite{rapcsak2002minimization} derived sufficient second-order global optimality conditions, but these conditions are difficult to check in general and, for example, do not seem to hold for the heteroscedastic PCA problem. 
Works such as \cite{huang2009rank} and \cite{Pataki1998OnTR} consider a very similar problem to \eqref{eq:origobj}, but without the eigenvalue constraint in \eqref{eq:ncvxprimal}, making their SDP a rank-constrained separable SDP; see also \cite[Section 4.3]{luo2010sdp}. 
Pataki \cite{Pataki1998OnTR} studied upper bounds on the rank of optimal solutions of general SDPs, but in the case of \cref{eq:primal_problem}, since our problem introduces the additional constraint summing the $\bmX_i$, Pataki's bounds do not guarantee rank-one, or even low-rank, optimal solutions.

\visedit{A recent paper \cite{cifuentes2022stability} analyzes general sufficient conditions under which an SDP relaxation, which has a rank-one optimal solution, retains a rank-one optimal solution after the perturbation of the objective and/or constraint data. The analysis in \cite{cifuentes2022stability} does not seem to apply directly to our own work for two reasons: (i) the authors of \cite{cifuentes2022stability} analyze the basic Shor relaxation, 
a natural and popular SDP relaxation for quadratically constrained quadratic programs, which we show in \cref{appendix:shor_relax} is trivially not tight in our setting;
and (ii) their relaxation has a single-block matrix variable, which is analyzed to be rank-one at optimality, whereas we analyze several blocks $\X_1, \ldots, \X_k$, each of which is rank-one at optimality when the SDP is tight.}

Recent works have also studied convex relaxations of PCA and other low-rank subspace problems that seek to bound the eigenvalues of a single matrix  \cite{vu2013fantope,Morgenstern2019, won2021orthogonalSiamJMAA}, rather than the sum of multiple matrices as in our setting.  
The works in \cite{boumal_nips2016, Pumir2018SmoothedAO} {show that nonconvex Burer–Monteiro factorizations \cite{burer2003nonlinear}, which solve low-rank SDPs without orthogonality constraints, have no spurious local minima and that approximate second-order stationary points are approximate global optima.}
Other works have studied algorithms to optimize the nonconvex problem, like those in \cite{breloyMMStiefel2021,breloyRobustCovarianceHetero2016,SunBreloy2016,BreloyClutterSubspace2015,hong2021heppcat}, using minorize-maximize or Riemannian gradient ascent algorithms, which do not come with {global} optimality guarantees.
Our problem also has interesting connections to approximate joint diagonalization (AJD), which is well-studied and often applied to blind source separation or independent component analysis (ICA) problems \cite{theis_ica_jd, bouchard_malick_congedo2018, kleinsteuber_shen_2013, afsari_jd2008, Shi2011}. See \cref{appendix:related} of the supplement for further details.

\section{Theoretical Results}
\label{s:theory}

\subsection{Dual certificate of the SDP}
\label{s:theory:ss:dual_certificate}

In practical settings for high-dimensional data, a variety of iterative local methods are often applied to solve nonconvex problems over the Stiefel manifold, from gradient ascent by geodesics \cite{AbsMahSep2008,edelman1998geometry, abrudan2008steepestdescent} to majorization-minimization (MM) algorithms, where \cite{breloyMMStiefel2021} applied MM methods to solve \cref{eq:sum_heterogeneous_quadratics} with guarantees of convergence to a stationary point. While the computational complexity and memory requirements of these solvers scale well, their obtained solutions lack any global optimality guarantees. We seek to fill this gap by proposing a check for global optimality of a local solution.\footnote{\revise{To be clear, while our work does not guarantee that a local solution is globally optimal, we propose a certificate based on a sufficient condition to check if the local solution is globally optimal.}} \revise{Similar types of problems for running fast probabilistic algorithms and checking whether the candidate solution is the optimal solution to the convex relaxation also appear in \cite{bandeira2016note}.}

By \cref{lem:projtight}, an optimal solution of \cref{eq:primal_problem} with rank-one matrices $\bmX_i$ globally solves the original nonconvex problem \cref{eq:sum_heterogeneous_quadratics}. 
In this section, given a candidate $\Ub = [\ub_1 \cdots \ub_k] \in \text{St}(k,d)$ to \cref{eq:sum_heterogeneous_quadratics}, we investigate conditions guaranteeing that the rank-one matrices $\Xb_i = \ub_i \ub_i'$, which are feasible for \cref{eq:primal_problem}, in fact comprise an optimal solution of \cref{eq:primal_problem}, implying that $\Ub$ optimizes \cref{eq:sum_heterogeneous_quadratics}. Similar to \cite{won2022orthogonal,won2021orthogonalSiamJMAA,kyfan} for Fantope problems, our results yield a dual SDP certificate to verify the primal optimality of the candidates $\Xb_1, \hdots, \Xb_k$ constructed from a local solution $\Ub$. We show our certificate scales favorably in computation compared to the full SDP, with the most complicated computations of our algorithm requiring us to solve a feasibility problem in $k$ variables with several $d \times d$ linear matrix inequalities (LMI).

\begin{theorem}
\label{thm:dual_certificate}
\sloppypar{Let $\Ub \in \textup{St}(k,d)$, and let $\Lambdab = \sym{\sum_{i=1}^k \Ub'\bmM_i \Ub \E_i}$, where \visedit{$\sym{\bmA} := \frac{1}{2}(\bmA + \bmA')$,} $\E_i \triangleq \bme_i \bme_i'$ where $\bme_i$ is the $i^{\text{th}}$ standard basis vector in $\bbR^k$, {and $\bmM_i \succeq 0$ for all $i \in [k]$}. If there exist $\nub = [\nubi_1 \cdots \nubi_k] \in \bbR^k$ such that}
\begin{align}
\begin{split}
\label{eq:dual_certificate}
    \Ub(\Lambdab - \D_{\nub})\Ub' + \nubi_i \I - \bmM_i &\succeq 0 \quad \forall i=1,\hdots,k \\
    \Lambdab - \D_{\nub} &\succeq 0,
\end{split}
\end{align}
where $\D_{\nub} := \mathrm{diag}(\nubi_1,\hdots, \nubi_k)$, then $\Ub$ is a globally optimal solution to the original nonconvex problem \cref{eq:sum_heterogeneous_quadratics}. 
\end{theorem}

\noindent \revise{The proof, found in \cref{appendix:optimality_conditions},
uses the Karush-Kuhn-Tucker (KKT) conditions along with the conditions on $\nub$ to construct a dual certificate of SDP optimality. We note that \cref{thm:dual_certificate} is based on a strong sufficient condition, which in particular implies that any feasible $\Ub$ satisfying \cref{eq:dual_certificate} is a second-order stationary point.}

In light of \cref{thm:dual_certificate}, to test whether a candidate $\Ub$ is globally optimal, we simply assess whether system \cref{eq:dual_certificate} is feasible using an LMI solver. If it is indeed feasible, then $\Ub$ is globally optimal. On the other hand, if \cref{eq:dual_certificate} is infeasible,
it indicates one of two things: 1) The SDP is not tight, i.e., the SDP strictly upper bounds the original problem. The candidate $\Ub$ may or may not be globally optimal to the original nonconvex problem. 2) The SDP is tight, but the candidate $\Ub$ is a suboptimal local solution. 
Section \ref{appendix:sums_brocketts_linear_terms} also describes an extension of the certificate to the sum of Brocketts with additive linear terms. 

\revise{It is important to note that \cref{thm:dual_certificate} implies $\Ub$ is an \textit{exact} second-order stationary point. Since in practice it is not possible to obtain exact stationary points using numerical solvers, one may wonder if Theorem \ref{thm:dual_certificate} can be applied in practice.
However, given some $\Ub \in \mathrm{St}(k,d)$ obtained by a solver that only approximately satisfies dual feasibility, we can precisely characterize the suboptimality of this solution. To this end, we provide a corollary to \cref{thm:dual_certificate}, whose proof can be found in \cref{appendix:optimality_conditions}, where the semidefinite constraints are only approximately satisfied.
}

\begin{corollary}
\label{thm:dual_certificate_eps}
Let $\Ub \in \textup{St}(k,d)$ be a {feasible point} of \cref{eq:sum_heterogeneous_quadratics}, and let $\Lambdab = \sym{\sum_{i=1}^k \Ub'\bmM_i \Ub \E_i}.$
Let $\epsilon^*$ be the optimal value of
\begin{align}
\begin{split}
\label{eq:dual_certificate_eps}
    \min_{\epsilon \geq 0, ~\nub \in \bbR^k} ~ \epsilon \quad \mathrm{s.t.} \quad \Ub(\Lambdab - \D_{\nub})\Ub' + \nubi_i \I - \bmM_i &\succeq -\epsilon \I \quad \forall i=1,\hdots,k \\
    \Lambdab - \D_{\nub} &\succeq -\epsilon \I,
\end{split}
\end{align}
where $\D_{\nub} := \mathrm{diag}(\nubi_1,\hdots, \nubi_k)$. Then $\Ub$ is a near optimal solution to the original nonconvex problem \cref{eq:sum_heterogeneous_quadratics} in the sense that its objective value is bounded below by $p^* - \epsilon^* d$.
\end{corollary}


While SDP relaxations of nonconvex optimization problems can provide strong provable guarantees, their practicality can be limited by the time and space required to solve them, particularly when using off-the-shelf interior-point solvers, which in our case require $\mathcal{O}(d^3)$ \cite{benTal_aharon_nemirovski2001} storage and floating point operations (flops) per iteration of \cref{eq:dual_problem}.
The proposed global certificate in \cref{eq:dual_certificate} significantly reduces the number of variables from $\mathcal{O}(d^2)$ in \cref{eq:dual_problem} (upon eliminating the variables $\bmZ_i$) to merely $k$ variables in \cref{eq:dual_certificate}. Using \cite[Section 6.6.3]{benTal_aharon_nemirovski2001} it can be shown that computing the certificate only, based on a given $\Ub$, results in a substantial reduction in flops by a factor of $\mathcal{O}(d^3/k)$ over solving \cref{eq:dual_problem}. Subsequently, an MM solver with complexity on par with standard first-order based methods \cite{breloyMMStiefel2021}, whose cost is $\mathcal{O}(dk^2 + k^3)$ {per iteration}, combined with our global optimality certificate, is preferable to solving the full relaxation \cref{eq:primal_problem} for large problems. See \cref{appendix:complexity} for more details.

\subsection{SDP tightness in the close-to jointly diagonalizable (CJD) case}
\label{s:theory:ss:continuity}

While Section \ref{s:theory:ss:dual_certificate} provides a technique to certify the global optimality of a solution to the nonconvex problem, the check will fail if the point is not globally optimal or if the SDP is not tight. General conditions on $\bmM_i$ that guarantee tightness of \cref{eq:primal_problem} are still not known. However, when the matrices $\bmM_i$ are jointly diagonalizable, our problem reduces to a linear programming assignment problem \cite{bolla:98}, and by standard LP theory, a solution with rank-one $\X_i$ exists and the SDP (or equivalent LP) is a tight relaxation \cite{bolla:98}.

Our next major contribution is to show that a solution with rank-one $\X_i$ exists also for cases that are \emph{close-to jointly diagonalizable} (CJD).
We first give a continuity result showing there is a neighborhood around the diagonal case for which \cref{eq:primal_problem} is still tight. Then we show that for the HPPCA problem, the matrices $\bmM_i$ are close-to jointly diagonalizable and can be made arbitrarily close as the number of data points $n$ grows or as the noise levels diminish or become homoscedastic. This gives strong theoretical support for the tightness of the SDP for the HPPCA problem when $n$ is large or the noise levels are small or close in value.

\begin{definition}[Close-to jointly diagonalizable (CJD)]
We say that unit spectral-norm, symmetric matrices $\bmA$ and $\bmB$ are CJD if they are almost commuting, that is, when the commuting distance \edit{measured by some norm $\|\cdot\|$}, between $\bmA$ and  $\bmB$ is significantly less than 1: $$\|[\bmA, \bmB]\| := \|\bmA \bmB - \bmB \bmA\| \leq \delta \qquad \text{for some}~~0 < \delta \ll 1.$$
\end{definition}
The matrices $\bmA$ and $\bmB$ are jointly diagonalizable if and only if they commute, i.e., the commuting distance is zero.

\subsubsection{Continuity and tightness in the CJD case}

 In this section, we employ a technical continuity result for the dual feasible set to conclude that there is a neighborhood of problem instances around every diagonal instance for which \cref{eq:primal_problem} gives rank-one optimal solutions $\bmX_i$. All proofs for the results in this subsection are found in \cref{sec:continuity}.

Given a $k$-tuple of symmetric matrices $(\bmM_1,\ldots,\bmM_k)$, our
primal-dual pair is given by \cref{eq:primal_problem} and
\cref{eq:dual_problem}.
\noindent Note that, without loss of generality, we may assume each
$\bmM_i$ is positive semidefinite since the primal constraint $\tr(\bmX_i) =
1$ ensures that replacing $\bmM_i$ by $\bmM_i + \lambda_i \bmI \succeq 0$, where
$\lambda_i$ is a positive constant, simply shifts the objective value
by $\lambda_i$.
Thus, we assume $\bmM_i
\succeq 0$ for all $i=1,\ldots,k$.

For a fixed, user-specified upper bound $\mu > 0$, we define the closed convex set
\[
    \Cc :=
    \{\bmc = (\bmM_1,\ldots,\bmM_k) : 0 \preceq \bmM_i \preceq \mu \bmI \quad \forall \ i = 1,\ldots,k \},
\]
to be our set of admissible coefficient $k$-tuples. We know that both \cref{eq:primal_problem} and \cref{eq:dual_problem} have interior
points for all $\bmc \in \Cc$, so that strong duality holds for all $\bmc \in \Cc$. \new{The following results draw upon the fact that \cref{eq:primal_problem} is equivalent to a linear program (LP) when $\bmM_1,\hdots,\bmM_k$ are jointly diagonalizable, i.e., the problem is a diagonal SDP. While we require the assumption that the equivalent LP in the jointly diagonalizable case has a unique optimal solution, we find this is a reasonable, mild assumption based on \cite[Theorem 4]{genericityResultsLP2017}, which proves the uniqueness property holds generically for LPs.}

\begin{lemma} \label{lem:gt}
Let $\bmc = (\bmM_1,\ldots,\bmM_k) \in \Cc$. If $\bmM_i$ are jointly diagonalizable for $i =
1,\ldots,k$ and \new{the associated LP for (\ref{eq:primal_problem}) has a unique optimal solution}, then there
exists an optimal solution of (\ref{eq:dual_problem}) with $\rank(\bmZ_i) \ge d-1$
for all $i = 1,\ldots,k$.
\end{lemma}
\revise{
The result follows directly from the Goldman-Tucker theorem on strict complementarity for LPs.
}
\begin{definition} \label{def:c_distance}
    For $\bmc = (\bmM_1,\hdots,\bmM_k) \in \Cc$ and $\cb = (\Mb_1,\hdots, \Mb_k) \in \Cc$, define $\mathrm{dist}(\bmc,\cb) \triangleq \max_{i \in [k]} \|\bmM_i - \Mb_i\|_{\revise{\mathrm{tr}}}.$
\end{definition}

We are now ready to state our main result in this subsection.

\begin{theorem} \label{thm:main_continuity}
\sloppypar{Let $\cb := (\Mb_1, \ldots, \Mb_k) \in \Cc$ be given such that
$\Mb_i$, $i=1,\ldots,k$, are jointly diagonalizable and \new{the associated LP, which is derived from the diagonal SDP of \cref{eq:primal_problem}} with objective coefficients $\cb$, has a unique optimal
solution. Then there exists a full-dimensional neighborhood $\overline \Cc \ni \cb$
in $\Cc$ such that (\ref{eq:primal_problem}) has the rank-one property for all $\bmc =
(\bmM_1,\ldots,\bmM_k) \in \overline \Cc$.}
\end{theorem}

\begin{proof}[Proof]
Using Lemma \ref{lem:gt}, let $\bmy^0 := (\Yb, \Zb_i, \nubi_i)$
be the optimal solution of the dual problem (\ref{eq:dual_problem})
for $\cb = (\Mb_1,\ldots,\Mb_k)$, which has $\rank(\Zb_i)
\ge d-1$ for all $i$. Proposition \ref{pro:ycont} in Appendix~\ref{sec:continuity} considers a function $y(\bmc; \bmy^0)$ that returns the optimal solution of (\ref{eq:dual_problem}) for $\bmc
= (\bmM_1,\ldots,\bmM_k)$ closest to $\bmy^0$, and shows that this function is continuous. It follows that its preimage
\[
    y^{-1}\left( \{ (\bmY, \bmZ_i, \nu_i) : \rank(\bmZ_i) \ge d-1 \ \ \forall \ i \} \right)
\]
contains $\cb$ and is an open set because the set
of all $(\bmY, \bmZ_i, \nu_i)$ with $\rank(\bmZ_i) \ge d-1$ is
an open set. After intersecting with $\Cc$, we have shown existence of this full-dimensional
set $\overline\Cc$.
From complementarity of the KKT conditions of the assignment LP, $\rank(\bmZ_i) = d-1$ for $i=1,\hdots,k$. Applying \cref{lem:Zrank_orthoXi} then completes the theorem.
\end{proof}

The next corollary shows that for a general tuple of matrices $\bmc := (\bmM_1,\hdots, \bmM_k)$ that are \edit{pairwise} CJD \new{for small enough $\delta$}, \cref{eq:primal_problem} is tight and has the rank-one property. In the following results, we will then prove the HPPCA generative model results in $(\bmM_1, \ldots, \bmM_k)$ being CJD. While these are sufficient conditions, they are by no means necessary, and \Cref{appendix:example_not_almost_commuting} in the supplement gives an example of $\bmM_i$ that are \textit{not} CJD but where the convex relaxation has the rank-one property. \new{It is important to note the results do not quantify an exact $\delta$ for \cref{eq:primal_problem} to achieve the ROP, but only the existence of one.}

\begin{corollary}
\label{cor:Mi_commute:general}
\revise{\sloppypar{Let $\epsilon > 0$, and $\bmc := (\bmM_1,\hdots, \bmM_k)$ be a tuple of self-adjoint matrices, where $\|[\bmM_i, \bmM_j]\|_{\mathrm{tr}} := \|\bmM_i \bmM_j - \bmM_j \bmM_i\|_{\mathrm{tr}} \leq \epsilon$ for all $i,j \in [k]$, and assume $\|\bmM_i\| \leq 1$ for all $i \in [k]$. \edit{Then there exists a tuple of commuting self-adjoint matrices $\cb := (\Mb_1, \hdots, \Mb_k)$ with $[\Mb_i, \Mb_j] = 0$ for all $i, j \in [k]$ and a $\delta(\epsilon,k) > 0$ such that $\mathrm{dist}(\bmc,\cb)  \leq \delta(\epsilon,k)$ and $\delta(\epsilon,k)$ is a function satisfying $\lim_{\epsilon \rightarrow 0} \delta(\epsilon,k) = 0$.} Assume the associated LP, which is derived from the diagonal SDP of \cref{eq:primal_problem} and is parameterized by $\bar{\mathbf{c}}$, has a unique optimal solution.}}
    
\revise{If $\epsilon > 0$ is such that $\text{dist}(\bmc, \cb) \leq \delta(\epsilon,k)$ implies $\bmc \in \overline \Cc$, \edit{where $\overline \Cc$ is given by Theorem \ref{thm:main_continuity},} (SDP-P) parameterized by $\mathbf{c}$ has the rank-one property.}
    
\end{corollary}
\revise{
\begin{proof}
    The result follows from directly applying the extension of Lin's Theorem for a tuple of $k \geq 3$ matrices \cite[Theorem 3]{Filonov2010AHA} (see \cref{lem:hilbert_schmidt_lins_thm} in the supplement) to $(\bmM_1,\hdots,\M_k)$.
\end{proof}}

\revise{The next corollary gives a similar result, but tailored specifically to problem \cref{eq:origobj}.}

\begin{corollary}
\label{cor:Mi_commute:shq}
    \revise{\sloppypar{Let $\epsilon > 0$, and define $\bmc := (\bmM_1,\hdots, \bmM_k)$ for \cref{eq:origobj}, where $\|[\bmA_i, \bmA_j]\|_{\mathrm{tr}} \leq \epsilon$ for all $i,j \in [L]$, and assume $\|\bmA_i\| \leq 1$ for all $i \in [k]$.
    \edit{Then there exists a tuple of commuting self-adjoint matrices $\cb := (\Mb_1, \hdots, \Mb_k)$ with $[\Mb_i, \Mb_j] = 0$ for all $i, j \in [k]$ and a $\delta(\epsilon,k) > 0$ such that $\mathrm{dist}(\bmc,\cb)  \leq \delta(\epsilon,k)\sum_{\ell=1}^L \max_{i \in [k]} w_{\ell,i} $.}}}
\end{corollary}

\subsubsection{HPPCA possesses the CJD property}

Consider the heteroscedastic probabilistic PCA problem in \cite{hong2021heppcat} where $L$ data groups of $n_1,\hdots,n_L$ samples ($n = \sum_{\ell=1}^L n_\ell$) with known noise variances $v_1,\hdots,v_L$ respectively are generated by the model
    \begin{align}
    \label{eq:hppca:generative_model}
        \bmy_{\ell,j} = \bmU \bmTheta \bmz_{\ell,j} + \bmeta_{\ell,j} \in \bbR^d \quad \forall \ell \in [L], j \in [n_\ell].
    \end{align}
    Here, $\bmU \in \text{St}(k,d)$ is a planted subspace, $\bmTheta = \text{diag}(\sqrt{\lambda}_1,\hdots,\sqrt{\lambda}_k)$ represent the known signal amplitudes, $\bmz_{\ell,j} \overset{\mathrm{iid}}{\sim} \mathcal{N}(\bm{0},\I_k)$ are latent variables, and $\bmeta_{\ell,j} \overset{\mathrm{iid}}{\sim} \mathcal{N}(\bm{0}, v_\ell \I_d)$ are additive Gaussian heteroscedastic noises. Assume that $\lambda_i \neq \lambda_j$ for $i \neq j \in [k]$ and $v_\ell \neq v_m$ for $\ell \neq m \in [L]$. The maximum likelihood problem in \cite[Equation 3]{hong2021heppcat} with respect to $\bmU$ is then equivalently \cref{eq:origobj} for $\A_\ell = \sum_{j=1}^{n_\ell} \frac{1}{v_\ell}\bmy_{\ell,j} \bmy_{\ell,j}'$ and $w_{\ell,i} = \frac{\lambda_i}{\lambda_i + v_\ell} \in (0,1]$. Our next result says that, as the number of samples $n$ grows, the signal-to-noise ratio $\lambda_i/v_\ell$ grows, or the variances are close to the median noise variance, the matrices in the HPPCA problem are almost commuting \edit{under the spectral norm}. The proof is found in \cref{sec:continuity}.
   \begin{proposition}
   \label{prop:hppca:almost_commute}
Let $\bmc = (\frac{1}{n}\bmM_1,\hdots,\frac{1}{n}\bmM_k)$ be the (normalized) data matrices of the HPPCA problem. Then there exists commuting $\cb = (\Mb_1,\hdots,\Mb_k)$ (constructed in the proof) \edit{and a universal constant $C > 0$} such that for any \new{$\bar{v} \geq 0$} \edit{and any $t>0$}, with probability exceeding $1 - e^{-t}$,
        \begin{align}
        \label{eq:prop4.9:bound}
             \frac{\|\textcolor{black}{\frac{1}{n}}\bmM_i - \Mb_i \|}{\|\Mb_1\|} \leq \min\left\{\sum_{\ell=1}^L \frac{\gamma_\ell(\bar{v})}{\frac{\lambda_i}{v_\ell} + 1}, C \frac{\bar {\sigma}_i}{\bar{\sigma}_1} \max\left\{\sqrt{\frac{\frac{\bar{\xi}_i}{\bar{\sigma}_i} \log d  + t}{n}}, \frac{\frac{\bar{\xi}_i}{\bar{\sigma}_i}\log d + t}{n}\log(n) \right \}\right\},
        \end{align}
        where
        \begin{align*}
            \gamma_\ell(\bar{v}) := \left|\frac{\bar{v}}{v_\ell} - 1\right|, \qquad \bar{\sigma}_i &:= \|\Mb_i \| = \sum_{\ell=1}^L \frac{\frac{\lambda_i}{v_\ell}}{\frac{\lambda_i}{v_\ell} + 1} \frac{n_\ell}{n} \left(\frac{\lambda_1}{v_\ell} + 1\right),\\
            \bar{\xi}_i &:= \tr(\Mb_i) = \sum_{\ell=1}^L \frac{\frac{\lambda_i}{v_\ell}}{\frac{\lambda_i}{v_\ell} + 1}\frac{n_\ell}{n} \left(\frac{1}{v_\ell} \sum_{j=1}^k \lambda_j + d\right). \nonumber
        \end{align*}
   \end{proposition}
   \begin{remark}
    It seems natural to let $\bar{v} = v_\mathrm{med} = \min_{v} \sum_{\ell=1}^L |v - v_\ell|$, i.e., the median noise variance, which provides an upper bound for \cref{eq:prop4.9:bound}, i.e., \visedit{$$\min_{\bar{v}} \sum_{\ell=1}^L \frac{{\gamma_\ell(\bar{v})}}{\frac{\lambda_i}{v_\ell} + 1} = \min_{\bar{v}} \sum_{\ell=1}^L \frac{{|\bar{v}- v_\ell|}}{\lambda_i + v_\ell} \leq \sum_{\ell=1}^L \frac{{|v_\mathrm{med} - v_\ell|}}{\lambda_i + v_\ell}.$$}
\end{remark}
\visedit{The proof of \cref{prop:hppca:almost_commute} in \cref{sec:continuity} analyzes two cases separately, obtaining bounds for the normalized distance under the spectral norm between each $\M_i$ and $\Mb_i$. The final result in \cref{eq:prop4.9:bound} then takes the minimum of the two bounds.}
%
\visedit{The left argument of the minimum operator in \cref{eq:prop4.9:bound} reflects the effect of the heterogeneous noise. As all of the variances become close in value to some $\bar{v} \geq 0$, the matrices $\M_i$ become almost commuting, eventually becoming equal when all the variances are equal, i.e., the noise is homogeneous. In addition, the distance depends on the inverse signal-to-noise ratios between the eigenvalues $\lambda_i$ and variances $v_\ell$, so as the noise diminishes, the matrices $\M_i$ also become almost commuting.}

\visedit{The right argument of the minimum operator captures the effects of growing dimension, rank, and sample size using the concentration of the sample covariance matrices for Gaussian random variables \cite{Lounici2014}.
First, the normalized distance between each $\M_i$ and $\Mb_i$ grows as $\mathcal{O}(d \log d)$ and linearly with $\sum_{i=1}^k \lambda_i$ (which is related to the rank), as reflected by the terms $\bar{\xi}_i$.
Lastly, the bound diminishes as $\mathcal{O}(1/\sqrt{n})$, where $n$ is the total number of data samples; as the sample size grows, the matrices become almost commuting. 
}

\section{Numerical experiments}
\label{s:experiments}

All numerical experiments were computed using MATLAB R2018a on a MacBook Pro with a 2.6 GHz 6-Core Intel Core i7 processor. When solving SDPs, we use the SDPT3 solver of the CVX package in MATLAB \cite{cvx}. All code necessary to reproduce our experiments is available at \texttt{https://github.com/kgilman/Sums-of-Heterogeneous-Quadratics}. When executing each algorithm in practice, we remark that the results may vary with the choice of user specified numerical tolerances and other settings. \revise{Since \cref{thm:dual_certificate} requires an exact stationary point, and in practice, an iterative solver only returns an inexact stationary point, the KKT conditions may not be exactly satisfied. However, in practice, we found using smaller numerical precisions in the SDP and iterative solvers is often sufficient to achieve a \visedit{numerical} certificate, albeit inexact. When computing a first-order stationary point with an iterative solver, we terminate the algorithm when the norm of the gradient on the manifold is less than $10^{-10}$.} We point the reader to \cref{appendix:experiments} for further detailed settings.

\subsection{Assessing the rank-one property (ROP)}

In this subsection, we conduct experiments showing that, for many random instances of the HPPCA application, the SDP relaxation \cref{eq:primal_problem} is tight with optimal rank-one $\bmX_i$, yielding a globally optimal solution of \cref{eq:sum_heterogeneous_quadratics}. Similar experiments for other forms of \cref{eq:sum_heterogeneous_quadratics}, including where $\bmM_i$ are random low-rank PSD matrices, are found in \cref{appendix:experiments} and give similar insights. Here, the $\bmM_i$ were generated according to our HPPCA model in \eqref{eq:hppca:generative_model} where $\A_\ell$ = $\frac{1}{v_\ell} \sum_{i=1}^{n_\ell} \bmy_{\ell,i} \bmy_{\ell,i}'$ and weight matrices $\bmW_\ell$ are calculated as $\bmW_\ell = \text{diag}(w_{\ell,1},\hdots,w_{\ell,k})$ for $w_{\ell,i} = \lambda_i / (\lambda_i + v_\ell)$. We made $\bmlambda$ a $k$-length vector of entries uniformly spaced in the interval $[1,4]$, and we varied the ambient dimension $d$, rank $k$, samples \edit{$\bmn:=[n_1,\hdots,n_L]$}, and variances $\bmv$ for both $L=2$ and $L=3$ noise groups. Each random instance was generated from a new random draw of $\bmU$ on the Stiefel manifold, latent variables $\bmz_{\ell,i}$, and noise vectors $\bmeta_{\ell,i}$.

\revise{Tables \ref{tbl:trial_counts_hppca:a}, \ref{tbl:trial_counts:hppca:b}, and \ref{tbl:trial_counts:hppca:leq3} show the results of these experiments for various choices of dimension $d$, rank $k$, samples $\bmn$, and variances $\bmv$.} We solved the SDP for 100 random data instances for $d\leq 50$ and 20 random data instances for $d\geq 100$. The table shows the fraction of trials that resulted in rank-one $\bmX_i$ for all $i=1,\hdots,k$. We computed the average error of the sorted eigenvalues of each optimal solution $\Xb_i$ to $\bme_1$, i.e. $\frac{1}{k}\sum_{i=1}^k \|\diag(\bmSigma_i) - \bme_1\|_2^2$ where $\Xb_i = \bmV_i \bmSigma_i \bmV_i'$, and counted any trial with error greater than $10^{-5}$ as not tight. 

The SDP solutions possessed the ROP in the vast majority of trials. 
As we increased the total number of samples $n$ in Tables \ref{tbl:trial_counts_hppca:a} and \ref{tbl:trial_counts:hppca:leq3}, the convex relaxation became tight in 100\% of the trials, as predicted by the commuting error bound dependency on $\mathcal{O}(1/\sqrt{n})$ in \cref{prop:hppca:almost_commute}.
As $d$ or $k$ increased, we generally observed a few instances where the SDP was not tight, which conforms with the theory in \cref{prop:hppca:almost_commute}. 
As we decreased the spread of the variances, \cref{tbl:trial_counts:hppca:b} shows the fraction of tight instances increased, reaching 100\% in the homoscedastic setting, as expected \edit{because then all of the $\M_i$ are equal.} 
\revise{Likewise, \cref{tbl:trial_counts:hppca:leq3} shows this behavior for the $L=3$ case.}

\begin{table}[htbp]
\parbox{.48\linewidth}{
\centering
\resizebox{0.48\textwidth}{!}{
\begin{tabular}{|c|c|c|c|c|c|}
    \hline
    \multicolumn{2}{|c|}{} &\multicolumn{4}{|c|}{\textbf{Fraction of 100 trials with ROP}} \\
     \cline{3-6}
     \multicolumn{2}{|c|}{$\bmv = [1,4]$}  & $k=3$ & $k=5$ & $k=7$ & $k=10$ \\\hline
  \multirow{5}{*}{\rotatebox{90}{\textbf{$\bmn = [5,20]$}}}  
 
    &$d=10$&   1  &  0.99   &    1   &    1\\\cline{2-6} 
    &$d=20$&   1  &  0.98  &  0.98 &   0.99\\\cline{2-6} 
    &$d=30$&    0.99  &  0.93  &  0.98  &  0.97\\\cline{2-6} 
    &$d=40$&    0.98  &  0.91  &  0.99 &   0.98\\\cline{2-6} 
    &$d=50$&    0.97 &   0.95 &   0.96  &  0.98\\\cline{2-6} 
    \hline \hline
   
   \multirow{5}{*}{\rotatebox{90}{\textbf{$\bmn = [20,80]$}}}   &$d=10$&    1    &    1    &   1    &   1\\\cline{2-6} 
    &$d=20$&    1    &    1    &   1   &    1\\\cline{2-6} 
    &$d=30$&    1   &     1    &   1  &  0.98\\\cline{2-6} 
    &$d=40$&    1    &    1  &  0.97  &  0.95\\\cline{2-6} 
    &$d=50$&    1  &   0.98  &  0.98  &  0.97\\\cline{2-6} 
   \hline \hline
   
   \multirow{5}{*}{\rotatebox{90}{\scriptsize{\textbf{$ \bmn =[100,400]$}}}}  &$d=10$&    1  &   1  &   1   &   1 \\\cline{2-6}
    &$d=20$&    1  &   1  &   1   &   1 \\\cline{2-6}
    &$d=30$&    1  &   1  &   1   &   1 \\\cline{2-6}
    &$d=40$&    1  &   1  &   1   &   1 \\\cline{2-6}
    &$d=50$&    1  &   1  &   1   &   1 \\\cline{2-6}
   \hline
   
   
  \end{tabular}
  }

\caption{ Numerical experiments showing the fraction of trials where the SDP was tight for instances of the HPPCA problem as we varied $d$, $k$, and $\bmn$ using $L=2$ groups with noise variances $\bmv = [1,4]$.}
  \label{tbl:trial_counts_hppca:a}
}
\hspace{3mm}
\parbox{.48\linewidth}{
\centering
\resizebox{0.48\textwidth}{!}{
\begin{tabular}{|c|c|c|c|c|c|}
    \hline
    \multicolumn{2}{|c|}{} &\multicolumn{4}{|c|}{\textbf{Fraction of 100 trials with ROP}} \\
     \cline{3-6}
     \multicolumn{2}{|c|}{$\bmn = [10,40]$}  & $k=3$ & $k=5$ & $k=7$ & $k=10$ \\\hline
  \multirow{5}{*}{\rotatebox{90}{\textbf{$\bmv = [1,1]$}}}  
 
   &$d=10$&    1  &   1   &     1   &  1 \\\cline{2-6} 
    &$d=20$&    1  &   1   &     1   &  1 \\\cline{2-6} 
    &$d=30$&    1  &   1   &     1   &  1 \\\cline{2-6} 
    &$d=40$&    1  &   1   &  1   &  1 \\\cline{2-6} 
    &$d=50$&    1  &   1   &  1   &  1 \\\cline{2-6} 
   \hline  \hline
   
   \multirow{5}{*}{\rotatebox{90}{\textbf{$\bmv = [1,2]$}}}
   &$d=10$&    1  &   1   &     1   &  1 \\\cline{2-6} 
    &$d=20$&    1  &   1   &     1   &  1 \\\cline{2-6} 
    &$d=30$&    1  &   0.98   &     1   &  1 \\\cline{2-6} 
    &$d=40$&    1  &   1   &  0.99   &  1 \\\cline{2-6} 
    &$d=50$&    1  &   1   &  1   &  0.99 \\\cline{2-6} 
   \hline  \hline
   
   \multirow{5}{*}{\rotatebox{90}{\textbf{$\bmv = [1,3]$}}}   &$d=10$&    1  &   1   &     1   &  1 \\\cline{2-6} 
    &$d=20$&    1  &   1   &     1   &  1 \\\cline{2-6} 
    &$d=30$&    0.99  &   0.99   &    0.97   &  0.99 \\\cline{2-6} 
    &$d=40$&    1  &   0.98   &  0.97   &  0.99 \\\cline{2-6} 
    &$d=50$&    1  &   0.97   &  0.96   &  0.98 \\\cline{2-6} 
   \hline
   
   
  \end{tabular}
  }
    \\
  \vspace{2mm}
\resizebox{0.48\textwidth}{!}{
\revise{
\begin{tabular}{|c|c|c|c|}
    \hline
    \multicolumn{2}{|c|}{} &\multicolumn{2}{|c|}{\textbf{Fraction of 20 trials with ROP}} \\
     \cline{3-4}
     \multicolumn{2}{|c|}{$\bmv = [1,3]$}  & $k=5$ & $k=10$ \\\hline
  \multirow{3}{*}{\rotatebox{0}{\textbf{$\bmn = [10,40]$}}}  
 
    &$d=100$&   1  &  0.85  \\\cline{2-4} 
    &$d=200$&   0.95  &  0.35\\\cline{2-4}
    &$d=300$& 0.75 & 0.35 \\\cline{2-4}
    \hline \hline
   
   \multirow{3}{*}{\rotatebox{0}{\textbf{$\bmn = [50,200]$}}} 
    &$d=100$&   1  &   0.95\\\cline{2-4} 
    &$d=200$&   1   &   0.8\\\cline{2-4}  
    &$d=300$ & 1 & 0.85 \\\cline{2-4}
   \hline
  \end{tabular}
  }
}
\caption{ Numerical experiments showing the fractions of trials where the SDP was tight for instances of the HPPCA problem as we varied $d$, $k$, and $\bmv$ using $L=2$ groups with samples $\bmn = [10,40]$ and $\bmn = [50,200]$. Due to the large computation time of solving the full SDP for larger values of $d \geq 100$, we only ran 20 independent trials for each experiment setting.}
  \label{tbl:trial_counts:hppca:b}
}
\end{table}

\begin{table}[htpb]
\centering
\parbox{.48\linewidth}{
\centering
\revise{
\resizebox{0.48\textwidth}{!}{
\begin{tabular}{|c|c|c|c|c|c|}
    \hline
    \multicolumn{2}{|c|}{} &\multicolumn{4}{|c|}{\textbf{Fraction of 100 trials with ROP}} \\
     \cline{3-6}
\multicolumn{2}{|c|}{$d=20$, $v_1 = 1$, $v_3 = 4$}  & $v_2=1$ & $v_2=2$ & $v_2=3$ & $v_2=4$ \\\hline
\multirow{7}{*}{\rotatebox{90}{\textbf{$k = 5$}}} &
\multirow{1}{*}{\rotatebox{0}{\textbf{$\bmn = [20,20,60]$}}}  
    & 1  &   1    &    1   &    1\\ \cline{2-6} &
\multirow{1}{*}{\rotatebox{0}{\textbf{$\bmn = [20,80,60]$}}}
    & 1  &   1   &     1  &  0.99 \\\cline{2-6} &
\multirow{1}{*}{\rotatebox{0}{\textbf{$\bmn = [20,80,200]$}}}
    &  1  &   1   &     1    &   1\\\cline{2-6} &
\multirow{1}{*}{\rotatebox{0}{\textbf{$\bmn = [20,20,400]$}}}
    &  1  &   1   &  1   &     1\\\cline{2-6} &
\multirow{1}{*}{\rotatebox{0}{\textbf{$\bmn = [20,80,400]$}}}
    &   1   &  1  &   1   &  1\\\cline{2-6} &
\multirow{1}{*}{\rotatebox{0}{\textbf{$\bmn = [100,100,400]$}}}
    &   1   &  1  &   1   &  1\\\cline{2-6} &
\multirow{1}{*}{\rotatebox{0}{\textbf{$\bmn = [200,200,400]$}}}
    &   1   &  1  &   1   &  1\\\cline{2-6}
    \hline \hline
\multirow{5}{*}{\rotatebox{90}{\textbf{$k = 10$}}} &
\multirow{1}{*}{\rotatebox{0}{\textbf{$\bmn = [20,20,60]$}}}  
    &  0.99  &     1 &   0.99 &   0.97\\ \cline{2-6} &
\multirow{1}{*}{\rotatebox{0}{\textbf{$\bmn = [20,80,60]$}}}
    &    1   &  1  &   0.99  &  0.99\\\cline{2-6} &
\multirow{1}{*}{\rotatebox{0}{\textbf{$\bmn = [20,80,200]$}}}
    &   1    &    1   &    1   &    1\\\cline{2-6}&
\multirow{1}{*}{\rotatebox{0}{\textbf{$\bmn = [20,20,400]$}}}
    & 1   &    1  &  1   &     1\\\cline{2-6}& 
\multirow{1}{*}{\rotatebox{0}{\textbf{$\bmn = [20,80,400]$}}}
    &   1   &  1   &  1    &    1\\\cline{2-6} &
\multirow{1}{*}{\rotatebox{0}{\textbf{$\bmn = [100,100,400]$}}}
    &   1   &  1  &   1   &  1\\\cline{2-6} &
\multirow{1}{*}{\rotatebox{0}{\textbf{$\bmn = [200,200,400]$}}}
    &   1   &  1  &   1   &  1\\\cline{2-6} 
    \hline
  \end{tabular}
  }
  }
}
\parbox{.48\linewidth}{
\centering
\revise{
\resizebox{0.48\textwidth}{!}{
\begin{tabular}{|c|c|c|c|c|c|c|}
    \hline
    \multicolumn{2}{|c|}{} &\multicolumn{4}{|c|}{\textbf{Fraction of 100 trials with ROP}} \\
     \cline{3-6}
\multicolumn{2}{|c|}{$d=50$, $v_1 = 1$, $v_3 = 4$}  & $v_2=1$ & $v_2=2$ & $v_2=3$ & $v_2=4$ \\\hline
\multirow{7}{*}{\rotatebox{90}{\textbf{$k = 5$}}} &
\multirow{1}{*}{\rotatebox{0}{\textbf{$\bmn = [20,20,60]$}}}  
    &  1  &  1   &   0.98  &  0.96\\ \cline{2-6} &
\multirow{1}{*}{\rotatebox{0}{\textbf{$\bmn = [20,80,60]$}}} 
    &   1  &   1   &  0.99  &   0.96\\\cline{2-6} & 
\multirow{1}{*}{\rotatebox{0}{\textbf{$\bmn = [20,80,200]$}}} 
    &   1  &   1   &  0.98   & 0.99\\\cline{2-6} & 
\multirow{1}{*}{\rotatebox{0}{\textbf{$\bmn = [20,20,400]$}}}
    &    1   &  1  &   1   &  0.99\\\cline{2-6} &
\multirow{1}{*}{\rotatebox{0}{\textbf{$\bmn = [20,80,400]$}}}
    &   1   &  1   &  1   &  1\\\cline{2-6} &
\multirow{1}{*}{\rotatebox{0}{\textbf{$\bmn = [100,100,400]$}}}
    &   1   &  1   &  0.99   &  1\\\cline{2-6} &
\multirow{1}{*}{\rotatebox{0}{\textbf{$\bmn = [200,200,400]$}}}
    &   1   &  1   &  1   &  1\\\cline{2-6} 
    \hline \hline
\multirow{5}{*}{\rotatebox{90}{\textbf{$k = 10$}}} &
\multirow{1}{*}{\rotatebox{0}{\textbf{$\bmn = [20,20,60]$}}}  
    &  1  &  0.97  &  0.96  &  0.92\\ \cline{2-6} &
\multirow{1}{*}{\rotatebox{0}{\textbf{$\bmn = [20,80,60]$}}}
    &   1  &   1  &   0.98 &   0.94\\\cline{2-6} & 
\multirow{1}{*}{\rotatebox{0}{\textbf{$\bmn = [20,80,200]$}}}
    &  1  &   0.98  &  0.99  &  0.99\\\cline{2-6} & 
\multirow{1}{*}{\rotatebox{0}{\textbf{$\bmn = [20,20,400]$}}}
    & 0.99 &   0.99  &  1   &  0.99\\\cline{2-6} &
\multirow{1}{*}{\rotatebox{0}{\textbf{$\bmn = [20,80,400]$}}}
    &    1  &   1  &   1   &  0.99\\\cline{2-6} &
\multirow{1}{*}{\rotatebox{0}{\textbf{$\bmn = [100,100,400]$}}}
    &   1   &  1   &  1  &  1\\\cline{2-6} &
\multirow{1}{*}{\rotatebox{0}{\textbf{$\bmn = [200,200,400]$}}}
    &   1   &  1   &  1   &  1\\\cline{2-6} 
    \hline
  \end{tabular}
  }
  }
}

\caption{\revise{Numerical experiments showing the fractions of trials where the SDP was tight for instances of the HPPCA problem as we varied  $d$, $k$, $\bmn$, and $v_2$ for $L=3$ groups with noise variances $\bmv = [1,v_2,4]$. The left and right tables show the results for $d=20$ and $d=50$, respectively. We swept $v_2$ in a way such that $\bmn = [20,20,60], \bmv = [1,4,4$] is statistically equivalent to the problem in \cref{tbl:trial_counts_hppca:a} for $\bmn = [20, 80], \bmv = [1,4]$ and,  similarly, $\bmn = [20, 80, 400], \bmv = [1,1,4]$ is statistically equivalent to the problem in \cref{tbl:trial_counts_hppca:a} for $\bmn = [100,400], \bmv = [1,4]$.}}
  \label{tbl:trial_counts:hppca:leq3}
\end{table}

\subsection{Assessing global optimality of local solutions}

In this section, we used the Stiefel majorization-minimization (StMM) solver with a linear majorizer from \cite{breloyMMStiefel2021} to obtain a local solution $\Ub_{\text{MM}}$ to \cref{eq:sum_heterogeneous_quadratics} for various inputs $\bmM_i$ \new{and used \cref{thm:dual_certificate} to certify if the local solution is globally optimal or if the certificate fails.} For comparison, we obtained candidate solutions $\Xb_i$ from the SDP and performed a rank-one SVD of each to form $\Ub_{\text{SDP}}$, i.e. 

{$$\Ub_{\text{SDP}} = \mathcal{P}_{\text{St}}([\ub_1 \cdots \ub_k]), \quad \ub_i = \argmax_{\bmu : \|\bmu\|_2 = 1} \bmu' \Xb_i \bmu, $$}

\noindent while measuring how close the solutions are to being rank-one. In the case the SDP is not tight, the rank-one directions from the $\Xb_i$ will not be orthonormal, so as a heuristic, we projected $\Ub_{\text{SDP}}$ onto the Stiefel manifold by the orthogonal Procrustes solution, denoted by the operator $\mathcal{P}_{\text{St}}(\cdot)$ \cite{breloyMMStiefel2021}.

\subsubsection{Synthetic CJD matrices}
\label{s:exp:ss:CJD}

To empirically verify our theory from Section \ref{s:theory}, we generated each $\bmM_i \in \bbR^{d \times d}$ to be a diagonally dominant matrix resembling an approximately rank-$k$ sample covariance matrix,
such that, in a similar manner to HPPCA, $\bmM_1 \succeq \bmM_2 \succeq \cdots \succeq \bmM_k \succeq 0$. Specifically, we first constructed $\bmM_k = \D_k + \N_k$, where $\D_k$ is a diagonal matrix with $k$ nonzero entries drawn uniformly at random from $[0,1]$, and $\N_k = \frac{1}{10d} \Ss \Ss'$ for $\Ss \in \bbR^{d \times 10d}$ whose entries are drawn i.i.d.~as $\clN(0, \sigma \I)$ for varying $\sigma$. 
We then generated the remaining $\bmM_i$ for $i=k-1,\hdots,1$ as $\bmM_i = \bmM_{i+1} + \D_i + \N_i$ with new random draws of $\D_i$ and $\N_i$ 
and normalized all by $1/\max_{i \in [k]} \|\bmM_i\|$ so that $\|\bmM_i \| \leq 1$ for all $i \in [k]$. 
With this setup, by varying $\sigma$, we swept through a range of commuting distances under the spectral norm, i.e. $\max_{i,j \in [k]} \| \bmM_i \bmM_j - \bmM_j \bmM_i \|$. 
For all experiments, we generated problems with parameters $d=10$, $k=3$, and ran StMM for 2,000 maximum iterations or until the norm of the gradient on the manifold was less than $10^{-10}$.

\cref{fig:sim:orderedASD:objgap} shows the gap of the objective values between the SDP relaxation (before projection onto the Stiefel) and the nonconvex problem ($p_{\text{SDP}} - p_{\text{StMM}}$) versus the commuting distance. \cref{fig:sim:orderedASD:Udist} shows the distance between the two obtained solutions computed as $\frac{1}{\sqrt{k}}\||\Ub_{\text{StMM}}'\Ub_{\text{SDP}}| - \bmI_k\|_F$ (where $|\cdot|$ denotes taking the elementwise absolute value) versus commuting distance. \cref{fig:sim:orderedASD:cert} shows the percentage of trials where $\Ub_{\text{StMM}}$ could not be certified globally optimal. Like before,  we declared an SDP's solution ``tight" if \new{the mean error of its solutions to a rank-one matrix with binary eigenvalues, i.e., $ \frac{1}{k} \sum_{i=1}^k \| \lambda^{(i)}_\downarrow - \bme_1 \|_2,$ was less than $10^{-5}$, where $\lambda^{(i)}_\downarrow$ denotes the sorted eigenvalues of $\bmX_i$ in descending order, and $\bme_1$ is the first standard basis vector in $\bbR^d$.} Trials with the marker ``$\circ$" indicate trials where global optimality was certified. The marker ``\textcolor{red}{$\times$}" represents trials where $\Ub$ was not certified as globally optimal and the SDP relaxation was not tight; ``\textcolor{blue}{$\triangle$}" markers indicate trials where the SDP was tight, but \cref{eq:dual_certificate} was not satisfied, implying a suboptimal local maximum.

Towards the left of \cref{fig:sim:orderedASD:objgap}, with small $\sigma$ and the $(\bmM_1,\hdots,\bmM_k)$ all being very close to commuting, 100\% of experiments returned tight rank-one SDP solutions. Notably, there appears to be a sharp cut-off point where this behavior ends, and the SDP relaxation was not tight in a small percentage of cases.
While the large majority of trials still admitted a tight convex relaxation, these results empirically corroborate the sufficient conditions derived in \cref{thm:main_continuity} and \cref{cor:Mi_commute:general}.

Where the SDP is tight, \cref{fig:sim:orderedASD} shows the StMM solver returned the globally optimal solution in more than $95\%$ of the problem instances. Indicated by the ``\textcolor{blue}{$\triangle$}" markers, the remaining cases can only be certified as stationary points, implying a local maximum was found. Indeed, we observed a correspondence between trials with both large objective value gap and distance of the candidate solution to the globally optimal solution returned by the SDP.
\begin{figure}
    \centering
    \subfloat[]{\includegraphics[height=1.75in]{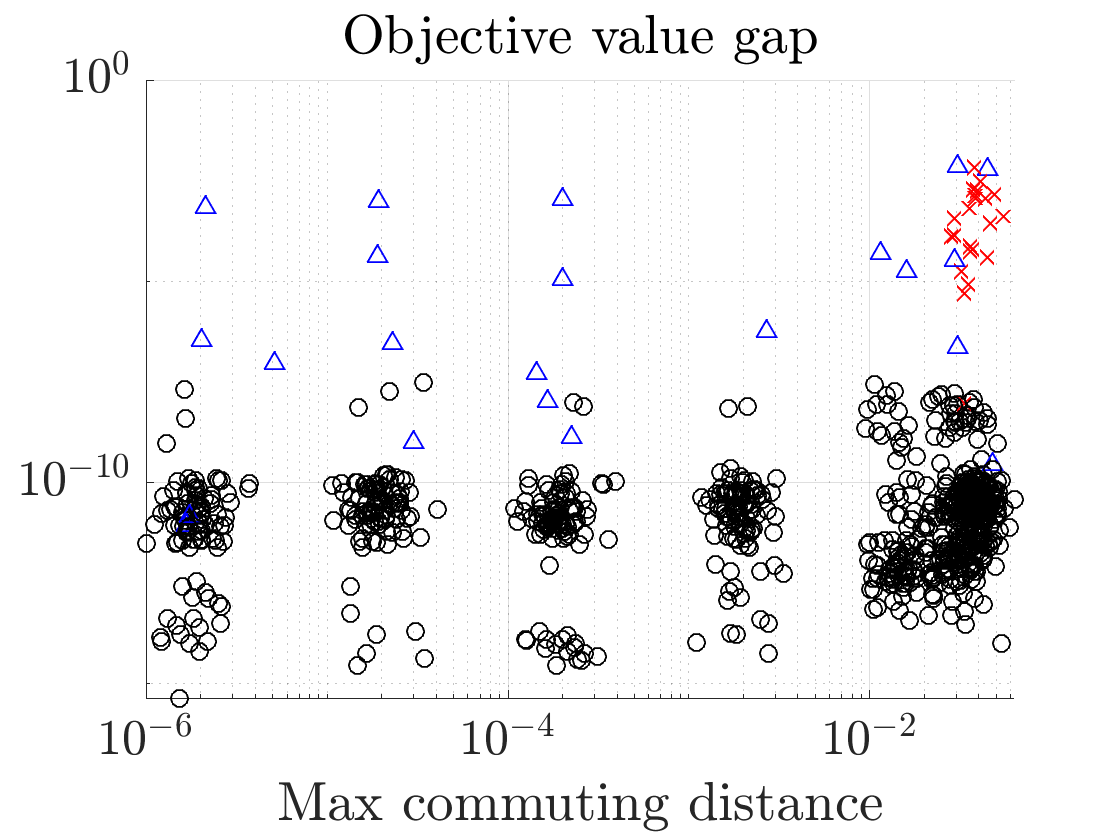} \label{fig:sim:orderedASD:objgap}}
    \subfloat[]{\includegraphics[height=1.75in]{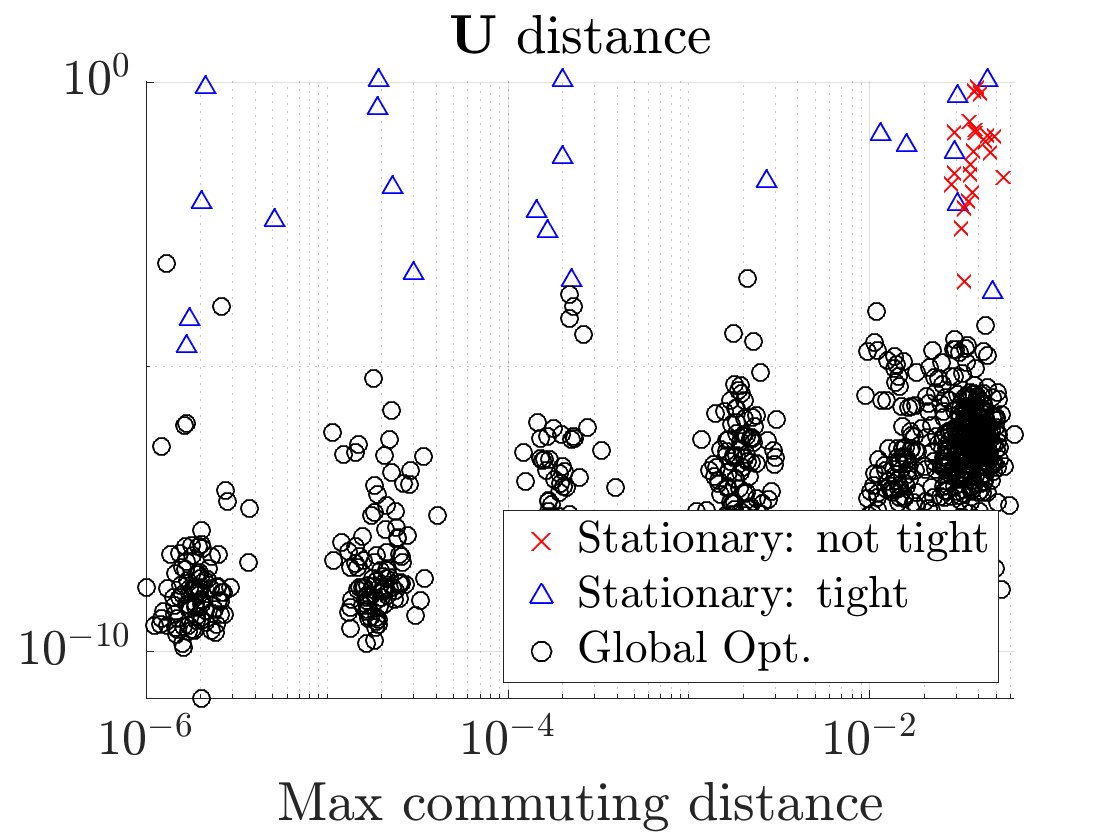} \label{fig:sim:orderedASD:Udist}}
    \caption{Numerical simulations for synthetic CJD matrices for $d=10, k=3$ with increasing $\sigma$ and 100 random problem instances for each setting. \visedit{As $\sigma$ grows, the max commuting distance grows.}}
    \label{fig:sim:orderedASD}
\end{figure}


\subsubsection{HPPCA}
We repeated the experiments just described for $\bmM_i$ generated by the model in \cref{eq:hppca:generative_model} for $d = 50$, $\bmlambda = [4, 3.25, 2.5, 1.75, 1]$, and $L=2$ noise groups with variances $\bmv = [1,4]$. \new{For each of 100 trials, we drew a random model with a different generative $\bmU$ for sample sizes $\bmn = [n_1, 4n_1]$, where we swept through increasing values of $n_1$ on the horizontal axis in \cref{fig:sim:hppca:cert}.} For each experiment, we normalized the $\bmM_i$ by the maximum of their spectral norms, and then recorded the results obtained from the SDP and StMM solvers with respect to the computed maximum commuting distance of the $\bmM_i$ in \cref{fig:sim:hppca}. We ran StMM for a maximum of 10,000 iterations, and recorded both the global optimality certification of each StMM run and if the SDP was tight.

 \cref{prop:hppca:almost_commute} suggests that, even with poor SNR like in this example, as the number of data samples increases, the $\bmM_i$ should concentrate to be nearly commuting. This was indeed what we observed: as the number of samples increased \new{in \cref{fig:sim:hppca:cert}}, the maximum commuting distance of the $\M_i$ decreased, \new{i.e., the simulations moved to the left on the horizontal axes of \cref{fig:sim:hppca:objgap,fig:sim:hppca:Udist}.} In this nearly-commuting regime, the SDP obtained tight rank-one $\bmX_i$ in 100\% of the trials, and all of the StMM runs attained the global maximum, suggesting a seemingly benign nonconvex landscape. In contrast, we observed several trials in the low-sample setting where the SDP failed to be tight and a dual certificate was not attained. Also within this regime, several trials of the StMM solver found suboptimal local maxima.

\begin{figure}
    \centering
    \subfloat[]{\includegraphics[height=1.75in]{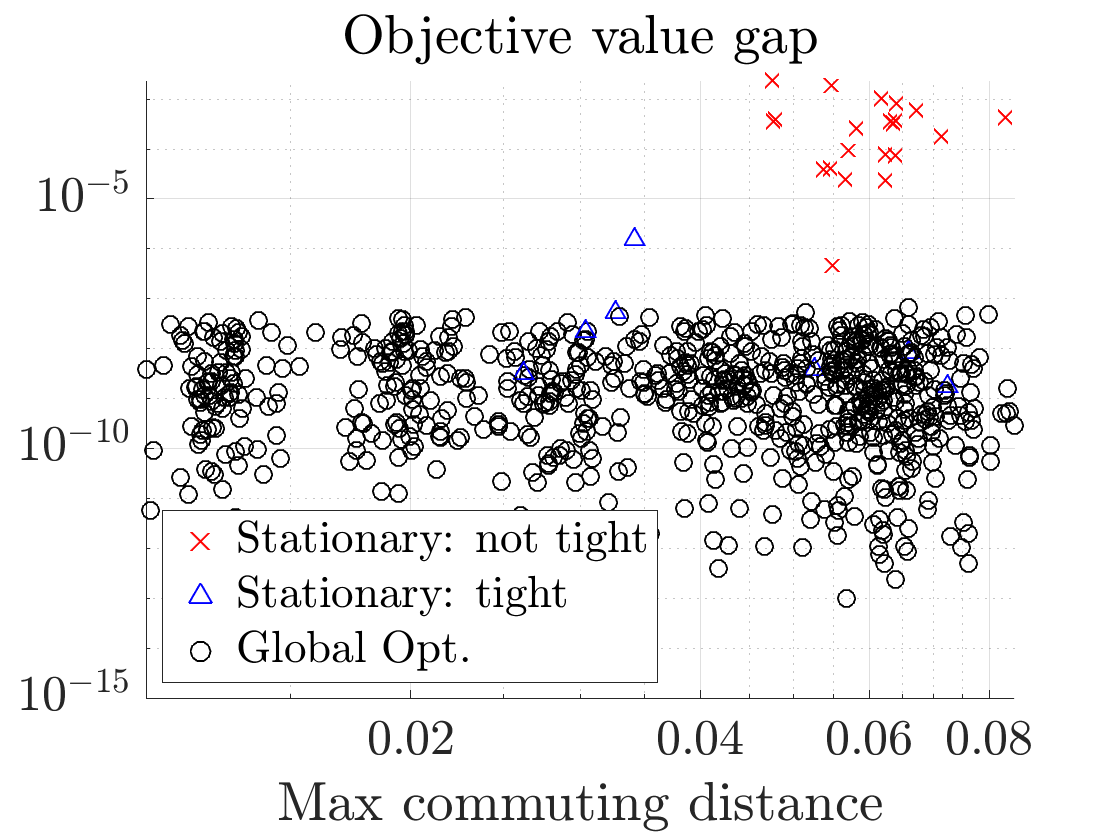} \label{fig:sim:hppca:objgap}}
    \subfloat[]{\includegraphics[height=1.75in]{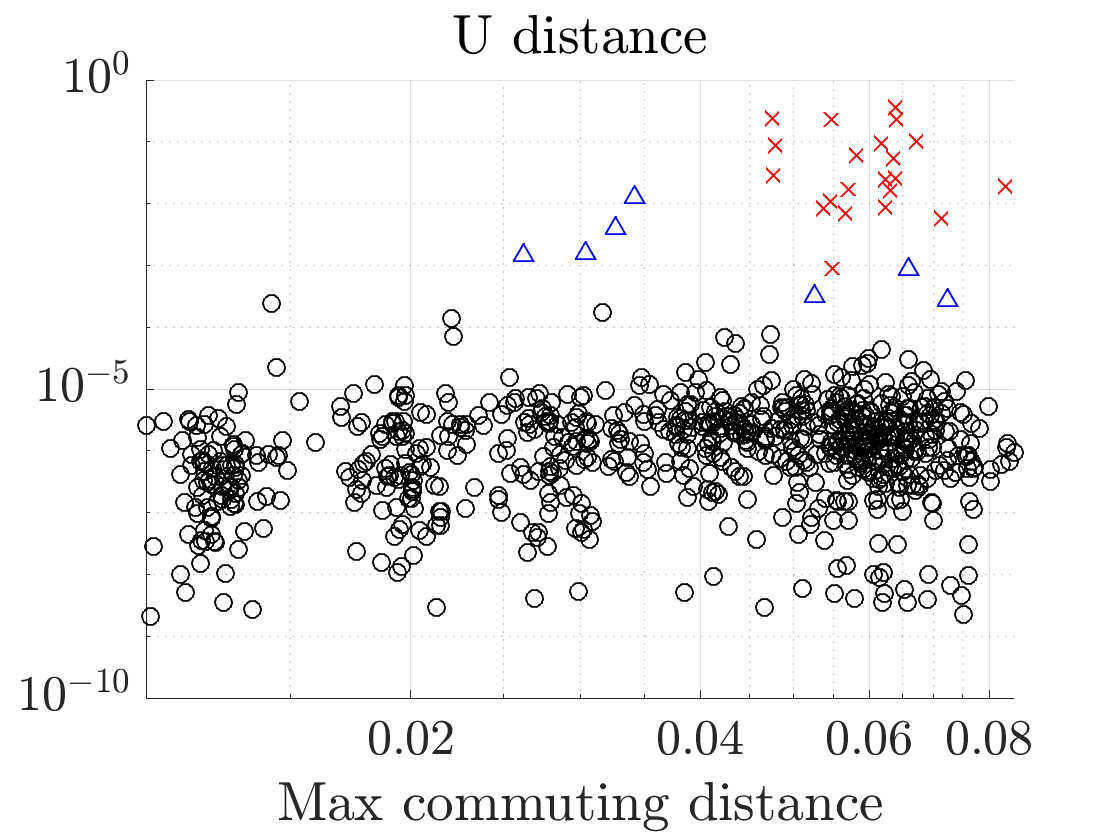} \label{fig:sim:hppca:Udist}}
    \caption{Numerical simulations for $\bmM_i$ generated by the HPPCA model in \cref{eq:hppca:generative_model} for $d= 50$, $k = 5$, noise variances $\bmv = [1,4]$, and $\bmlambda = [4, 3.25, 2.5, 1.75, 1]$ with increasing samples $n$. \visedit{As $n$ grows, the max commuting distance gets smaller.}}
    \label{fig:sim:hppca}
\end{figure}

\begin{figure}
    \centering
    \subfloat[Synthetic CJD matrices for $d=10$, $k=3$ with increasing $\sigma$.]{\includegraphics[height=1.95in]{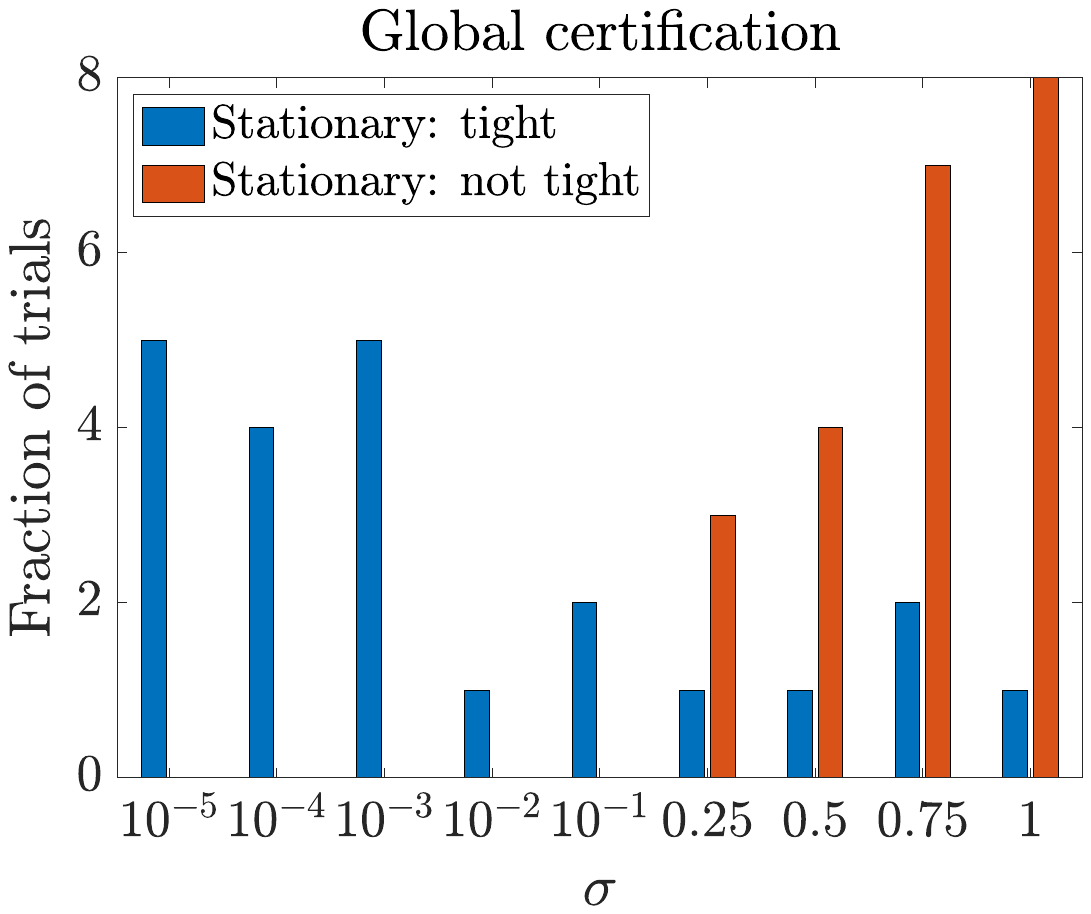} \label{fig:sim:orderedASD:cert}}
    \hspace{4mm}
    \subfloat[{Data matrices generated by the HPPCA model in \cref{eq:hppca:generative_model} with increasing samples $\bmn = [n_1, 4n_1]$.}]{\includegraphics[height=2in]{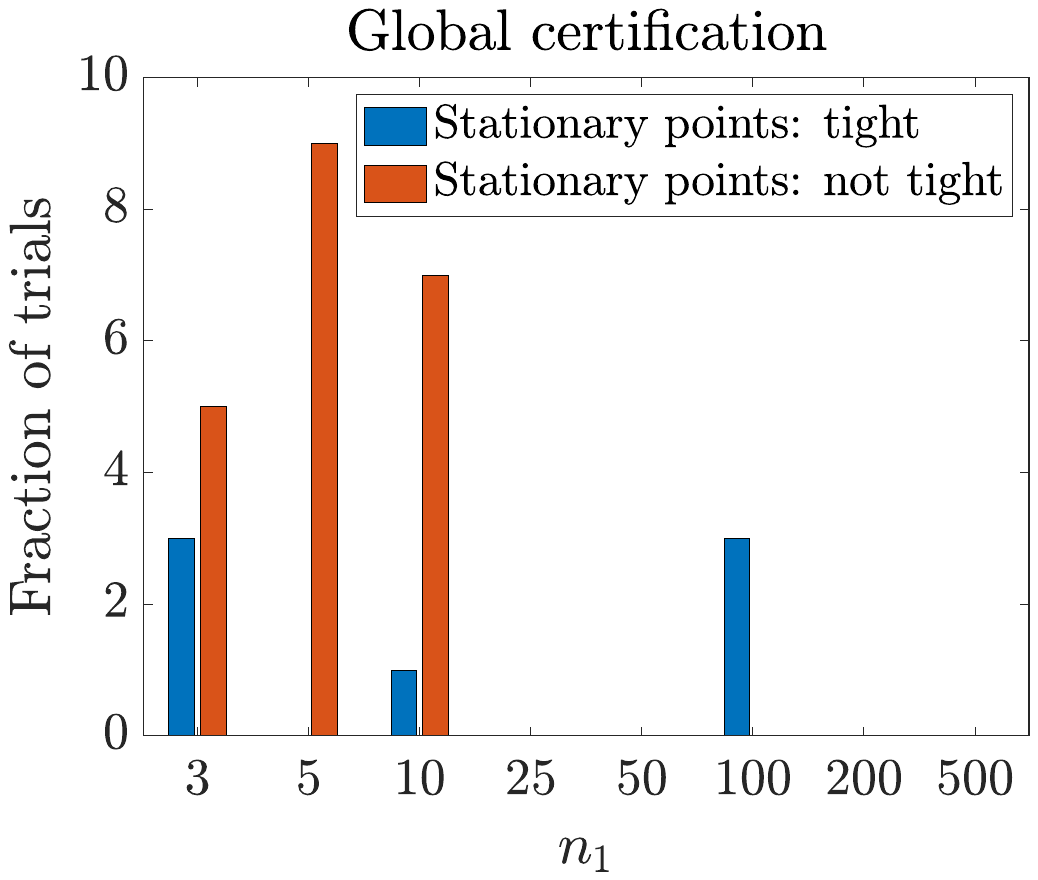} \label{fig:sim:hppca:cert}}
    \caption{Percentages of global certification of StMM solutions out of 100 trials. The fractions not shown are tight instances certified as global. }
    \label{fig:sim:global_cert_plots}
\end{figure}

\subsection{Computation time}
\begin{figure}
   \begin{center}
     \includegraphics[height=2.5in]{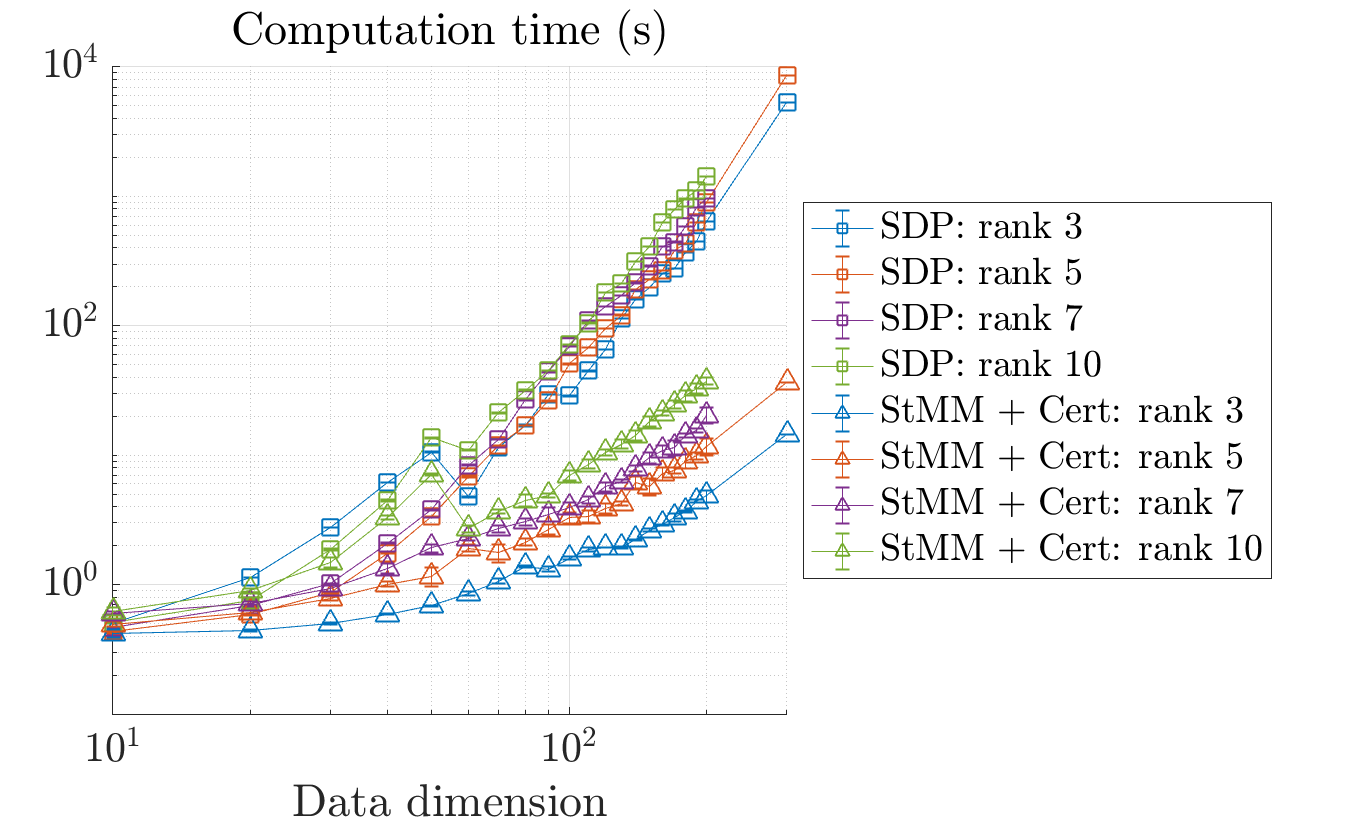}
   \end{center}
   \caption{Computation time of \cref{eq:primal_problem} versus StMM for 2000 iterations with global certificate check \cref{eq:dual_certificate} for HPPCA problems as the data dimension varies. We used $\bmv = [1,4]$, and $\bmn=[100,400]$ and made $\bmlambda$ a $k$-length vector with entries equally spaced in the interval $[1,4]$, \edit{where the rank of the model is $k$.} Markers indicate the median computation time taken over 10 trials, and error bars show the standard deviation. \revise{Due to memory and computation limitations for $d=300$, we only performed one timing test for $k=3$ and $k=5$.} }
  \label{fig:sim:comp_time_sdp_vs_certificate}
 \end{figure}

\cref{fig:sim:comp_time_sdp_vs_certificate} compares the scalability of our SDP relaxation in \cref{eq:primal_problem} to the StMM solver with the global certificate check in \cref{eq:dual_certificate} for synthetically generated HPPCA problems of varying data dimension. We measured the median computation time across 10 independent trials of both algorithms. 
The experiment strongly demonstrates the computational superiority of the first-order method with our certificate compared to the full SDP, \revise{as predicted by the computational complexity analysis in \cref{s:theory:ss:dual_certificate}.} StMM+Certificate scaled nearly 60 times better in computation time for the largest dimension with $k=3$ and 15 times for $k=10$, 
while offering a crucial theoretical guarantee to a nonconvex problem that may contain spurious local maxima. Thus, we can solve the nonconvex problem posed in \cref{eq:sum_heterogeneous_quadratics} using any choice of solver on the Stiefel manifold and perform a fast check of its terminal output for global optimality.

\section{Future Work \& Conclusion}
\label{s:conclude}
In this work, we proposed a novel SDP relaxation for the sums of heterogeneous quadratic forms problem, from which we derived a global optimality certificate to check a local solution of a nonconvex program. Our other major contribution proved a continuity result showing sufficient conditions guaranteeing the relaxation has the ROP and providing both theoretical and empirical support that a motivating signal processing application--the HPPCA problem--possesses a tight relaxation in many instances. 

While the global certificate scales well compared to solving the full SDP, the LMI feasibility program still requires forming and factoring $d \times d$ size matrices, requiring storage of $\mathcal{O}(d^2)$ elements. One potential solution is to apply recent works like \cite{yurtsever2021scalable} to our problem, which use randomized algorithms to reduce the storage and arithmetic costs for scalable semidefinite programming. Further, it remains interesting to prove a sufficient analytical certificate as well as proving more general sufficient conditions on the $\M_i$ that guarantee the ROP. \edit{A key future extension is to precisely quantify the size of the region in \cref{thm:main_continuity} where the SDP has the ROP.}

\visedit{Another direction for future research would be to generalize Theorem \ref{thm:main_continuity} or to simplify its proof. 
While the problem in \cite{cifuentes2022stability} is distinct from our own for the reasons discussed in \Cref{s:related_work}, it would be interesting to determine whether the ideas and insights of their theory can be applied in our case.}



\section*{Acknowledgments}
The authors would like to thank Nicolas Boumal for his helpful discussions, references, and notes relating to dual certificates of low-rank SDP's and manifold optimization. They would also like to thank David Hong and Jeffrey Fessler for their feedback on this paper and their discussions relating to heteroscedastic PPCA. We also would like to mention and give special thanks to Alex Wang who pointed out an error in a previous version of this manuscript and for his discussions on how to correct it. 


\appendix

\section{Proofs of Results in \texorpdfstring{\Cref{s:sdp}}{Section 2}}
\label{appendix:sdp_relaxation_pfs}

\begin{proof}[Proof of \cref{lem:strongdualityholds}]
        The problem is convex and satisfies Slater's condition, see Lemma \ref{lem:slater}. Specifically, for optimal primal solutions $\Xb_i$ and optimal dual solutions $\Yb$, $\Zb_i$, and $\nubi_i$ for all $i \in [k]$, we have $\langle \bmI - (\sum_{i=1}^k \Xb_i), \Yb \rangle = 0$ and therefore $\tr(\Yb) = \langle  \Yb ,\sum_{i=1}^k\Xb_i \rangle$. Then
        \begin{align*}
            d^* &= \left \langle \sum_{i=1}^k \bmM_i + \Zb_i - \nubi_i\bmI,  \Xb_i\right\rangle + \sum_{i=1}^k \nubi_i = \tr\left(\sum_{i=1}^k \bmM_i \Xb_i\right),
        \end{align*}
        since $\langle \Zb_i, \Xb_i \rangle= 0$ and $\sum_{i=1}^k\nubi_i(1 - \tr(\Xb_i)) = 0$.
        Thus, $p^* = d^*$.
    \end{proof}

\begin{lemma} \label{lem:slater}
The primal problem in \eqref{eq:primal_problem} is strictly feasible {for $k < d$}. 
\end{lemma}
\begin{proof}
To be strictly feasible we must have $\bmX_i$, $i=1,\dots,k$ such that $$0 \prec \sum_{i=1}^k \bmX_i \prec \bmI, \quad \tr(\bmX_i) = 1, \quad \bmX_i \succ 0, \quad  i=1,\dots,k.$$ Suppose $\bmX_i = \frac{1}{d} \bmI$ for all $i$. Then $\tr(\bmX_i) = 1$ and $\bmX_i \succ 0$ for all $i$, and $\sum_{i=1}^k \bmX_i = \frac{k}{d} \bmI$, satisfying $0 \prec \sum_{i=1}^k \bmX_i \prec \bmI$ {when $k < d$}. 
\end{proof}

\begin{proof}[Proof of \cref{lem:projtight}]
Since the problem in \cref{eq:primal_problem} has a larger constraint set than \cref{eq:sum_heterogeneous_quadratics}, any solution to \cref{eq:primal_problem} that satisfies the constraints of \cref{eq:sum_heterogeneous_quadratics} also constitutes a solution to this original nonconvex problem. 

For the ``if" direction, assume that the optimal $\bmX_i$ for \cref{eq:primal_problem} have the rank-one property. Since $\tr(\bmX_i)=1$ by definition of \cref{eq:primal_problem}, when we decompose $\bmX_i = \bmu_i \bmu_i'$ we have $\bmu_i$ that are norm-1. In order for $\sum_{i=1}^k \bmX_i \preccurlyeq \bmI$, the $\bmu_i$ must be orthogonal. 
For the ``only if" direction, assume that the solution to the SDP relaxation in \eqref{eq:primal_problem} is the optimal solution to the original nonconvex problem in \eqref{eq:sum_heterogeneous_quadratics} {in the sense that $\bmX_i = \bmu_i \bmu_i'$ gives the optimal $\bmU = \left[ \begin{matrix} \bmu_1 & \cdots & \bmu_k\end{matrix} \right]$. Then by definition we see} that the $\bmX_i$ have the rank-one property.

\end{proof}

\begin{lemma}
\label{lem:orthrank1}
 \sloppypar{Suppose $\bmX_i$ for $i=1,\dots,k$ each have trace 1 and satisfy $\lambda_1(\bmX_i) = 1$, and therefore each $\bmX_i$ is rank 1. We decompose $\bmX_i = \bmu_i \bmu_i'$ and note that $\bmu_i$ are norm-1. Then $\sum_{i=1}^k \bmX_i$ satisfies} $0 \preccurlyeq \sum_{i=1}^k \bmX_i \preccurlyeq \bmI$ if and only if $\bmu_i'\bmu_j=0 \quad \forall i\neq j\;.$
\end{lemma}
\begin{proof}
Forward direction: Suppose $\bmX = \sum_{i=1}^k \bmX_i$ has eigenvalues in $[0, 1]$ and $\tr(\bmX) = k$.  Since $\mathrm{rank}(\bmX)\leq k$ by the subadditivity of rank, this implies both that $\bmX$ is rank-$k$ and its eigenvalues are either zero or one. Note then that $$\tr (\bmX\bmX') = k = \tr \left( (\sum_i \bmu_i \bmu_i')(\sum_i \bmu_i \bmu_i') \right) = \sum_i (\bmu_i' \bmu_i)^2 + \tr\left( 2 \sum_{i\neq j}(\bmu_i' \bmu_j)^2\right)\;.$$ Since $\bmu_i$ are norm-1 then the sum $\sum_i (\bmu_i' \bmu_i)^2 = k$. This means $$\tr\left( 2 \sum_{i\neq j}(\bmu_i' \bmu_j)^2\right)= 0\;,$$ which is true if and only if $\bmu_i' \bmu_j = 0$.

The backward direction is immediate because when  $\bmu_i' \bmu_j = 0$ for $i\neq j$, $\sum_{i=1}^k \bmu_i \bmu_i'$ is the eigenvalue decomposition of $\bmX$ with $k$ eigenvalues equal to one. 
\end{proof}

\begin{proof}[Proof of \cref{lem:Zrank_orthoXi}]
     Suppose $\bmZ_i$ is rank $d-1$. By complementarity at optimality, we have $\bmZ_i \bmX_i  = 0 \quad \forall i$, which means $\bmX_i$ lies in the nullspace of $\bmZ_i$, which has dimension 1, so each $\bmX_i$ is rank-one. By primal feasibility, $\tr(\bmX_i) = 1$, so $\lambda_1(\bmX_i) = 1 \quad \forall i=1,\hdots,k$. By \cref{lem:orthrank1}, the optimal solution is an orthogonal projection matrix, and the optimal $\bmX_i$ are orthogonal.
\end{proof}

\section{Proof of \texorpdfstring{\cref{thm:dual_certificate} and \cref{thm:dual_certificate_eps}}{Theorem 4.1 and Corollary 4.2}}
\label{appendix:optimality_conditions}

\begin{proof}[Proof of \cref{thm:dual_certificate}]
By Lemma \ref{lem:strongdualityholds}, primal and dual feasible solutions of \cref{eq:primal_problem} and \cref{eq:dual_problem}, $\Xb_i, \Zb_i, \Yb, \nub$, are simultaneously optimal if and only if they satisfy the following KKT conditions \cite{boyd2004convex}, where the variables and constraints are indexed by $i \in [k]$:
\begin{align*}
        &\Xb_i \succeq 0, \quad \sum_{i=1}^k \Xb_i \preceq \bmI, \quad \tr(\Xb_i) = 1 \label{KKT:a} \tag{KKT-a}\\
        &\Yb = \bmM_i + \Zb_i - \nubi_i \bmI,  \qquad \Yb \succeq 0 \label{KKT:b} \tag{KKT-b}\\
        & \left\langle \bmI - \sum_{i=1}^k \Xb_i, \Yb\right\rangle = 0 \label{KKT:c} \tag{KKT-c}\\
        & \langle \Zb_i, \Xb_i \rangle = 0 \label{KKT:d} \tag{KKT-d}\\
        &\Zb_i \succeq 0 \label{KKT:e} \tag{KKT-e}.
\end{align*}

Similar to the work in \cite{won2021orthogonalSiamJMAA}, our strategy is then to construct $\Xb_i$ and $\Yb, \Zb_i, \nub$ satisfying these conditions. Given $\Ub$ and $\nub$ in the statement of the theorem, we define $\Xb_i = \ub_i \ub_i'$, $\Yb = \Ub(\Lambdab - \bmD_{\nub})\Ub'$, and $\Zb_i = \Yb + \bar{\nu}_i \bmI - \bmM_i$. By construction, $\Xb_i$ satisfy \cref{KKT:a}. \edit{Also by construction $\Yb = \bmM_i + \Zb_i - \nub_i \bmI$, and the assumption that $\Lambdab \succeq \bmD_{\nub}$ ensures $\Yb \succeq 0$ to satisfy (\ref{KKT:b}).} One can also verify that $\langle \bmI - \Xb, \Yb \rangle = 0$ by construction, thus satisfying \cref{KKT:c}. So it remains to show 
$\langle \Zb_i, \Xb_i \rangle = 0$ and $\Zb_i \succeq 0$.


Moreover, $\Zb_i \succeq 0$ by the assumption in \cref{eq:dual_certificate}, satisfying \cref{KKT:e}. We finally verify \cref{KKT:d}, i.e. $\langle \Zb_i, \Xb_i \rangle = 0$, with $\Ub = [\ub_1 \cdots \ub_k]$:
    \begin{align*}
        \langle \Zb_i, \Xb_i \rangle &= \langle \Yb + \bar{\nu}_i \bmI - \bmM_i, \Xb_i \rangle  = \langle \Ub(\Lambdab - \bmD_{\nub})\Ub' + \bar{\nu}_i \bmI - \bmM_i, \ub_i \ub_i'\rangle \\
        &= \ub_i' \Ub \,\Ub' \sum_{j=1}^k \bmM_j \Ub \bmE_j \Ub' \ub_i - \ub_i' \Ub \bmD_{\nub} \Ub' \ub_i + \bar{\nu}_i - \ub_i' \bmM_i \ub_i \\
        &= \bme_i' \Ub' \sum_{j=1}^k \bmM_j \ub_j \bme_j' \bme_i - \bme_i'\bmD_{\nub}\bme_i + \bar{\nu}_i - \ub_i' \bmM_i \ub_i \\
        &= \ub_i' \bmM_i \ub_i - \bar{\nu}_i + \bar{\nu}_i - \ub_i' \bmM_i \ub_i = 0.
    \end{align*} \end{proof}

\begin{remark}
    \edit{Given the fact that the Lagrange multipliers $\nubi_i$ corresponding to the trace constraints are nonnegative by \cref{lem:nuneg}, this also implies that $\Lambdab \succeq 0$.
We note that this indeed fulfills a necessary condition for $\Ub$ to be a second-order stationary point by \cref{lem:stationary_conditions} and \cref{lem:lambda:psd} in the supplement.}
\end{remark}
See \cref{appendix:dual_certificate_remarks} for additional remarks.

For the following results in this paper, we require a proof that the optimal $\nub$ in \cref{eq:dual_problem} are nonnegative.

\begin{lemma} \label{lem:nuneg}
Assume all $\bmM_i$ are PSD, and $k < d$. Then all $\nu_i \geq 0$ at optimality.
\end{lemma}

\begin{proof}

\edit{For a contradiction suppose the optimal $\bmnu$ has at least one coordinate that is strictly negative.}
Without loss of generality, let $\nu_1 < 0$ be the smallest (most negative) coordinate of $\bmnu$, and rewrite the objective in terms of $\bmM_1$ and eliminating $\bmY$ as
\begin{align}
    d^* =\min_{\nu_i, \bmZ_i}&   \tr(\bmZ_1 + \bmM_1) - d \nu_1 + \sum_{i=1}^k \nu_i \label{dual:onlyM1} \\
    \text{s.t. }
    &\bmM_i + \bmZ_i \succcurlyeq \nu_i\bmI \quad \forall i=1,\hdots,k \nonumber \\
    & \bmM_1 + \bmZ_1 - \nu_1\bmI = \bmM_j + \bmZ_j - \nu_j\bmI \quad \forall j=2,\hdots,k \label{eq:equalityconstraints} \\
    & \bmZ_i \succcurlyeq 0 \quad \forall i=1,\dots,k.\nonumber
\end{align}
Now consider new variables $\{\tilde\nu_i, \tilde\bmZ_i\}_{i=1}^k$, where we let $\tilde \nu_1 = 0$, $\tilde \nu_i = \nu_i - \nu_1$ for $i=2,\dots,k$, and leave all the $\bmZ$ variables unchanged: $\tilde \bmZ_i = \bmZ_i$ for all $i$. 

These new variables are still feasible. Certainly $\bmM_1+\tilde \bmZ_1 = \bmM_1 + \bmZ_1 \succcurlyeq \tilde \nu_1 \bmI = 0$ as both $\bmM_1, \bmZ_1$ are PSD. Also $\bmM_1 + \tilde \bmZ_1 - \tilde \nu_1\bmI = \bmM_j + \tilde \bmZ_j - \tilde \nu_j\bmI$, since substituting in, we have $\bmM_1 + \bmZ_1 = \bmM_j + \bmZ_j - (\nu_j - \nu_1) \bmI$, which was feasible for the original optimal point. From this last equation note that since $\bmM_1+\bmZ_1\succcurlyeq 0$, then $\bmM_j + \bmZ_j - (\nu_j - \nu_1)\bmI = \bmM_j + \tilde \bmZ_j - \tilde \nu_j \bmI \succeq 0$. 

However, \revise{with the assumption that $k<d$}, this yields a contradiction because we have reduced the objective value from $$\tr(\bmZ_1+\bmM_1) - d \nu_1 + \sum_{i=1}^k \nu_i \qquad \text{to} \qquad \tr(\bmZ_1+\bmM_1) - k \nu_1 + \sum_{i=1}^k \nu_i\;.$$ 
Therefore $\nu_i<0$ cannot be optimal.
\end{proof}

\begin{proof}[Proof of \cref{thm:dual_certificate_eps}]
We first argue that this problem attains an optimal solution as follows. We note that (\ref{eq:dual_certificate_eps}) is feasible by taking $\nub = 0$ and $\epsilon$ sufficiently large. Next, the optimal value of (\ref{eq:dual_certificate_eps}) is clearly bounded below by 0. In addition, for any fixed $\bar\epsilon$, one can see that the level set of feasible points $(\epsilon, \nub)$ with $\epsilon \le \bar\epsilon$ is bounded via the constraint $0 \preceq  \D_{\nub} \preceq \Lambdab + \epsilon \I$, which in particular bounds each entry of $\nub$ \edit{from below by \cref{lem:nuneg}} and from above by the corresponding diagonal entry of $\Lambdab$. Hence, an optimal solution $(\epsilon^*, \nub^*)$ is attained.
Let $\epsilon^*$ be the unique optimal value of the optimization problem. From this $\epsilon^*$, now we construct a solution to the following approximate KKT conditions \cite{boyd2004convex} of (SDP-P), indexing the variables and constraints by $i \in [k]$:
\begin{align*}
        &\Xb_i \succeq 0, \quad \sum_{i=1}^k \Xb_i \preceq \bmI, \quad \tr(\Xb_i) = 1 \label{KKT:a_eps} \tag{eps-KKT-a}\\
        &\Yb = \bmM_i + \Zb_i - \nubi_i \bmI,  \qquad \Yb \succeq -\epsilon^* \I \label{KKT:b_eps} \tag{eps-KKT-b}\\
        & \left\langle \bmI - \sum_{i=1}^k \Xb_i, \Yb \right\rangle = 0 \label{KKT:c_eps} \tag{eps-KKT-c}\\
        & \langle \Zb_i, \Xb_i \rangle = 0 \label{KKT:d_eps} \tag{eps-KKT-d}\\
        &\Zb_i \succeq -\epsilon^* \I \label{KKT:e_eps} \tag{eps-KKT-e}.
\end{align*}
Given a $\Ub$ and optimal $\nub$ to \cref{eq:dual_certificate_eps}, we define $\Xb_i = \ub_i \ub_i'$, $\Yb = \Ub(\Lambdab - \bmD_{\nub})\Ub'$, and $\Zb_i = \Yb + \nubi_i \bmI - \bmM_i$. By construction, $\Xb_i$ satisfy \cref{KKT:a_eps}, and it is clear that $\Yb = \bmM_i + \Zb_i - \nub_i \bmI$ satisfies the first condition in \cref{KKT:b_eps}. One can also verify that $\langle \bmI - \Xb, \Yb \rangle = 0$ by construction, thus satisfying \cref{KKT:c_eps}. 

One can easily show that $\Lambdab - \bmD_{\nub} \succeq -\epsilon^* \I$
ensures $\Yb \succeq -\epsilon^* \I$ (\ref{KKT:b_eps}).
Moreover, $\Zb_i \succeq -\epsilon^* \I$ by the assumption in \cref{eq:dual_certificate_eps},
satisfying \cref{KKT:e_eps}. Just as we did in the proof of Theorem 4.1 we finally verify \cref{KKT:d_eps}, i.e.
$\langle \Zb_i, \Xb_i \rangle = 0$, with $\Ub = [\ub_1 \cdots \ub_k]$.

Let us now focus on $\Yb, \Zb_i$, which are approximately feasible for the dual problem. By defining $\Y := \Yb + \epsilon^* \I$, $\bmZ_i := \Zb_i + \epsilon^* \I$, $\bmZ := (\bmZ_1,\hdots,\bmZ_k)$, and $\bmnu := \nub$, we recover dual feasibility, i.e., $\bmY \succeq 0$ and $\bmZ_i \succeq 0$. Hence, the duality gap between $\Xb_1, \hdots, \Xb_k$ and $\bmY, \bmZ, \bmnu$ is nonnegative and, in fact, equals $\epsilon^* d$ due to the approximate KKT system:
\begin{align*}
    d(\bmY,\bmZ,\bmnu) - p(\Ub) &= \tr(\bmY) + \sum_i \bmnu_i - \sum_i \langle \bmM_i,  \Xb_i \rangle\\
    &= \tr(\bmY) + \sum_i \bmnu_i - \sum_i \langle  \bmY - \bmZ_i + \bmnu_i \I  , \Xb_i \rangle\\
    &= \tr(\bmY) + \sum_i \bmnu_i - \left \langle \bmY , \sum_i \Xb_i  \right \rangle + \sum_i \langle \bmZ_i , \Xb_i \rangle - \sum_i \bmnu_i \\
    &= \left \langle \bmY , \I - \sum_i \Xb_i \right \rangle + \sum_i \langle \bmZ_i ,\Xb_i \rangle \\
    &= \left \langle \Yb + \epsilon^* \I  , \I - \sum_i \Xb_i \right \rangle + \sum_i \langle \Zb_i + \epsilon^*\I , \Xb_i \rangle \\
    &= \epsilon^* \tr \left(\I - \sum_i \Xb_i \right) + \epsilon^* \sum_i \tr(\Xb_i)  = \epsilon^* \tr(\I) = \epsilon^* d.
\end{align*}
In other words, letting $p(\Ub)$ be the primal objective associated with $\Ub$ and
$d(\bmY,\bmZ,\bmnu)$ be the dual objective associated with $\bmY, \bmZ, \bmnu$, we have shown that the duality gap $d(\bmY,\bmZ,\bmnu) - p(\Ub) =
\epsilon^* d$, which implies $p(\Ub) = d(\bmY,\bmZ,\bmnu) - \epsilon^* d \ge d^*(\bmY,\bmZ,\bmnu) - \epsilon^* d = p^* - \epsilon^* d$.
\end{proof}

\section{Proofs of intermediate results supporting \texorpdfstring{\cref{thm:main_continuity}}{Theorem 4.7}} \label{sec:continuity}

\revise{Next, we give general convex analysis results that allow us to prove Theorem \ref{thm:main_continuity}.} 

Let $\Cc \subseteq \mathbb{R}^n$ be a closed, convex set. For all $\bmc \in
\Cc$, consider a primal-dual pair of linear conic programs parameterized
by $c$:


\begin{align}
    p(\bmc) &:= \min_{\bmx} \{\bmc' \bmx : \bmA \bmx = \bmb, \bmx \in \K\}  \label{equ:p} \tag{$P;\bmc$}\\
    d(\bmc) &:= \max_{\bmy} \{ \bmb' \bmy : \bmc - \bmA' \bmy \in \K^* \} \label{equ:d} \tag{$D;\bmc$}
\end{align}

\noindent Here, the data $\bmA \in \mathbb{R}^{m \times n}$ and $\bmb \in
\mathbb{R}^m$ are fixed; $\K \subseteq \mathbb{R}^n$ is a closed, convex
cone; and $\K^* := \{ \bms \in \mathbb{R}^n : \bms' \bmx \ge 0 \ \forall \ \bmx
\in \K \}$ is its polar dual. We imagine, in particular, that $\K$ is
a direct product of a nonnegative orthant, second-order cones, and
positive semidefinite cones, corresponding to linear, second-order-cone,
and semidefinite programming.

Define $\mathrm{Feas}(P) := \{ \bmx \in \K : \bmA \bmx = \bmb \}$ and $\mathrm{Feas}(D;\bmc)
:= \{ \bmy : \bmc - \bmA' \bmy \in \K^* \}$ to be the feasible sets of $(P;\bmc)$
and $(D;\bmc)$, respectively. We assume:

\begin{assumption} \label{ass:sets}
$\mathrm{Feas}(P)$ is interior feasible, and $\mathrm{Feas}(D;\bmc)$ is
interior feasible for all $\bmc \in \Cc$.
\end{assumption}

\noindent Then, for all $\bmc$, strong duality holds between $(P;\bmc)$ and
$(D;\bmc)$ in the sense that $p(\bmc) = d(\bmc)$ and both $p(\bmc)$ and $d(\bmc)$ are
attained in their respective problems. Accordingly, we also define
$$\mathrm{Opt}(D;\bmc) := \{ \bmy \in \mathrm{Feas}(D;\bmc) : \bmb' \bmy = d(\bmc) \}$$ to be
the nonempty, dual optimal solution set for each $\bmc \in \Cc$.

In addition, we assume the existence of linear constraints $\bmf - \bmE' \bmy
\ge 0$, independent of $\bmc$, such that $$\mathrm{Extra}(D) := \{ \bmy : \bmf - \bmE'
\bmy \ge 0 \}$$ satisfies:

\begin{assumption} \label{ass:extra}
For all $\bmc \in \Cc$, $\mathrm{Feas}(D;\bmc) \cap \mathrm{Extra}(D)$ is interior
feasible and bounded, and $\mathrm{Opt}(D;\bmc) \subseteq \mathrm{Extra}(D)$.
\end{assumption}

\noindent In words, irrespective of $\bmc$, the extra constraints $\bmf - \bmE'
\bmy \ge 0$ bound the dual feasible set without cutting off any optimal
solutions and while still maintaining interior, including interiority with respect
to $\bmf - \bmE' \bmy \ge 0$. Note also that Assumption \ref{ass:extra} implies
the recession cone of $\mathrm{Feas}(D;\bmc) \cap \mathrm{Extra}(D)$ is trivial
for (and independent of) all $\bmc$, i.e., $\{ \bmDelta \bmy : -\bmA' \bmDelta \bmy \in
\K^*, -\bmE' \bmDelta \bmy \ge 0 \} = \{0\}$.

We first prove a continuity result related to the dual feasible
set, in which we use the following definition of a convergent sequence
of bounded sets in Euclidean space: a sequence of bounded sets $\{ L^k
\}$ converges to a bounded set $\bar L$, written $\{ L^k \} \to \bar L$, if and only
if: (i) given any sequence $\{ \bmy^k \in L^k \}$, every limit point $\bar
\bmy$ of the sequence satisfies $\bar \bmy \in \bar L$; and (ii) every member
$\bar \bmy \in \bar L$ is the limit point of some sequence $\{ \bmy^k \in L^k
\}$.

\begin{lemma} \label{lem:dualsetconv}
Under Assumptions \ref{ass:sets} and \ref{ass:extra}, let $\{ \bmc^k
\in \Cc \} \to \cb$ be any convergent sequence. Then
\[
    \left\{ \mathrm{Feas}(D;\bmc^k) \cap \mathrm{Extra}(D) \right\}
    \ \ \to \ \
    \mathrm{Feas}(D;\cb) \cap \mathrm{Extra}(D).
\]
\end{lemma}

\begin{proof}
    See \cref{appendix:lemmaC1:proof} in the supplement for the proof.
\end{proof}

\begin{lemma} \label{lem:optsetconv}

Under Assumptions \ref{ass:sets} and \ref{ass:extra}, let $\{ \bmc^k
\in \Cc \} \to \cb$ be any convergent sequence. Then
\[
    \left\{ \mathrm{Opt}(D;\bmc^k) \right\}
    \ \ \to \ \
    \mathrm{Opt}(D;\cb).
\]
\end{lemma}

\begin{proof}
  See \cref{appendix:lemmaE2:proof} in the supplement for the proof.
\end{proof}

Finally, for given $\bmc \in \Cc$ and fixed $\bmy^0 \in \mathbb{R}^m$, we define
the function
\[
    y(\bmc) := y(\bmc;\bmy^0)
    = \argmin \{ \|\bmy - \bmy^0\| : \bmy \in \mathrm{Opt} (D;\bmc) \},
\]
i.e., $y(\bmc)$ equals the point in $\mathrm{Opt}(D;\bmc)$, which is closest
to $\bmy^0$. Since $\mathrm{Opt} (D;\bmc)$ is closed and convex, $y(\bmc)$ is
well defined. We next use \cref{lem:optsetconv} to show that $y(\bmc)$ is continuous in $\bmc$.

\begin{proposition} \label{pro:ycont}
Under the Assumptions \ref{ass:sets} and \ref{ass:extra}, given $\bmy^0
\in \mathbb{R}^m$,
the function $y(\bmc) := y(\bmc;\bmy^0)$ is continuous in $\bmc$.
\end{proposition}

\begin{proof}
We must show that, for any convergent $\{ \bmc^k \} \to \cb$, we also have convergence $\{ y(\bmc^k) \} \to y(\cb)$. This follows because $\{ \mathrm{Opt}(D;\bmc^k) \} \to \mathrm{Opt}(D;\cb)$ by \cref{lem:optsetconv}.
\end{proof}

Theorem \ref{thm:main_continuity} uses Proposition \ref{pro:ycont} in its proof. Here we discuss how the primal-dual pair \eqref{eq:primal_problem}-\eqref{eq:dual_problem} satisfy the assumptions for the proposition. We would like to establish conditions under which (\ref{eq:primal_problem}) has the rank-one property. For this, we apply the general theory developed
above, specifically Proposition \ref{pro:ycont}. To show that the general theory applies, we must define the closed,
convex set $\Cc$, which contains
the set of admissible objective matrices/coefficients $(\bmM_1,\ldots,\bmM_k)$ and which satisfies Assumptions \ref{ass:sets} and \ref{ass:extra}.
In particular, for a fixed, user-specified upper bound $\mu > 0$, we define $ \Cc :=
    \{ \bmc = (\bmM_1,\ldots,\bmM_k) : 0 \preceq \bmM_i \preceq \mu \bmI \quad \forall \ i = 1,\ldots,k \}$
to be our set of admissible coefficient $k$-tuples. 
\edit{In addition, we have shown in \cref{lem:nuneg} that all $\bmM_i \succeq 0$ implies that all $\nu_i$ are nonnegative at optimality. Thus, we enforce the redundant constraint that $\nu_i \geq 0$ for all $i \in [k]$.}

We know that both (\ref{eq:primal_problem}) and (\ref{eq:dual_problem}) have interior
points for all $\bmc \in \Cc$, so that strong duality holds. For the
dual in particular, the equation $\mu \bmI = \bmM_i + ((\mu + \epsilon) \bmI -
\bmM_i) - \epsilon \bmI$ shows that, for all $\epsilon > 0$, $ \bmY(\epsilon) := \mu \bmI, ~~ \bmZ(\epsilon)_i := (\mu + \epsilon) \bmI - \bmM_i, ~~\nu(\epsilon)_i := \epsilon$
is interior feasible with objective value $d \mu + k \epsilon$. In
particular, the redundant constraint $\bmnu \geq 0$ is satisfied strictly.
This verifies Assumption \ref{ass:sets}.

We next verify Assumption \ref{ass:extra}. Since the objective value
just mentioned is independent of $\bmc = (\bmM_1,\ldots,\bmM_k)$, we can take
$\epsilon = 1$ and enforce the extra constraint $\tr(\bmY) + \sum_{i=1}^k
\nu_i \le d \mu + k$ without cutting off any dual optimal solutions
and while still maintaining interior. In particular, the solution
$(\bmY(\tfrac12), \bmZ(\tfrac12)_i, \bmnu(\tfrac12)_i)$ corresponding to
$\epsilon = \tfrac12$ satisfies the new, extra constraint strictly.
Finally, note that $\tr(\bmY) + \sum_i \nu_i \le d \mu + k$ bounds $\bmY$ and
$\bmnu$ in the presence of the constraints $\bmY \succeq 0$ and $\bmnu \ge 0$,
and consequently the constraint $\bmZ_i = \bmY - \bmM_i + \nu_i \bmI$ bounds $\bmZ_i$
for each $i$.

We now repeat the discussion leading up to Theorem \ref{thm:main_continuity} for completeness. The first lemma says that the diagonal problem has dual variables $\bmZ_i$ such that $\rank(\bmZ_i) \ge d-1$, implying that the primal variables $\bmX_i$ are rank-one.

\begin{proof}[Proof of \cref{lem:gt}]
Because of the jointly diagonalizable property, we may assume without loss of generality that each $\bmM_i$ is diagonal. So (\ref{eq:primal_problem}) is equivalent to the
assignment LP
\[
    \max \left\{ \sum_{i=1}^k \diag(\bmM_i)' \diag(\bmX_i) :
        \begin{array}{l}
            \bme' \diag(\bmX_i) = 1, \ \diag(\bmX_i) \ge 0 \ \ \forall \ i=1,\ldots,k \\
            \sum_{i=1}^k \diag(\bmX_i) \le \bme
        \end{array}
    \right\},
\]
where $\bme$ is the vector of all ones, and (\ref{eq:dual_problem}) is equivalent to the LP
\[
    \min \left\{
        \bme' \diag(\bmY) + \sum_{i=1}^k \nu_i :
        \begin{array}{l}
            \diag(\bmY) = \diag(\bmM_i) + \diag(\bmZ_i) - \nu_i \bme \ \ \forall \ i = 1,\ldots,k \\
            \diag(\bmZ_i) \ge 0 \ \ \forall \ i = 1,\ldots,k, \quad \diag(\bmY) \ge 0 \\
        \end{array}
    \right\}.
\]

Since the primal is an assignment problem, its unique optimal solution
has the property that each $\diag(\bmX_i)$ is a standard basis vector
(i.e., each has a single entry equal to 1 and all other entries equal to
0). By the Goldman-Tucker strict complementarity theorem for LP, there
exists an optimal primal-dual pair such that $\diag(\bmX_i) + \diag(\bmZ_i)
> 0$ for each $i$. Hence, there exists a dual optimal solution with
$\rank(\bmZ_i) \ge d-1$ for each $i$, as desired.
\end{proof}

\begin{proof}[Proof of \cref{cor:Mi_commute:shq}]
\sloppypar{
\revise{We apply \cref{lem:hilbert_schmidt_lins_thm} to $(\bmA_1,\hdots,\bmA_L)$.
Then there exist Hermitian symmetric matrices $\bar \bmA_\ell$ such that $\|[\bar \bmA_\ell, \bar \bmA_m]\|_\mathrm{tr} = 0$ for all $\ell, m \in [L]$ such that $\|\bmA_\ell - \bar \bmA_\ell\|_\mathrm{tr} \leq \delta (\epsilon,k)$ for all $\ell \in [L]$.  Let $\Mb_i := \sum_{\ell=1}^L w_{\ell,i} \bar \bmA_{\ell}$.} Then the matrices $\Mb_i$ commute and are jointly diagonalizable:}
        \begin{align}
            [\Mb_i, \Mb_j] = \Mb_i \Mb_j -  \Mb_j \Mb_i &=  2\sum_{\ell \neq m}^L w_{\ell,i}  w_{m,j}( \bar \bmA_\ell \bar \bmA_m - \bar \bmA_m \bar \bmA_\ell) = 0.
        \end{align}
    Now we measure the distance between each $\bmM_i$ and $\Mb_i$:
    \begin{align}
        \|\bmM_i - \Mb_i\|_\mathrm{tr} = \left \| \sum_{\ell=1}^L w_{\ell,i} (\bmA_\ell - \bar \bmA_\ell) \right\|_\mathrm{tr} \leq \sum_{\ell=1}^L w_{\ell,i} \|\bmA_\ell - \bar \bmA_\ell \|_\mathrm{tr} \leq \sum_{\ell=1}^L w_{\ell,i}\delta(\epsilon,k).
    \end{align}
\end{proof}

\edit{The following lemma is used in the proof of \cref{prop:hppca:almost_commute}.}
\begin{lemma}
\label{lem:hppca:commute:sample}
    Let $\Mb_i := \mathbb{E}[{\frac{1}{n}}\bmM_i] \in \bbR^{d \times d}$, where the expectation is taken with respect to the {normalized} data observations, and let $C > 0$ be a universal constant.
    Then $\|[\Mb_i, \Mb_j ]\| = 0$, and with probability at least $1 - e^{-t}$ for $ t > 0$
    \begin{align}
        \frac{\|{\frac{1}{n}}\bmM_i - \Mb_i \|}{\|\Mb_1\|} \leq 
        C \frac{\bar {\sigma}_i}{\bar{\sigma}_1} \max\left\{\sqrt{\frac{\frac{\bar{\xi}_i}{\bar{\sigma}_i} \log d  + t}{n}}, \frac{\frac{\bar{\xi}_i}{\bar{\sigma}_i}\log d + t}{n}\log(n) \right \},~~\text{where}
    \end{align}
    \begin{align*}
    \bar{\sigma}_i &= \|\Mb_i \| = \sum_{\ell=1}^L \frac{\frac{\lambda_i}{v_\ell}}{\frac{\lambda_i}{v_\ell} + 1} \frac{n_\ell}{n} \left(\frac{\lambda_1}{v_\ell} + 1\right),\\
    \bar{\xi}_i &= \tr(\Mb_i) = \sum_{\ell=1}^L \frac{\frac{\lambda_i}{v_\ell}}{\frac{\lambda_i}{v_\ell} + 1}\frac{n_\ell}{n} \left(\frac{1}{v_\ell} \sum_{j=1}^k \lambda_j + d\right).
    \end{align*}
\end{lemma}
\begin{proof}
    \sloppypar{Let $\tilde{\bmy}_{\ell,j} := \sqrt{\frac{w_{\ell,i}}{v_\ell}} \bmy_{\ell,j}$ be a rescaling of the data vectors. Then $\tilde{\bmy}_{\ell,j} \overset{iid}{\sim} \mathcal{N}(\bm{0}, w_{\ell,i}(\frac{1}{v_\ell} \bmU \bmTheta^2 \bmU' + \bmI))$. After rescaling, for notational purposes let $\bmM_i = \frac{1}{n} \sum_{\ell=1}^L \sum_{j=1}^{n_\ell} \tilde{\bmy}_{\ell,j}\tilde{\bmy}_{\ell,j}'$. Taking the expectation over the data, we have}
    \begin{align}
        \mathbb{E}[\bmM_i] = \frac{1}{n} \sum_{\ell=1}^L \sum_{j=1}^{n_\ell} \mathbb{E}[\tilde{\bmy}_{\ell,j}\tilde{\bmy}_{\ell,j}'] = \sum_{\ell=1}^L w_{\ell,i}\frac{n_\ell}{n} \left(\frac{1}{v_\ell} \bmU \bmTheta^2 \bmU' + \bmI\right). 
    \end{align}
    Let $\bmU_\perp \in \bbR^{d \times d-k}$ be an orthonormal basis spanning the orthogonal complement of $\mathrm{Span}(\bmU)$. Noting that $\bmI = \bmU \bmU' + \bmU_\perp \bmU_\perp'$, rewrite $\mathbb{E}[\bmM_i]$ in terms of its eigendecomposition by
    \begin{align}
        \mathbb{E}[\bmM_i] &= \bmU \left(\sum_{\ell=1}^L w_{\ell,i}\frac{n_\ell}{n}\left( \frac{1}{v_\ell} \bmTheta^2 + \bmI_k\right)\right) \bmU' + \left(\sum_{\ell=1}^L w_{\ell,i}\frac{n_\ell}{n}\right) \bmU_\perp \bmU_\perp'\\
        &= \begin{bmatrix} \bmU & \bmU_\perp \end{bmatrix} \begin{bmatrix} \edit{\bmSigma_i} & 0 \\ 0 & \edit{\gamma_i} \bmI_{d-k} \end{bmatrix} \begin{bmatrix} \bmU' \\ \bmU_\perp'\end{bmatrix},
    \end{align}
    where $\edit{\bmSigma_i} := \sum_{\ell=1}^L w_{\ell,i}\frac{n_\ell}{n}\left( \frac{1}{v_\ell} \bmTheta^2 + \bmI_k\right) $ and $\edit{\gamma_i} := \sum_{\ell=1}^L w_{\ell,i}\frac{n_\ell}{n}$, from which we obtain the expressions for $\bar{\sigma}_i = \|\mathbb{E}[\bmM_i]\|$ and $\bar{\xi}_i = \tr(\mathbb{E}[\bmM_i])$.
    Then invoking \cref{lem:sample_covar_error_bound} in the supplement to bound the concentration of a {normalized} sample covariance matrix to its expectation with high probability yields the final result.
\end{proof}

\begin{proof}[Proof of \cref{prop:hppca:almost_commute}]
    
    We argue there are two possible sets of commuting $(\Mb_1,\hdots,\Mb_k)$ that $(\bmM_1,\hdots,\bmM_k)$ can converge to, depending on the signal to noise ratios $\frac{\lambda_i}{v_\ell}$ and the number of samples $n$.
        
        Consider that we can scale all the $\bmM_i$ in \cref{eq:primal_problem} by a positive scalar constant without changing the optimal solution. Since all the $\bmM_i$ can be arbitrarily scaled in this manner, and thereby changing any distance measure, we will choose to normalize the matrices $\bmM_i$ and $\Mb_i$ by the number of samples and the largest spectral norm of the $\Mb_i$, which is equivalent to also normalizing the distance. \new{Using the definition of the weights $w_{\ell,i}$ in HPPCA, it is straightforward to show that $\bmM_1 \succeq \bmM_2 \succeq \cdots \succeq \bmM_k$.}
        Accordingly,
        we normalize by $1/\|n \Mb_1\|$.
        
        First, if the variances are zero or all the same, i.e. noiseless or homoscedastic noisy data, then all the $\bmM_i$ are equal. Otherwise, in the case where each SNR $\lambda_i/v_\ell$ of the $i^{\text{th}}$ components is large or close to the same value for all $\ell \in [L]$, the weights $w_{\ell,i} = \frac{\lambda_i/v_\ell}{\lambda_i/v_\ell + 1}$ are very close to 1 or some constant less than 1, respectively. Therefore, let  $\Mb := \frac{1}{n}\sum_{\ell=1}^L {\bar{v}} \bmA_\ell$ for some {$\bar{v} \geq 0$} for all $i \in [k]$, where recall from \cref{eq:hppca:generative_model} that $\bmA_\ell = \sum_{j=1}^{n_\ell} \frac{1}{v_\ell}\bmy_{\ell,j} \bmy_{\ell,j}'$. Then
    \begin{align}
    \label{eq:hppca:almost_commute:variance}
        \frac{\|\textcolor{black}{\frac{1}{n}} \bmM_i - \Mb\|}{\| \Mb\|} = \frac{\frac{\lambda_i}{\lambda_i + \bar{v}} \|\sum_{\ell=1}^L \frac{(\bar{v} - v_\ell)/v_\ell}{\lambda_i/v_\ell + 1} \bmA_\ell \|}{\frac{\lambda_i}{\lambda_i + \bar{v}} \sum_{\ell=1}^L \|\bmA_\ell\|} \leq \frac{\sum_{\ell=1}^L \frac{|\bar{v} - v_\ell|/v_\ell}{\lambda_i/v_\ell + 1} \|\bmA_\ell\|}{\|\sum_{\ell=1}^L \bmA_\ell\|} \leq \sum_{\ell=1}^L \frac{\frac{|\bar{v}- v_\ell|}{v_\ell}}{\frac{\lambda_i}{v_\ell} + 1},
    \end{align}
    \noindent where the last inequality above results from the fact $\frac{\|\bmA_\ell\|}{\|\sum_{\ell=1}^L \bmA_\ell\|} \leq 1$ for all $\ell \in [L]$ using Weyl's inequality for symmetric PSD matrices \cite{horn_johnson:85}.
    While the bound above depends on the SNR {and the gaps between the variances}, it fails to capture the effects of the sample sizes, which also play an important role in how close the $\bmM_i$ are to commuting. Even in the case where the variances are larger and more heterogeneous, since the $\bmM_i$ form a weighted sum of sample covariance matrices, given enough samples, they should concentrate to their respective sample covariance matrices, which commute between $i,j \in [k]$. We show exactly this using the concentration of sample covariances to their expectation in \cite{Lounici2014}, and choose $\cb = (\Mb_1, \hdots, \Mb_k)$ for $\Mb_i := \mathbb{E}[{\frac{1}{n}}\bmM_i]$, where the expectation here is with respect to the {normalized} data generated by the model in \cref{eq:hppca:generative_model}.
    
    Let $\Mb_i := \mathbb{E}[{\frac{1}{n}}\bmM_i] \in \bbR^{d \times d}$, where the expectation is taken with respect to the {normalized} data observations. Then by \cref{lem:hppca:commute:sample} and taking the minimum with \cref{eq:hppca:almost_commute:variance}, we obtain the final result.
    \end{proof}

\printbibliography
\begin{refsection}
\section{Related work}
\label{appendix:related}
In this extended related work discussion, we first describe works very closely related to our problem in \cref{eq:sum_heterogeneous_quadratics}, and then describe works more generally related to SDP relaxations of rank or orthogonality constrained problems. 

\cite{supp:bolla:98}, \cite{supp:rapcsak2002minimization}, and \cite{supp:Berezovskyi2008} also previously investigated the sum of heterogeneous quadratic forms in \cref{eq:sum_heterogeneous_quadratics}. The work in \cite{supp:bolla:98} only studied the structure of this problem for some special cases where all of the matrices $\bmM_i$ were either equal, diagonal, or commuting. \cite{supp:rapcsak2002minimization} derived sufficient second-order global optimality conditions for the Hessian of the Lagrangian. However, these conditions, in general, do not always hold and are usually difficult to check in practice since they require computing the eigenvalues of the large $kd \times kd$ Hessian matrix. \cite{supp:Berezovskyi2008} also analyzed the Lagrangian of the problem, but only for the case of Boolean problem variables.

Works such as \cite{supp:huang2009rank} and \cite{supp:Pataki1998OnTR} consider a very similar problem to \eqref{eq:sum_heterogeneous_quadratics}, but without the constraint summing the $\bmX_i$ in \eqref{eq:ncvxprimal}, making their SDP a rank-constrained separable SDP; see also \cite[Section 4.3]{supp:luo2010sdp}. {Pataki \cite{supp:Pataki1998OnTR} studied upper bounds on the rank of optimal solutions of general SDPs, but in the case of \cref{eq:primal_problem}, since our problem introduces the additional constraint summing the $\bmX_i$, Pataki's bounds do not guarantee rank-1, or even low-rank, optimal solutions.}

Our problem also has interesting connections to the well-studied problem in the literature of approximate joint diagonalization (AJD), which is often applied to blind source separation or independent component analysis (ICA) problems \cite{supp:theis_ica_jd, supp:bouchard_malick_congedo2018, supp:kleinsteuber_shen_2013, supp:afsari_jd2008, supp:Shi2011}. Given a set of symmetric PSD matrices that represent second order data statistics, one seeks the matrix, usually constrained to lie in the set of orthogonal or invertible matrices, that jointly diagonalizes the set of matrices optimally, albeit approximately. When all matrices in the set commute, the diagonalizer is simply the shared eigenspace, but often in practice, due to noise, finite samples, or numerical errors, the set does not commute and can only be approximately diagonalized.

Expanding our matrix variable $\bmU \in \bbR^{d \times k}$ to a full basis $\bmU \in \bbR^{d \times d}$, the heteroscedastic probabilistic PCA (HPPCA) problem in \cref{eq:origobj} is equivalent to
\begin{align}
    \label{eq:hppca_frobnorm}
    \min_{\bmU \in \R^{d \times d}: \bmU'\bmU = \bmU \bmU' = \bmI} \sum_{\ell=1}^L \frac{1}{2}\|\bmU' \bmA_\ell \bmU - \bmW_\ell\|_{\mathrm{F}}^2 + C,
\end{align}
where $\bmW_\ell = \mathrm{diag}(w_{\ell,1},\hdots,w_{\ell,k},0,\hdots,0) \succeq 0$, and $C$ is a constant with respect to $\bmU$. The objective functions in \cite[Equation 4]{supp:pham:hal-00371941} and \cite[Equation 8]{supp:bouchard_ajd2019} bear great similarity to \cref{eq:hppca_frobnorm}.
\edit{However, in AJD, the diagonal matrices $\bmW_\ell$ depend on $\bmU$ and are optimization variables, whereas in \cref{eq:hppca_frobnorm} they are considered fixed and known \textit{a priori}.}
Accordingly, problems \cref{eq:origobj} and \cref{eq:hppca_frobnorm} can be loosely interpreted as finding the $\bmU$ that best approximately jointly diagonalizes the data second-order statistics $\bmA_\ell$ to each $\bmW_\ell$. The AJD literature often employs Riemannian manifold optimization to solve the chosen objective function iteratively. To the best of our knowledge, no work has yet shown an analytical solution beyond the case when all the matrices commute nor proven global optimality criteria for these nonconvex programs.

\sloppypar The works in \cite{supp:boumal_nips2016, supp:Pumir2018SmoothedAO} \textcolor{black}{show nonconvex Burer–Monteiro factorizations \cite{supp:burer2003nonlinear} to solve low-rank SDPs have no spurious local minima and that approximate second-order stationary points are approximate global optima,} but these are distinct from our problem in which the columns of the orthonormal basis are constrained together in \eqref{eq:ncvxprimal}. Other works have studied optimizers to the nonconvex problem, like those in \cite{supp:breloyMMStiefel2021,supp:breloyRobustCovarianceHetero2016,supp:SunBreloy2016,supp:BreloyClutterSubspace2015}, using minorize-maximize or Riemannian gradient ascent algorithms. While efficient and scalable, these methods do not have global optimality guarantees beyond proof of convergence to a critical point.
Recent works have also studied convex relaxations of PCA and other low-rank subspace problems that bound the eigenvalues of a single matrix  \cite{supp:vu2013fantope,supp:Morgenstern2019, supp:won2021orthogonalSiamJMAA}, rather than the sum of multiple matrices as in our setting.
\cite{supp:won2021orthogonalSiamJMAA, supp:won2022orthogonal} study the SDP relaxation of maximizing the sum of traces of matrix quadratic forms on a product of Stiefel manifolds using the Fantope and propose a global optimality certificate. We emphasize their problem pertains to optimizing a trace sum over multiple orthonormal bases, each on a different Stiefel manifold, whereas our problem optimizes over the columns of a single basis on the Stiefel and is distinct from theirs. Extending the theory of the dual certificate from \cite{supp:kyfan} to the orthogonal trace maximization problem, the work in \cite{supp:won2021orthogonalSiamJMAA} proposes a simple way to test the global optimality of a given stationary point from an iterative solver of the nonconvex problem. Then in \cite{supp:won2021orthogonalSiamJMAA}, the same authors prove that for an additive noise model with small noise, their SDP relaxation is tight, and the solution of the nonconvex problem is globally optimal with high probability.

{
Many works study SDP relaxations of low-rank problems without Fantope constraints, a few of which we highlight here. The works in \cite{supp:journee2010low,supp:boumal2020deterministic,supp:bandeira2017tightness} \textcolor{black}{study the use of Burer-Monteiro factorizations to solve SDPs for}
optimization problems with multiple linear constraints. From the local properties of candidate solutions, they devise dual certificates to check for global optimality. \cite{supp:zhou2021conditions,supp:burer2005local} show for low-rank SDPs with rank-$r$ and $m$ linear constraints, no spurious local minima exist if $(r+1)(r+2)/2 > m + 1$; \cite{supp:burer2005local} also proves convergence of the nonconvex Burer-Monteiro factorization to the optimal SDP solution, with \cite{supp:cifuentes2022polynomial} strengthening this result, showing such algorithms converge provably in polynomial time, given that $r \gtrsim \sqrt{2(1 + \eta)m}$ for any fixed constant $\eta > 0$.

Similar to our work, the authors in \cite{supp:podosinnikova2019overcomplete} seek to recover multiple rank-one matrices, but for the overcomplete ICA problem. They solve separate SDP relaxations for each atom of the dictionary, using a deflation method to find the atoms in succession. In contrast, our work estimates all of the rank-one matrices simultaneously and requires that their first principal components form an orthonormal basis, whereas the dictionary atoms in ICA are only constrained to be unit-norm. 

}
\section{Supplement to \texorpdfstring{\cref{thm:dual_certificate}}{Theorem 4.1}}
\subsection{Derivation of \texorpdfstring{\cref{eq:dual_problem}}{(SDP-D}}
\label{appendix:dual:derivation}

The Lagrangian function of \cref{eq:primal_problem}, with dual variables $\bmnu \in \R^k, \bmY \in \mathbb{S}^d_+, \bmZ_i \in \mathbb{S}^d_+$ for $i=1,\dots,k$, is
\begin{align}
    &\clL(\bmX_i, \bmnu, \bmY, \bmZ_i) = \\ &-\tr\left(\sum_{i=1}^k \bmM_i \bmX_i\right) - \sum_{i=1}^k \nu_i(1 - \tr(\bmX_i)) - \tr\left(\bmY\left(\bmI - \sum_{i=1}^k \bmX_i \right)\right) - \sum_{i=1}^k \tr(\bmZ_i \bmX_i), \nonumber
\end{align}
for which the dual function is
\begin{align}
\begin{split}
    g(\bmY, \bmZ_i, \bmnu) &= \inf_{\bmX_i} \clL(\bmX_i, \bmnu, \bmY, \bmZ_i) \\
    &= \begin{cases} -\tr(\bmY) - \sum_{i=1}^k \nu_i & \text{s.t. }\bmY = \bmM_i + \bmZ_i - \nu_i \bmI \quad \forall i\in[k] \\
    -\infty & \text{otherwise.}\end{cases}
\end{split}
\end{align}
This yields the dual problem
\begin{align}
    \max_{\bmY,\bmZ_i,\bmnu} g(\bmY, \bmZ_i, \bmnu) \quad \text{s.t.}~~\bmY \succeq 0,~~ \bmZ_i \succeq 0,~~\bmY = \bmM_i + \bmZ_i - \nu_i \bmI, \quad \forall i\in[k].
\end{align}

\begin{lemma}
    \label{lem:stationary_conditions}
    Let $\clF(\bmU)$ denote the objective function with respect to $\bmU$ in \cref{eq:sum_heterogeneous_quadratics} over $\mathrm{St}(k,d)$. If a point $\bmU \in \mathrm{St}(k,d)$ is a {second-order stationary point (SOSP)} of $\clF$, then
    \begin{align*}
     \bmLambda &= \bmU'\sum_{i=1}^k \bmM_i \bmU \bmE_i ~~\text{is symmetric},~\text{and} \\
    -\sum_{i=1}^k \langle  \dot{\bmU}, \bmM_i \dot{\bmU} \bmE_i \rangle &+ \langle \dot{\bmU} \bmLambda, \dot{\bmU} \rangle \geq 0 \quad \forall \dot{\bmU} \in \bbR^{d \times k} ~\text{such that}~ \dot{\bmU}' \bmU + \bmU' \dot{\bmU} = 0.
\end{align*}
\end{lemma}

\begin{proof}
    Taking $\bar \clF$ to be the quadratic function in \cref{eq:sum_heterogeneous_quadratics} {(scaled by $\frac{1}{2}$)} over Euclidean space, the Euclidean gradient $\bm{\nabla} \bar \clF(\bmU) =  \sum_{i=1}^k \bmM_i \bmU \bmE_i = \begin{bmatrix}\bmM_1 \bmu_1 & \cdots & \bmM_k \bmu_k \end{bmatrix}$, where $\bmE_i := \bme_i \bme_i' \in \bbR^{k \times k}$ and $\bme_i$ is the $i^{th}$ standard basis vector in $\bbR^{k}$. The Euclidean Hessian can also easily be derived as $\bm{\nabla}^2 \bar \clF(\bmU)[\dot{\bmU}] = \sum_{i=1}^k \bmM_i \dot{\bmU} \bmE_i$. Restricting $\bar \clF$ to the Stiefel manifold, let $\clF := \bar \clF |_{\mathrm{St}(k,d)}$. If $\bmU \in \mathrm{St}(k,d)$ is a {SOSP} of \cref{eq:sum_heterogeneous_quadratics}, then
    \begin{align}
        \bm{\nabla} \clF(\bmU) = 0 \quad \text{and} \quad \bm{\nabla}^2 \clF(\bmU) \preceq 0,
    \end{align}
    where $\bm{\nabla} \clF$ and $\bm{\nabla}^2 \clF$ denote the Riemannian gradient and Hessian of $\clF$, respectively. 
    
    Let $\mathrm{sym}(\bmA) = \frac{1}{2}(\bmA + \bmA')$, $\mathrm{skew}(\bmA) = \frac{1}{2}(\bmA - \bmA')$, and $[\bmA, \bmB] = \bmA \bmB - \bmB \bmA$. From \cite{supp:AbsMahSep2008}, the gradient on the manifold for {SOSPs} $\bmU$ satisfies
    \begin{align}
        \bm{\nabla} \clF = \bm{\nabla} \bar \clF - \bmU \mathrm{sym}(\bmU' \bm{\nabla} \bar \clF(\bmU)) &= (\bmI - \bmU \bmU') \bm{\nabla} \bar \clF + \bmU \mathrm{skew}(\bmU' \bm{\nabla} \bar \clF) ,\\
        &= (\bmI - \bmU \bmU') \sum_{i=1}^k \bmM_i \bmU \bmE_i + {\frac{1}{2}} \bmU \sum_{i=1}^k [\bmU' \bmM_i \bmU, \bmE_i] \label{eq:stiefel_grad}\\
        &= 0,
    \end{align}
    We note the left and right expressions of the Riemannian gradient in \cref{eq:stiefel_grad} lie in the orthogonal complement of $\mathrm{Span}(\bmU)$ and the $\mathrm{Span}(\bmU)$, respectively, so $\bm{\nabla} \clF$ vanishes if and only if $(\bmI - \bmU \bmU')\bm{\nabla} \bar \clF = 0$, and $\sum_{i=1}^k [\bmU' \bmM_i \bmU, \bmE_i]  = 0$, implying $\bmU' \bm{\nabla} \bar \clF = \bm{\nabla} \bar \clF' \bmU$. Letting  $\bmLambda := \mathrm{sym}(\bmU'\bm{\nabla} \bar \clF)$, this also implies 
    \begin{align}
        \bmU \bmLambda = \bm{\nabla} \bar \clF = \sum_{i=1}^k \bmM_i \bmU \bmE_i,
    \end{align}
    and multiplying both sides by $\bmU'$ yields the expression for $\bmLambda$, which is symmetric as shown above so we can drop the $\mathrm{sym}(\cdot)$ operator.
    
    It can be shown the Riemannian Hessian is negative semidefinite if and only if
    \begin{align}
        \langle \dot{\bmU}, \bm{\nabla}^2 \bar \clF(\bmU)[\dot{\bmU}] - \dot{\bmU} \bmLambda \rangle \leq 0
    \end{align}
    for all $\dot{\bmU} \in T_{\bmU} \mathrm{St}(d,k)$, where $ T_{\bmU} \mathrm{St}(d,k)$ is the tangent space of the Stiefel manifold, i.e. the set $T_{\bmU} \mathrm{St}(d,k) = \{ \dot{\bmU} \in \bbR^{d \times k} : \bmU' \dot{\bmU} + \dot{\bmU}' \bmU = 0\}$. Plugging in the expressions for $\bmLambda$ and the Hessian of $\bar \clF$ yield the main result.

    \vspace{-\belowdisplayskip}\[\]
\end{proof}

The following lemma is adapted from \cite[Corollary 4.2]{supp:bolla:98}
\begin{lemma}
\label{lem:lambda:psd}
    Let $\Ub \in \bbR^{d \times k}$ be a {second-order stationary point} of \cref{eq:sum_heterogeneous_quadratics} and {$\bmM_i \succeq 0$ for all $i \in [k]$}. Then $\Lambdab = \Ub' \sum_{i=1}^k \bmM_i \Ub \bmE_i$ is positive semidefinite. 
\end{lemma}
\begin{proof}
    Since $k < d$, there exists a unit vector $\bmz$ in the span of $\Ub_\perp \in \bbR^{d \times d-k}$ where $\Ub'\Ub_\perp = 0$. Let $\bma = [a_1,\hdots,a_k]' \in \bbR^k$ be an arbitrary nonzero vector. Let $\dot{\bmU} := \bmz \bma'$, and let $\Lambdab = \sum_{i=1}^k \Ub' \bmM_i \Ub \bmE_i$. Then clearly $\Ub' \dot{\bmU} = 0$, and $\Ub' \dot{\bmU} + \dot{\bmU}' \Ub = 0$, so the second-order stationary necessary condition in \cref{lem:stationary_conditions} applies:
    \begin{align}
        \bma' \Lambdab \bma = \langle \dot{\bmU} \Lambdab, \dot{\bmU} \rangle \geq \sum_{i=1}^k \langle  \dot{\bmU}, \bmM_i \dot{\bmU} \bmE_i \rangle = \sum_{i=1}^k (a_i)^2 \bmz' \bmM_i \bmz \geq 0.
    \end{align}
    Therefore, since $\bma' \Lambdab \bma \geq 0$ for arbitrary $\bma$, $\Lambdab$ is positive semidefinite.
\end{proof}
\subsection{Remarks on \texorpdfstring{\cref{thm:dual_certificate}}{Theorem 4.1}}
\label{appendix:dual_certificate_remarks}
{One may ask if there is an analytical way to verify the {dual variables $\Yb$ and $\Zb_i$} are PSD without computing the LMI feasibility problem in \cref{eq:dual_certificate}.} While it is possible to derive sufficient upper bounds on the feasible $\nubi_i$ {to guarantee $\Lambdab \succeq \D_{\nub}$ so that $\Yb \succeq 0$}, this is insufficient to certify {$\Zb_i \succeq 0$} based on these bounds alone. This is in contrast to \cite{supp:won2021orthogonalSiamJMAA}; their particular dual certificate matrix is monotone in the Lagrange multipliers (analogous to our $\nubi_i$), so it is sufficient to test the positive semidefiniteness of the certificate matrix using the analytical upper bounds. Let $\Ub_{\perp_i}$ denote an orthonormal basis for $\mathrm{Span}(\bmI - \ub_i \ub_i')$. Here, since $ \Zb_i = \Ub \,\Lambdab \, \Ub' - \sum_{j \neq i}^k \nubi_j \ub_j \ub_j' + \nubi_i \Ub_{\perp_i} \Ub_{\perp_i}' - \bmM_i,$ each $\Zb_i$ is monotone in $\nubi_i$ but not in $\nubi_j$ for $j \neq i$. Therefore, there is tension between inflating $\nubi_i$ and guaranteeing all the $\Zb_i$ are PSD. {As such, an analytical solution to check that {$\Lambdab \succeq \D_{\nub}$ and} the $\Zb_i$ are PSD remains unknown, requiring computation of the LMI feasibility problem in \cref{eq:dual_certificate}.}
\subsection{More details on arithmetic complexity} 
\label{appendix:complexity}

While SDP relaxations of nonconvex optimization problems can provide strong provable guarantees, their practicality is limited by the time and space required to solve them, particularly when using off-the-shelf interior-point solvers. Interior-point methods are provably polynomial-time, but in our case the number of floating point operations to solve \cref{eq:primal_problem} grows as $\mathcal{O}(d^3)$ \cite{supp:bentalnem}, which practically limits $d$ to be in the few hundreds.

On the other hand, the study of the SDP relaxation admits improved practical tools to transfer theoretical guarantees to the nonconvex setting, i.e., to investigate when the convex relaxation is tight, and if it is, when a candidate solution of the nonconvex problem is globally optimal. In comparison to the dual problem \cref{eq:dual_problem} (upon eliminating the variables $\bmZ_i$), the proposed global certificate significantly reduces the number of variables from $\mathcal{O}(d^2)$ to merely $k$ variables. Precisely, the total computational savings can be shown using \cite[Section 6.6.3]{supp:benTal_aharon_nemirovski2001}, for which  \cref{eq:dual_problem} scales in arithmetic complexity as $\mathcal{O}((kd)^{1/2} k d^6)$ floating point operations (flops) and the certificate scales by $\mathcal{O}((kd)^{1/2} k^2 d^3)$ flops, showing a substantial reduction by a factor of $\mathcal{O}(d^3/k)$ flops. Subsequently, an MM solver in \cite{supp:breloyMMStiefel2021} with a linear majorizer, whose cost is $\mathcal{O}(dk^2 + k^3)$ {per iteration}, combined with our global optimality certificate, is an obvious preference to solving the full SDP in \cref{eq:primal_problem} for large problems. {Given the global certificate tool in \cref{thm:dual_certificate}, if \cref{eq:sum_heterogeneous_quadratics} has a tight convex relaxation, we can reliably and cheaply certify the terminal output of a first-order solver with possibly fewer restarts and without resorting to heuristics in nonconvex optimization, which commonly entail computing many multiple algorithm runs from different initializations and taking the solution with the best objective value.}

\section{Intermediary results for \texorpdfstring{\cref{thm:main_continuity}}{Theorem 4.7} and its corollaries}

\subsection{Proof of Lemma \ref{lem:dualsetconv}}
\label{appendix:lemmaC1:proof}

\begin{proof}
For notational convenience, define $L^k := \mathrm{Feas}(D;\bmc^k)
\cap \mathrm{Extra}(D)$ and $\overline L := \mathrm{Feas}(D;\cb) \cap
\mathrm{Extra}(D)$. Note that $L^k$ and $\overline L$ are bounded with interior
by Assumption \ref{ass:extra}. We wish to show $\{ L^k \} \to \overline L$.

We first note that any sequence $\{ \bmy^k \in L^k \}$ must be bounded.
If not, then $\{ \Delta \bmy^k := \bmy^k / \|\bmy^k\| \}$ is a bounded sequence
satisfying
\[
    \| \Delta \bmy^k \| = 1, \quad \frac{\bmc^k}{\|\bmy^k\|} -\bmA' \Delta \bmy^k \in \K^*,
    \quad \frac{\bmf}{\|\bmy^k\|} - \bmE' \Delta \bmy^k \ge 0
\]
and hence has a limit point $\Delta \bmy$ satisfying
\[
    \Delta \bmy \ne 0, \quad -\bmA' \Delta \bmy \in \K^*, \quad -\bmE' \Delta \bmy \ge 0,
\]
but this is a contradiction by the discussion after the statement of
Assumption \ref{ass:extra}. We thus conclude that any sequence $\{\bmy^k
\in L^k\}$ has a limit point.

Appealing to the definition of the convergence of sets stated before the
lemma, we first let $\overline \bmy$ be a limit point of any $\{ \bmy^k \in
L^k \}$ and prove that $\overline \bmy \in \overline L$. Since
\[
    \bmc^k - \bmA' \bmy^k \in \K^*, \quad \bmf - \bmE' \bmy^k \ge 0
\]
for all $k$, by taking the limit of $\{ \bmc^k \}$ and $\{ \bmy^k \}$, we have
$\cb - \bmA' \overline \bmy \in \K^*$ and $\bmf - \bmE' \overline \bmy \ge 0$ so that
indeed $\overline \bmy \in \overline L$.

Next, we must show that every $\overline \bmy \in \overline L$ is the limit point of
some sequence $\{ \bmy^k \in L^k \}$. For this proof, define
\[
    \kappa(\overline \bmy) := \min \{ k : \overline \bmy \in L^\ell \ \ \forall \ \ell \ge k \},
\]
\sloppypar{\noindent i.e., $\kappa(\overline 
\bmy)$ is the smallest $k$ such that $\overline \bmy$ is a
member of every set in the tail $L^k, L^{k+1}, L^{k+2}, \ldots$. By
convention, if there exists no such $k$, we set $\kappa(\overline \bmy) =
\infty$.}

Let us first consider the case $\overline \bmy \in \mathrm{int}(\overline L)$. We claim
$\kappa(\overline \bmy) < \infty$, so that setting $\bmy^k = \overline \bmy$ for all $k \ge
\kappa(\overline \bmy)$ yields the desired sequence converging to $\overline \bmy$. Indeed,
as $\overline \bmy$ satisfies $\cb - \bmA' \overline \bmy \in \mathrm{int}(\K^*)$ and
$\bmf - \bmE' \overline \bmy > 0$, the equation
\[
    \bmc^k - \bmA' \overline \bmy = \left( \cb - \bmA' \overline \bmy \right) + \left( \bmc^k - \cb \right)
\]
shows that $\{ \bmc^k - \bmA' \overline \bmy \}$ equals $\cb - \bmA' \overline \bmy \in
\mathrm{int}(\K^*)$ plus the vanishing sequence $\{ \bmc^k - \cb \}$.
Hence its tail is contained in $\mathrm{int}(\K^*)$, thus proving
$\kappa(\overline \bmy) < \infty$, as desired.

Now we consider the case $\overline \bmy \in \mathrm{bd}(\overline L)$. Let $\bmy^0 \in
\mathrm{int}(\overline L)$ be arbitrary, so that $\kappa(\bmy^0) < \infty$ by
the previous paragraph. For a second index $\ell = 1,2,\ldots$, define
$\bmz^\ell := (1/\ell) \bmy^0 + (1 - 1/\ell) \overline \bmy \in \mathrm{int}(\overline L)$.
Clearly, $\kappa(\bmz^\ell) < \infty$ for all $\ell$ and $\{\bmz^\ell\} \to
\overline \bmy$. We then construct the desired sequence $\{ \bmy^k \in L^k \}$
converging to $\overline \bmy$ as follows. First, set
\begin{align*}
    k_1 &:= \kappa(\bmz^1) = \kappa(\bmy^0) \\
    k_\ell &:= \max \{ k_{\ell - 1} + 1, \kappa(\bmz^\ell) \} \ \ \forall \ \ell = 2, 3, \ldots
\end{align*}
and then, for all $\ell$ and for all $k \in [k_\ell, k_{\ell + 1} - 1]$, define
$\bmy^k := \bmz^\ell$. Essentially, $\{ \bmy^k \}$ is the sequence $\{ \bmz^\ell \}$, except
with entries repeated to ensure $\bmy^k$ is in fact a member of $L^k$ for all $k$.
Hence, $\{ \bmy^k \}$ converges to $\overline \bmy$ as desired.
\end{proof}
\subsection{Proof of Lemma \ref{lem:optsetconv}}
\label{appendix:lemmaE2:proof}
\subsubsection{Setup}

Let linearly independent matrices $\bmA_1, \ldots, \bmA_m \in \SSS^n$ be
given, and define the linear function $\mathcal{A} : \SSS^n \to \mathbb{R}^m$ by
\[
    \mathcal{A}(\bmX) = \begin{pmatrix} \bmA_1 \bullet \bmX \\ \vdots \\ \bmA_m \bullet \bmX \end{pmatrix}.
\]
We consider the following family of spectrahedra parameterized by $\bmb \in
\mathbb{R}^m$:
\begin{equation*}
    \Feas(\bmb) := \{ \bmX \succeq 0 : \mathcal{A}(\bmX) = \bmb \}.
\end{equation*}
Specifically, given a convergent sequence $\{ \bmb^k \} \to \bb$, we
wish to understand conditions guaranteeing that $\{ \Feas(\bmb^k) \}$
converges to $\Feas(\bb)$. (Convergence of sets is defined precisely
in the paragraph after next.)

For simplicity, we assume that all sets $\{ \Feas(\bmb^k) \}$ and
$\Feas(\bb)$ are bounded with interior, i.e., each contains a
feasible point satisfying $\bmX \succ 0$ and the recession cone $\{ \Delta
\bmX \succeq 0 : \mathcal{A}(\Delta \bmX) = 0 \}$, which is common to all $\Feas(\bmb)$,
is trivial. Topologically speaking, the set $\{ \bmX \succ 0 : \mathcal{A}(\bmX) = \bmb
\}$ is the relative interior of $\Feas(\bmb)$.

We use the following definition of a convergent sequence of bounded
sets: a sequence of bounded sets $\{ \Lc^k \}$ converges to a bounded
set $\overline \Lc$, written $\{ \Lc^k \} \to \overline \Lc$, if and only if: (i)
given any sequence $\{ \bmX^k \in \Lc^k \}$, every limit point $\Xb$
of the sequence satisfies $\Xb \in \overline \Lc$; and (ii) every member
$\Xb \in \overline \Lc$ is the limit point of some sequence $\{ \bmX^k \in
\Lc^k \}$.

\subsubsection{Convergence of feasible sets}

Relative to $\{ \Feas(\bmb^k) \}$ and $\Feas(\bb)$, we define
$\Lc(\{\Feas(\bmb^k)\})$ to be the collection of all limit points of the
sequence of sets $\{ \Feas(\bmb^k) \}$:
\[
    \Lc(\{\Feas(\bmb^k)\}) := \left\{ \Xb : \exists \ \{ \bmX^k \in \Feas(\bmb^k) \}
    \text{ s.t. } \Xb \text{ is a limit point of } \{ \bmX^k \} \right\}.
\]
\sloppypar{Then convergence $\{ \Feas(\bmb^k) \} \to \Feas(\bb)$ is equivalent to
the statement $\Lc(\{ \Feas(\bmb^k)\}) = \Feas(\bb)$. The left-to-right
containment is straightforward.}

\begin{proposition} \label{pro:lefttoright}
$\Lc(\{ \Feas(\bmb^k)\}) \subseteq \Feas(\bb)$.
\end{proposition}

\begin{proof}
Let $\Xb \in \Lc(\{ \Feas(\bmb^k)\})$. By definition, passing to a
subsequence if necessary, there exists a sequence $\{ \bmX^k \in \Feas(\bmb^k)
\}$ converging to $\Xb$. The individual feasibility systems $\bmX^k \succeq
0, \mathcal{A}(\bmX^k) =\bmb^k$ along with $\{\bmb^k \} \to \bb$ ensure that $\Xb
\succeq 0, \mathcal{A}(\Xb) = \bb$, i.e., that $\Xb \in \Feas(\bb)$, as
desired.
\end{proof}

Proving the right-to-left inclusion $\Lc(\{ \Feas(\bmb^k)\}) \supseteq
\Feas(\bb)$ is more involved. We start by showing that $\Relint(
\Feas(\bb))$ is a subset of $ \Lc(\{ \Feas(\bmb^k)\})$.

\begin{lemma}
Every sequence $\{\bmX^k \in \Feas(\bmb^k) \}$ has a limit point. In
particular, $\Lc(\{\Feas(\bmb^k) \})$ is nonempty.
\end{lemma}

\begin{proof}
We argue that $\{\bmX^k\}$ is bounded, so that it has a limit point $\Xb \in \Lc( \{ \Feas(\bmb^k) \} )$. Suppose for contradiction that the
sequence is unbounded. Then there exists a subsequence $\{\bmX^k \}$ of
feasible solutions with $\|\bmX^k \|_\mathrm{F} \to \infty$. It follows that the
normalized subsequence $\{ \Delta\bmX^k :=\bmX^k / \|\bmX^k\|_\mathrm{F} \}$ is bounded
and satisfies
\[
    \Delta\bmX^k \ge 0, \quad \mathcal{A}( \Delta\bmX^k) =\bmb^k / \|\bmX^k\|_\mathrm{F}, \quad
    \|\Delta\bmX^k \|_\mathrm{F} = 1.
\]
Hence, there exists a limit point $\overline{\Delta \bmX}$
satisfying $\overline{\Delta \bmX} \ge 0, \mathcal{A} ( \overline{\Delta \bmX} ) = 0,
\|\overline{\Delta \bmX} \|_\mathrm{F} = 1$. However, this contradicts the assumption
that the recession cone is trivial.
\end{proof}

\begin{proposition} \label{pro:partialconverse}
$\Relint(\Feas(\bb)) \subseteq \Lc( \{ \Feas(\bmb^k) \} )$.
\end{proposition}

\begin{proof}
For notational convenience, define $\Lc := \Lc( \{ \Feas(\bmb^k) \} )$. Let
$\Xb \in \Relint(\Feas(\bb))$ be given, i.e., $\Xb$ satisfies
$\Xb \succ 0$ and $\mathcal{A}(\Xb) = \bb$. We will show $\Xb \in \Lc$
by ``bootstrapping'' it from an arbitrary $\hat \bmX \in \Lc$. Note that
$\Lc \ne \emptyset$ by the lemma---so that $\hat \bmX$ exists---and that
$\hat \bmX \in \Feas(\bb)$ by Proposition \ref{pro:lefttoright}. By
definition, passing to a subsequence if necessary, there exists $\{\bmX^k
\in \Feas(\bmb^k) \} \to \hat \bmX$.

Define $\Delta \bmX := \Xb - \hat \bmX$. We claim that $\{\bmX^k + \Delta
\bmX \}$, which clearly converges to $\hat \bmX + \Delta \bmX = \Xb$,
establishes $\Xb \in \Lc$. It remains to verify $\bmX^k + \Delta \bmX \in
\Feas(\bmb^k)$ for large $k$. Since $\mathcal{A}(\Delta \bmX) = \mathcal{A}(\Xb - \hat \bmX) =
\bb - \bb = 0$, it holds that $\mathcal{A}(\bmX^k + \Delta \bmX) =\bmb^k + 0 =\bmb^k$
for all $k$. Moreover, since $\Xb \succ 0$, the tail of $\{\bmX^k + \Delta
\bmX \}$ must eventually satisfy $\bmX^k + \Delta \bmX \succ 0$, as desired.
\end{proof}

\noindent We remark that Propositions \ref{pro:lefttoright} and
\ref{pro:partialconverse} together show $\Relint(\Feas(\bb))
\subseteq \Lc( \{ \Feas(\bmb^k) \} ) \subseteq \Feas(\bb)$. If $\Lc(
\{ \Feas(\bmb^k) \} )$ were a closed set, then we would have the desired
result that $\Lc( \{ \Feas(\bmb^k) \} ) = \Feas(\bb)$. However, we do
not have a direct proof that it is closed.

Next, we show that every extreme point of $\Feas(\bb)$ is a member
of $\Lc(\{ \Feas(\bmb^k)\})$ with a special property. The notation
$\Ext(\Feas(\bmb))$ indicates the set of extreme points of $\Feas(\bmb)$ for a
given $\bmb$.

\begin{proposition} \label{pro:extreme}
Let $\Xb \in \Ext(\Feas(\bb))$. Then there exists a full sequence
$\{\bmX^k \in \Feas(\bmb^k) \}$, not just a subsequence, converging to $\Xb$. In particular, $\Xb \in \Lc(\{ \Feas(\bmb^k)\})$.
\end{proposition}

\noindent To prove the proposition, we recall that there exists $\Cb$
such that $\Xb$ is the unique optimal solution of
\[
    v(\bb, \Cb) := \min \{ \Cb \bullet \bmX : \bmX \in \Feas(\bb) \}.
\]
We also define
\[
    v(\bmb^k, \Cb) := \min \{ \Cb \bullet \bmX : \bmX \in \Feas(\bmb^k) \}.
\]

\begin{lemma}
$\{ v(\bmb^k, \Cb) \} \to v(\bb, \Cb)$.
\end{lemma}

\begin{proof}
Note that $\{ v(\bmb^k, \Cb) \}$ is bounded. If not, then there exists
an unbounded sequence $\{\bmX^k \in \Opt(\bmb^k, \Cb) \}$ of optimal
solutions such that $\Cb \bullet\bmX^k \to -\infty$ with $\|\bmX^k\|_\mathrm{F} \to
\infty$. As in the proof of the above lemma, this contradicts that the
recession cone is trivial. So in fact $\{ v(\bmb^k, \Cb) \}$ is bounded.

Then, to prove the result, let $\hat v$ be an arbitrary limit point of
$\{ v(\bmb^k, \Cb) \}$. We will show $\hat v = v(\bb, \Cb)$ using
Propositions \ref{pro:lefttoright} and \ref{pro:partialconverse}.

First, let $\{\bmX^k \in \Opt(\bmb^k, \Cb) \}$ be a subsequence of optimal
solutions such that $\Cb \bullet\bmX^k = v(\bmb^k, \Cb) \to \hat v$.
Passing to another subsequence if necessary, $\{\bmX^k \}$ converges to
some $\hat \bmX \in \Lc( \{ \Feas(\bmb^k) \} ) \subseteq \Feas(\bb)$ by
Proposition \ref{pro:lefttoright}. Hence, $\hat v = \Cb \bullet \hat
\bmX \ge v(\bb, \Cb)$. Next, let $\epsilon > 0$ be fixed, and take
$\Xb_\epsilon \in \Relint(\Feas(\bb))$ with $\Cb \bullet \Xb_\epsilon \le v(\bb, \Cb) + \epsilon$. Since $\Xb_\epsilon
\in \Lc( \{ \Feas(\bmb^k) \} )$ by Proposition \ref{pro:partialconverse},
there exists $\{\bmX^k \in \Feas(\bmb^k) \} \to \Xb_\epsilon$. It follows
that $$v(\bmb^k, \Cb) \le \Cb \bullet\bmX^k \to \Cb \bullet \Xb_\epsilon \le v(\bb, \Cb) + \epsilon,$$ which proves $\hat v \le
v(\bb, \Cb) + \epsilon$.

Summarizing, for every fixed $\epsilon > 0$, we have $\hat v \le v(\bb, \Cb) + \epsilon \le \hat v + \epsilon$. Hence, $\hat v \le v(\bb, \Cb) \le \hat v$, as desired.
\end{proof}

\noindent Using this lemma, we can now prove Proposition
\ref{pro:extreme}.

\begin{proof}
For all $k$, let $\bmX^k$ be an arbitrary solution of the system
\[
    \Cb \bullet \bmX = v(\bmb^k, \Cb), \quad
    \mathcal{A}(\bmX) =\bmb^k, \quad \bmX \succeq 0,
\]
i.e., an optimal solution of the $k$-th optimization. Then $\{\bmX^k \}$
is bounded, and every limit point must be a solution of
\[
    \Cb \bullet \bmX = v(\bb, \Cb), \quad
    \mathcal{A}(\bmX) = \bb, \quad \bmX \succeq 0,
\]
i.e., must equal $\Xb$. Hence, $\{\bmX^k \}$ converges to $\Xb$.
\end{proof}

As a corollary, we now have our main result in this subsection.

\begin{corollary} \label{cor:main}
$\Lc( \{ \Feas(\bmb^k) \} ) = \Feas(\bb)$, i.e., $\{ \Feas(\bmb^k) \}$
converges to $\Feas(\bb)$.
\end{corollary}

\begin{proof}
Since every point in $\Feas(\bb)$ is a convex combination of extreme
points, we can simply take the same convex combination of full sequences
converging to the extreme points to show that each $\bmX \in \Feas(\bb)$ is
also a member of $\Lc( \{ \Feas(\bmb^k) \})$.
\end{proof}

\subsubsection{Convergence of optimal sets}

%

Now let $\Cb$ be an arbitrary objective matrix. For any $\bmb$, we
introduce the notation
\begin{align*}
    v(\bmb, \Cb) &:= \min \{ \Cb \bullet \bmX : \bmX \in \Feas(\bmb) \}, \\
    \Opt(\bmb, \Cb) &:= \{ \bmX \in \Feas(\bmb) : \Cb \bullet \bmX = v(\bmb, \Cb) \}
\end{align*}
and ask: when does $\{ \Opt(\bmb^k, \Cb) \}$ converge to $\Opt(\bb,
\Cb)$? As with the above lemma, we have that $\{ v(\bmb^k, \Cb) \}
\to v(\bb, \Cb)$. In analogy with the previous subsection, we also
define
\[
    \Lc(\{\Opt(\bmb^k,\Cb)\}) := \left\{ \Xb : \exists \ \{\bmX^k \in \Opt(\bmb^k,\Cb) \}
    \text{ s.t. } \Xb \text{ is a limit point of } \{\bmX^k \} \right\}
\]
We immediately have a result, which is analogous to Proposition
\ref{pro:lefttoright}.

\begin{proposition} \label{pro:lefttoright2}
$\Lc( \{ \Opt(\bmb^k, \Cb) \}) \subseteq \Opt(\bb, \Cb)$.
\end{proposition}

\begin{proof}
The proof is similar to the proof of Proposition \ref{pro:lefttoright}
except we conceptually replace
\[
    \mathcal{A}(\bmX) \text{ by } \begin{pmatrix} \mathcal{A}(\bmX) \\ \Cb \bullet \bmX \end{pmatrix},
    \quad
  \bmb^k \text{ by } \begin{pmatrix}\bmb^k \\ v(\bmb^k, \Cb) \end{pmatrix},
    \quad
    \bb \text{ by } \begin{pmatrix} \bb \\ v(\bb, \Cb) \end{pmatrix}.
\]
\end{proof}

Next we would like to prove a result that is anlogous to Proposition
\ref{pro:partialconverse}, but this is more challenging because
$\Opt(\bb, \Cb)$ may not contain a positive definite
solution as did $\Feas(\bb)$ in the proof of Proposition
\ref{pro:partialconverse}. We will need an additional assumption on
$\Opt(\bb, \Cb)$.

Indeed, let $\Xb \in \Relint(\Opt(\bb, \Cb))$ with $r :=
\rank(\Xb)$ be arbitrary. Because $\Opt(\bb, \Cb)$ is a face of
$\Feas(\bb)$, it is characterized by $\Range(\Xb)$. Specifically,
let $\Xb = \Qb \, \Qb'$ be any factorization of $\Xb$ with
$\Qb \in \mathbb{R}^{n \times r}$. Then it is well-known that
\[
    \Opt(\bb, \Cb) = \left\{
    \bmX :
    \begin{array}{l}
        \mathcal{A}(\bmX) = \bmb \\
        \bmX = \Qb \bmY \Qb' \\
        \bmY \succeq 0 
    \end{array}
    \right\}
\]
and
\[
    \Relint(\Opt(\bb, \Cb)) = \left\{
    \bmX :
    \begin{array}{l}
        \mathcal{A}(\bmX) = \bmb \\
        \bmX = \Qb \bmY \Qb' \\
        \bmY \succ 0 
    \end{array}
    \right\}.
\]
Note that $\bmY$ has size $r \times r$. In particular, because $\Xb \in
\Relint(\Opt(\bb, \Cb))$, the system
\begin{align*}
    (\Qb' \bmA_i \Qb) \bullet \bmY &= \bb_i \quad\quad\quad \forall \ i=1,\ldots,m \\
    \bmY &\succeq 0 
\end{align*}
is interior feasible, where we have used properties of the trace inner
product to write $\bmA_i \bullet \bmX = \bmA_i \bullet (\Qb \bmY \Qb')
= (\Qb' \bmA_i \Qb) \bullet \bmY$. In words, the affine subspace
defined by the $m$ linear equations intersects the interior of the
full-dimensional positive semidefinite cone.

This leads us to our assumption on $\Opt(\bb, \Cb)$. We wish to
have a condition, which will guarantee that the above system remains
interior feasible even if the right-hand-side values $\bb_i$ are
perturbed a bit. A sufficient condition is that the matrices $\Qb'
\bmA_i \Qb_i$, $i = 1,\ldots,m$, are linearly independent.

\begin{proposition} \label{pro:partialconverse2}
Suppose $\{ \Qb' \bmA_i \Qb \}_{i=1}^m$ are linearly
independent. Then $\Relint(\Opt(\bb, \Cb)) \subseteq \Lc( \{
\Opt(\bmb^k, \Cb) \})$.
\end{proposition}

\begin{proof}
For notational convenience, define $\Lc := \Lc( \{ \Opt(\bmb^k, \Cb) \}
)$, and take arbitrary $\bmX^0 \in \Relint(\Opt(\bb, \Cb))$. We wish
to show $\bmX^0 \in \Lc$, that is, there exists a subsequence of points,
each a member of $\Opt(\bmb^k, \Cb)$, converging to $\bmX^0$. From the
discussion before the proposition, there exists $\bmY^0 \succ 0$ such that
$\bmX^0 = \Qb \bmY^0 \Qb'$.

To construct the desired sequence, we note from the discussion before
the proposition that the linear independence of $\{ \Qb' \bmA_i \Qb
\}$ ensures that there exists a subsequence of systems
\begin{align*}
    (\Qb' \bmA_i \Qb) \bullet \bmY &=\bmb^k_i \quad\quad\quad \forall \ i=1,\ldots,m \\
    \bmY &\succeq 0,
\end{align*}
each of which is interior feasible. Take $\{ \bmY^k \}$ to be such an
interior-feasible subsequence with a limit point $\hat \bmY$, and define
$\{\bmX^k := \Qb \bmY^k \Qb' \}$ and $\hat \bmX := \Qb \hat \bmY \Qb'$. We have $\bmX^k \in \Opt(\bmb^k, \Cb)$ converging to $\hat \bmX \in
\Opt(\bb, \Cb)$ by construction.

Given the constructed sequence $\{\bmX^k \in \Opt(\bmb^k, \Cb) \} \to
\hat \bmX \in \Opt(\bb, \Cb)$, we will now show $\bmX^0 \in \Lc$ by
``bootstrapping'' it from $\hat \bmX$. Define $\Delta \bmX := \bmX^0 - \hat
\bmX$. We claim that $\{\bmX^k + \Delta \bmX \}$, which clearly converges to
$\hat \bmX + \Delta \bmX = \bmX^0$, establishes $\bmX^0 \in \Lc$. It remains
to verify $\bmX^k + \Delta \bmX \in \Opt(\bmb^k, \Cb)$ for large $k$. Since
\[
    \Cb \bullet \Delta \bmX = \Cb \bullet (\bmX^0 - \hat \bmX) = v(\bb, \Cb)
    - v(\bb, \Cb) = 0
\]
and
\[
    \mathcal{A} (\Delta \bmX) = \mathcal{A} (\bmX^0 - \hat \bmX) = \bb - \bb = 0,
\]
it holds that
\[
    \Cb \bullet (\bmX^k + \Delta \bmX) = \Cb \bullet\bmX^k = v(\bmb^k, \Cb)
    \quad \text{ and } \quad
    \mathcal{A} (\bmX^k + \Delta \bmX) = \mathcal{A} (\bmX^k) =\bmb^k,
\]
i.e., each $\bmX^k + \Delta \bmX$ satisfies the linear constraints $\mathcal{A} (\bmX) =
\bmb$ and attains the optimal value $v(\bmb^k, \Cb)$. We still need to show
$\bmX^k + \Delta \bmX \succeq 0$ for large $k$.

To prove this, we write
\begin{align*}
   \bmX^k + \Delta \bmX
    &= \Qb \bmY^k \Qb' + \Qb (\bmY^0 - \hat \bmY) \Qb' \\
    &= \Qb (\bmY^k + \bmY^0 - \hat \bmY) \Qb'.
\end{align*}
Since $\{ \bmY^k \} \to \hat \bmY$ and $\bmY^0 \succ 0$, it follows that the tail
of $\bmX^k + \Delta \bmX$ is positive semidefinite.
\end{proof}

With Propositions \ref{pro:lefttoright2} and \ref{pro:partialconverse2}
in hand, the analogies of Proposition \ref{pro:extreme} and Corollary
\ref{cor:main} are proven in the same way.

\begin{proposition} \label{pro:extreme2}
Let $\Xb \in \Ext(\Opt(\bb))$. Then there exists a full sequence
$\{\bmX^k \in \Opt(\bmb^k) \}$, not just a subsequence, converging to $\Xb$. In particular, $\Xb \in \Lc(\{ \Opt(\bmb^k)\})$.
\end{proposition}

\begin{corollary} \label{cor:main2}
$\Lc( \{ \Opt(\bmb^k) \} ) = \Opt(\bb)$, i.e., $\{ \Opt(\bmb^k) \}$
converges to $\Opt(\bb)$.
\end{corollary}

\begin{lemma}
    Let $\mathrm{Opt}(\cb)$ be the optimal set of the dual problem \cref{eq:dual_problem} parameterized by $\cb = (\Mb_1,\hdots,\Mb_k)$ such that $\Mb_i$ for all $i \in [k]$ are jointly diagonalizable, and assume the associated LP of the \cref{eq:primal_problem} has a unique optimal solution. Then the linear independence property in \cref{pro:partialconverse2} holds.
\end{lemma}

\begin{proof}
 When $\Mb_i$ for all $i \in [k]$ are jointly diagonalizable, \cref{eq:dual_problem} reduces to a linear program:
 
 \begin{align}
    \min& ~~\bmp' \bmb \label{sdpd-std}\\
    &\text{s.t.}~~ \A \bmp = \overline{\bmm}, ~~ \bmp \geq 0, \nonumber
\end{align}
where
$\bmp := \begin{bmatrix} \bmy' & \bmz_1' & \cdots & \bmz_k' & \nu_1 & \cdots & \nu_k \end{bmatrix}'$ is the dual variables stacked into a single vector in $\mathbb{R}^{dk + d + k}$, and
\begin{align}
\bmb &:= \begin{bmatrix} \bme_d' & 0_{d}' & \hdots & 0_d' & \bme_k' \end{bmatrix}', \\
\A &:= \begin{bmatrix}  \bmI_d & -\bmI_d & 0 & \cdots & 0 & \bme & 0 & \cdots & 0 \\ \bmI_d & 0 & -\bmI_d & \cdots & 0 & 0 & \bme & \cdots & 0 \\
\vdots & \vdots & \vdots & \ddots & \vdots & \vdots & \vdots & \ddots & \vdots \\ \bmI_d & 0 & 0 & \cdots & -\bmI_d & 0 & \cdots & \cdots & \bme\end{bmatrix}, \quad \bmm := \begin{bmatrix} \bmm_1 \\ \bmm_2 \\ \vdots \\ \bmm_k \end{bmatrix}.
\end{align}

 Reexpressing the linear program as an SDP using nonnegative diagonal matrices with $\bmy$, $\bmz_i$ and $\bmm_i$ along the diagonals, the equivalent dual problem is
\begin{align}
\begin{split}
    \min_{\bmX \succeq 0}& ~~ \langle \overline{\bmC}, \bmX \rangle \\
    &\text{s.t.}~\langle \bmA_i, \bmX \rangle = \overline{m}_i \quad \forall i \in [dk]
\end{split}
\end{align}
where $\bmX = \mathrm{diag}(\begin{bmatrix}\bmy' & \bmz_1' & \cdots & \bmz_k' & \nu_1 & \cdots & \nu_k \end{bmatrix}')$ and $\overline{m}_i$ is the $i^{th}$ entry of the vector formed by concatenating the diagonalized data matrices, and $\overline{\bmC}:=\mathrm{diag}(\bmb)$. The linear constraints parameterized by $\bmA_i := \mathrm{diag}(\bmA_{i,:})$, where $\bmA_{i,:}$ is the $i^{th}$ row of $\A$, capture the equalities $\bmy = \bmm_i + \bmz_i - \nu_i \bme_d$.

Let $\Xb \in \mathrm{Relint}(\mathrm{Opt}(\overline{\bmm}, \overline{\bmC}))$ and $(\bmy, \bmz_1, \hdots, \bmz_k,\nu_1,\hdots,\nu_k)$ be the optimal solution to the dual LP, where $\bmy = \diag(\bmY)$ and $\bmz_i = \diag(\bmZ_i)$ are the vectors extracted from the diagonal matrices, and $\diag(\bmy) = \bmY$ and $\diag(\bmz_i) = \bmZ_i$ are diagonal matrices.

From \cref{lem:gt}, the unique optimal solution to the assignment LP has the property that each $\diag(\bmX_i)$ is a standard basis vector, and the associated dual variables $\diag(\bmZ_i)$ are rank $d-1$. Combined with the the KKT complementarity condition $\bmX_i \bmZ_i = 0$, then each $\diag(\bmZ_i)_j = 0$ for the single $j \in [d]$ where $\diag(\bmX_i)_j = 1$. A similar result using the Goldman-Tucker strict complementarity theorem for LP holds for $\diag(\bmY)$ and $\diag(\bmI - \sum_{i=1}^k \bmX_i)$: there exists an optimal primal-dual pair such that $\diag(\bmI - \sum_{i=1}^k \bmX_i) + \diag(\bmY) > 0$. Hence, there exists a dual optimal solution with $\rank(\bmY) \ge k$. From KKT complementarity $(\bmI - \sum_{i=1}^k \bmX_i)\bmY = 0$, we have necessarily that $\rank(\bmY) = k$, and $\diag(\bmY)_j > 0$ for all $j \in [d]$ such that $\sum_{i=1}^k \diag(\bmX_i)_j = 1$, and zero on the remaining $d-k$ coordinates. Therefore, $\diag(\bmY)_j > 0 $ for all $i \in [k], j \in [d]$ such that $\diag(\bmZ_i)_j = 0$, and zero elsewhere. 

Then
\begin{align}
    \mathrm{rank}(\Xb) \leq \mathrm{nnz}(\diag(\bmY)) + \sum_{i=1}^k \mathrm{nnz}(\diag(\bmZ_i)) + k = k(d+1),
\end{align}
where an additional $k$ nonzeros are possible from the $\nu_i$'s.
Then there exists a $\Qb \in \bbR^{(dk + d + k) \times r}$ for $dk \leq r \leq dk + k$ such that $\Xb = \Qb\,\Qb'$. Let $\Omega \subset \{1,\hdots,dk + d + k\}$, where $|\Omega| = r$, denote the set of nonzero entries on the diagonal of $\Xb$.

Let $\Qb = \Xb_\Omega^{1/2}$, where $\Xb_{\Omega}$ denotes the submatrix restriction of $\Xb$ to columns with nonzero entries. Without loss of generality \edit{by \cref{lem:nuneg}}, assume $\nu_i > 0$ for all $i \in [k]$. Expressing $\{\Qb' \bmA_i \Qb\}_{i=1}^{dk}$ as a linear system of equations over the indices in $\Omega$,
\begin{align}
\label{eq:duallp:Aq}
    \bmA_{\Qb} &:= \begin{bmatrix}  \bmA_{\bmy} & \bmA_{\bmz_1} & 0 & \cdots & 0 & \nu_1 \bme & 0 & \cdots & 0 \\ \bmA_{\bmy} & 0 & \bmA_{\bmz_2} & \cdots & 0 & 0 & \nu_2 \bme & \cdots & 0 \\
\vdots & \vdots & \vdots & \ddots & \vdots & \vdots & \vdots & \ddots & \vdots \\ \bmA_{\bmy} & 0 & 0 & \cdots & \bmA_{\bmz_k} & 0 & \cdots & \cdots & \nu_k\bme\end{bmatrix}.
\end{align}

Above, $\bmA_{\bmy}$ denotes the diagonal matrix $\diag(\bmy)$ restricted to its $k$ columns with nonzero entries, and similarly each $\bmA_{\bmz_i}$ denotes the diagonal matrix $-\diag(\bmz_i)$ restricted to its $d-1$ columns with nonzero entries. From complementarity, the first $k + dk$ columns of $\bmA_{\overline{\bmQ}}$ contain $k + k(d-1) = dk$ linearly independent columns. Thus, the matrix has full row-rank, indicating the matrices $\{\Qb' \bmA_i \Qb\}_{i=1}^{dk}$ are linearly independent.
\end{proof}
\subsection{Supporting lemmas}

\begin{lemma}{Lin's Theorem \cite{supp:loring2016almost,supp:glashoff_bronstein13}:}
\label{lem:lins_thm}
     For all $\epsilon > 0$ there exists a $\delta > 0$ such that if $\|[\A,\B]\|_2 := \|\A \B - \B \A \|_2 \leq \delta$ for Hermitian symmetric matrices $\A$ and $\B$ where $\|\A\| \leq 1$ and $\|\B\| \leq 1$, then there exist Hermitian symmetric, commuting matrices $\widetilde \A$ and $\widetilde \B$ in $\R^{d \times d}$ such that $\|[\tilde \A, \tilde \B]\| = 0$ and $\|\A - \tilde \A\|_2 \leq \epsilon$ and $\|\B - \tilde \B\|_2 \leq \epsilon$.
\end{lemma}

\revise{\begin{lemma}{Hilbert-Schmidt analogue to Lin's Theorem \cite[Theorem 4]{supp:glebsky2010almost}, \cite[Theorem 3]{supp:Filonov2010AHA}:}
\label{lem:hilbert_schmidt_lins_thm}
     Let $(\bmA_1,\hdots,\bmA_k)$ be a tuple of self-adjoint matrices of unit spectral norm for $k \geq 3$. 
     For all $\epsilon > 0$ there exists a $\delta(\epsilon,k) > 0$, where $\delta(\epsilon,k) \rightarrow 0$ as $\epsilon \rightarrow 0$, such that if $\|[\A_i,\A_j]\|_{\mathrm{tr}} := \|\A_i \A_j - \A_j \A_i \|_{\mathrm{tr}} \leq \epsilon$, then there exist Hermitian symmetric, commuting matrices $(\widetilde \A_1,\hdots,\widetilde \A_k)$ such that $[\widetilde \A_i, \widetilde \A_j] = 0$ for $i,j \in [k]$ and $\|\A_i - \widetilde \A_i\|_{\mathrm{tr}} \leq \delta(\epsilon,k)$ for $i \in [k]$.
\end{lemma}
}

\begin{lemma}{\edit{Concentration of the sample covariance matrix for centered Gaussian random variables \cite{supp:Koltchinskii2017,supp:Lounici2014}:}}
\label{lem:sample_covar_error_bound}
Let \textcolor{black}{$\bmy_1, \hdots, \bmy_n \in \bbR^d$} be i.i.d. centered Gaussian random variables 
with covariance operator $\bmSigma$ and sample covariance $\hat \bmSigma  = \frac{1}{n} \sum_{i=1}^n \bmy_i \bmy_i'$. Then
with some constant $C > 0$ and with probability at least $1 - e^{-t}$ for $t > 0$,
\begin{align*}
    \|\hat \bmSigma - \bmSigma \| \leq C \|\bmSigma\| \max \left\{ \sqrt{\frac{\tilde{r}(\bmSigma) \log d + t}{n}} , \frac{(\tilde{r}(\bmSigma)\log d + t)\log n}{n} \right\},
\end{align*}
where $\tilde{r}(\bmSigma) := \tr(\bmSigma) / \|\bmSigma\|$.
\end{lemma}

\section{Investigations into \cref{eq:sum_heterogeneous_quadratics} and variants}

\subsection{Shor relaxation}
\label{appendix:shor_relax}

Quadratically constrained quadratic programs, such as the one we study in this paper in \cref{eq:sum_heterogeneous_quadratics}, can alternatively be relaxed using the popular Shor SDP formulation \cite{supp:cifuentes2022stability, supp:huang2009rank}.
Here, we derive this particular relaxation,
but we show it does not return tight solutions with the ROP for our problem. First, let $\M$ be the matrix with the data matrices on its block-diagonal,
\begin{equation*}
    \M = \begin{bmatrix}\M_1 & & \\ & \ddots & \\ & & \M_k\end{bmatrix}.
\end{equation*}
Rewriting the Stiefel manifold constraints in \cref{eq:sum_heterogeneous_quadratics} gives the equivalent optimization problem in the variable $\bmx = [\bmu_1' \hdots \bmu_k']' \in \bbR^{dk}$:
\begin{align}
    \max_{\bmx \in \R^{dk}}  \bmx' \M \bmx \quad \mathrm{s.t.}~~ \bmx' \bmC_{i,j} \bmx = \mathbb{I}_{i = j}\quad \forall i,j \in [k],
\end{align}
where $\mathbb{I}_{i=j}$ denotes the indicator function that is equal to 1 if $i = j$ and 0 if $i \neq j$, 
and the matrices $\bmC_{i,j} \in \mathbb{R}^{dk \times dk}$ capture the trace-1 and orthonormality constraints on the columns of $\bmU$. More precisely, $\bmC_{i,j} = \bmE_{i,j} \otimes \I_d \quad \forall i, j \in [k]$,
where $\bmE_{i,j} = \bme_i \bme_j'$ and $\bme_i$ is the $i^\mathrm{th}$ standard basis vector in $\bbR^k$.
Lifting the optimization problem in terms of the variable $\bmX = \bmx \bmx'$, we obtain the equivalent problem
\begin{align}
\begin{split}
    &\max_{\bmX \in \mathbb{S}^{dk}} \tr(\M \bmX) \quad \mathrm{s.t.}~~ \tr(\bmC_{i,j} \bmX) = \mathbb{I}_{i=j} \quad \forall i,j \in [k], \quad \mathrm{rank}(\bmX) = 1.
\end{split}
\end{align}
Relaxing the above nonconvex problem by dropping the rank constraint gives the SDP
\begin{align}
\label{eq:shor_sdp}
    \max_{\bmX \succcurlyeq 0} \tr(\M \bmX ) \quad \mathrm{s.t.}~~ \tr(\bmC_{i,j} \bmX ) = \mathbb{I}_{i=j} \quad \forall i,j \in [k].
\end{align}
\edit{We note here that to the best of our knowledge, the constraint $\sum_{i=1}^k \bmX_i \preceq \I$ cannot be captured in this framework.}

We now show the optimal solution to \cref{eq:shor_sdp} has an analytical solution and does not recover the solution to the original nonconvex problem on the Stiefel manifold. Let $\bmv_i$ be the first principal eigenvector of $\M_i$ for all $i \in [k]$, and let $\bmV_i \triangleq \bmv_i \bmv_i'$. Now let $$\bmX^* \triangleq \begin{bmatrix} \bmV_1 & & \\ & \ddots & \\ & & \bmV_k \end{bmatrix}.$$
One can check that $\bmX^*$ is a feasible solution to \cref{eq:shor_sdp}. It is also optimal with respect to \cref{eq:shor_sdp}. However, the vectors $\bmv_1, \dots, \bmv_k$ need not be mutually orthogonal, so the SDP is not tight with respect to the original problem on the Stiefel manifold.
\subsection{Counterexample for convex-hull result} \label{sec:counterexample}

The feasible set of our primal semidefinite program (\ref{eq:primal_problem}) is, by construction,  a convex relaxation of the set
\begin{equation} \label{eq:decomposedstiefel}
\left\{
(\bmu_1 \bmu_1',\ldots,\bmu_k \bmu_k')
: \bmU' \bmU = \bmI
\right\},
\end{equation}
where the $i$-th variable $\X_i$ in the SDP is a relaxed version of the rank-1 product $\bmu_i \bmu_i'$.
In this paper, we have investigated when (\ref{eq:primal_problem}) satisfies the rank-1 property (ROP), i.e., when its optimal solution $(\X_1, \ldots, \X_k)$ has $\mathrm{rank}(\X_i) = 1$ for each $i=1,\ldots,k$.

A natural question is whether the feasible set of (\ref{eq:primal_problem}) captures the convex hull of \eqref{eq:decomposedstiefel} exactly. If this were true, then (\ref{eq:primal_problem}) would satisfy the ROP for all objective coefficients $(\bmM_1,\ldots,\bmM_k)$. In this section, we show by counterexample that this is {\em not\/} the case. Note, however, that the ROP may still hold for certain subclasses of $(\bmM_1,\ldots,\bmM_k)$. Indeed, we have shown in the paper that the ROP property of \cref{eq:primal_problem} holds for jointly diagonalizable $\bmM_i$, which conforms with the theory in \cite{supp:bolla:98}.

To build our counterexample demonstrating that the feasible set of \cref{eq:primal_problem} does not exactly capture the convex hull of \cref{eq:decomposedstiefel}, we set $d = 4$ and $k = 2$. We then claim that the matrix $\bmX = \bmX_1 + \bmX_2$ given by
\[
    \bmX_1 :=
    \frac12
    \begin{pmatrix}
        1 & 0 & 0 & 0 \\
        0 & 1 & 0 & 0 \\
        0 & 0 & 0 & 0 \\
        0 & 0 & 0 & 0
    \end{pmatrix}
\]
and
\[
    \bmX_2 :=
    \frac1{12}
    \begin{pmatrix}
        3 & 1 & 3 & 1 \\
        1 & 3 & 1 & 3 \\
        3 & 1 & 3 & 1 \\
        1 & 3 & 1 & 3
    \end{pmatrix}
\]
constitutes a feasible solution of (\ref{eq:primal_problem}) but, at the same time, cannot be a strict convex combination of points in (\ref{eq:decomposedstiefel}). Said differently, we will show $(\bmX_1, \bmX_2)$ is feasible for (\ref{eq:primal_problem}) but not for the convex hull of (\ref{eq:decomposedstiefel}), thus establishing that these two sets are different. 

One can easily check that $(\bmX_1, \bmX_2)$ is feasible for (\ref{eq:primal_problem}). Note also that $\rank(\bmX_1) = \rank(\bmX_2) = 2$, so that $(\bmX_1, \bmX_2)$ itself is not an element of \cref{eq:decomposedstiefel}.  In addition, it is easy to verify that $\rank(\bmX) = 4$ and $\lambda_{\max}[\bmX] = 1$. The contrapositive of the following proposition proves that,  because $\rank(\bmX) = 4$, $(\bmX_1, \bmX_2)$ cannot be a strict convex combination of points in (\ref{eq:decomposedstiefel}). 


\begin{proposition}
Let $d \ge k = 2$ be given. Suppose $\bmX = \bmX_1 + \bmX_2$ is feasible for \cref{eq:primal_problem} such that:

\begin{itemize}

\item $(\bmX_1,\bmX_2)$ is a strict convex combination of points in \cref{eq:decomposedstiefel}, i.e., for some integer $J \ge 2$, there exist positive scalars $\lambda_1, \ldots, \lambda_J$ and Stiefel matrices
\[
    \bmU^{(j)} := \begin{pmatrix} \bmu_1^{(j)} & \bmu_2^{(j)} \end{pmatrix}
    \in \mathbb{R}^{d \times 2}
    \quad \forall \ j = 1,\ldots,J
\]
such that
\[
    (\bmX_1,\bmX_2) 
    = \sum_{j=1}^J \lambda_j \left( \bmu_1^{(j)}
    (\bmu_1^{(j)})', \bmu_2^{(j)} (\bmu_2^{(j)})'\right),
    \quad \quad \sum_{j=1}^J \lambda_j = 1;
\]

\item $\rank(\bmX_1) = \rank(\bmX_2) = 2$;

\item $\lambda_{\max}[\bmX] = 1$.

\end{itemize}

\noindent Then $\rank(\bmX) \le 3$.
\end{proposition}

\begin{proof}
For each $i \in \{1,2\}$, it holds by assumption that
\begin{align*} \label{equ:local}
    \bmX_i &= \sum_{j=1}^J \lambda_j \bmu_i^{(j)} (\bmu_i^{(j)})' \\
    &=
    \begin{pmatrix}
        \sqrt{\lambda_1} \bmu_i^{(1)} & \cdots & \sqrt{\lambda_J} \, \bmu_i^{(J)}
    \end{pmatrix}
    \begin{pmatrix}
        \sqrt{\lambda_1} \bmu_i^{(1)} & \cdots & \sqrt{\lambda_J} \, \bmu_i^{(J)}
    \end{pmatrix}'.
\end{align*}
This equation ensures, in particular, that $$\mathrm{Range}(\bmX_i) = \mathrm{Span}\{ \bmu_i^{(1)}, \ldots, \bmu_i^{(J)} \};$$ see Lemma 1 of \cite{supp:MR2555055} for example.

We claim that we can reorder the indices $\{1,\ldots,J\}$ such that $\mathrm{Range}(\bmX_i) = \mathrm{Span}(\{ \bmu_i^{(1)}, \bmu_i^{(2)} \})$ for both $i=1,2$ simultaneously. Since $\rank(\bmX_1) = 2$ by assumption, it is clear that we may reorder the indices $\{1, \ldots, J\}$ without loss of generality such that $\mathrm{Range}(\bmX_1) = \mathrm{Span}(\{ \bmu_1^{(1)}, \bmu_1^{(2)} \})$.
%
%
If $J = 2$, the claim is obvious. So suppose $J > 2$ and that the claim does not hold for the current ordering. Then we can further reorder $\{3,\ldots,J\}$ such that
\begin{align*}
    \mathrm{Range}(\bmX_1) &= \mathrm{Span}(\{ \bmu_1^{(1)}, \bmu_1^{(2)}, \bmu_1^{(3)} \})
    \quad \text{with} \quad \bmu_1^{(1)} \nparallel \bmu_1^{(2)} \\
    \mathrm{Range}(\bmX_2) &= \mathrm{Span}(\{ \bmu_2^{(1)}, \bmu_2^{(2)}, \bmu_2^{(3)} \})
    \quad \text{with} \quad \bmu_2^{(1)} \parallel \bmu_2^{(2)}
    \text{ and } \bmu_2^{(1)} \nparallel \bmu_2^{(3)}.
\end{align*}
We now consider two exhaustive subcases. First, if $\bmu_1^{(1)} \nparallel \bmu_1^{(3)}$, then we see that 
$\mathrm{Range}(\bmX_1) = \mathrm{Span}(\{ \bmu_1^{(1)}, \bmu_1^{(3)} \})$ and $\mathrm{Range}(\bmX_2) = \mathrm{Span}(\{ \bmu_2^{(1)}, \bmu_2^{(3)} \})$.
So by another reordering of $\{1,2,3\}$, the claim is proved. The second subcase $\bmu_1^{(2)} \nparallel \bmu_1^{(3)}$ follows a similar argument. 

With the claim proven that  $\mathrm{Range}(\bmX_i) = \mathrm{Span}(\{ \bmu_i^{(1)}, \bmu_i^{(2)} \})$ for both $i$, define the linear subspaces $$\bmW_i := \mathrm{Span}\{ \bmu_i^{(1)}, \bmu_i^{(2)} \} = \mathrm{Span}\{ \bmu_i^{(1)}, \ldots, \bmu_i^{(J)} \}$$ for each $i=1,2$. Then the equation $\bmX = \bmX_1 + \bmX_2$ implies
\[
    \rank(\bmX) = \dim(\bmW_1 + \bmW_2) =
    \dim(\mathrm{Span}\{ \bmu_1^{(1)}, \bmu_1^{(2)}, \bmu_2^{(1)}, \bmu_2^{(2)} \}).
\]

Next, let $\bmv$ be a maximum eigenvector of $\bmX$ with $\|\bmv\| = 1$ by definition.  Also, for each $j \in \{1, \ldots, J\}$, define $$\bmV_j := \mathrm{Span}\{ \bmu_1^{(j)}, \bmu_2^{(j)} \} = \mathrm{Range}(\bmU^{(j)}),$$ and let
\[
    \alpha_j := (\bmv' \bmu_1^{(j)})^2 + (\bmv' \bmu_2^{(j)})^2 \le 1
\]
be the squared norm of the projection of $\bmv$ onto $\bmV_j$. Because $\lambda_{\max}[\bmX] = 1$, we have
\[
    1 = \bmv' \bmX \bmv = \sum_{j=1}^J \lambda_j \left( (\bmv' \bmu_1^{(j)})^2 + (\bmv' \bmu_2^{(j)})^2 \right) = \sum_{j=1}^J \lambda_j \alpha_j.
\]
Since each $\alpha_j \le 1$ and since $\bmlambda$ is a convex combination, it follows that $\alpha_j = 1$ for all $j$, which then implies $\bmv \in \bmV_j$ for all $j \in \{1, \ldots, J\}$. In particular, $\bmv \in \bmV_1 \cap \bmV_2$.

Finally, we have $\bmW_1 + \bmW_2 = \bmV_1 + \bmV_2$ because both Minkowski sums span the four vectors $\bmu_i^{(j)}$ for $i=1,2$ and $j=1,2$. Hence,
\begin{align*}
    \rank(\bmX) &= \dim(\bmW_1 + \bmW_2) \\
    &= \dim(\bmV_1 + \bmV_2) = \dim(\bmV_1) + \dim(\bmV_2) - \dim(\bmV_1 \cap \bmV_2) \\
    &\le 2 + 2 - 1 = 3.
\end{align*}
where the inequality follows because $\bmv \in \bmV_1 \cap \bmV_2$.
\end{proof}
\subsection{Example of SDP with rank-one solutions, but \texorpdfstring{$\bmM_i$}{Mi} that are not almost commuting}
\label{appendix:example_not_almost_commuting}
In our paper, we give sufficient conditions for when the SDP returns rank-one orthogonal primal solutions in the case the $\bmM_i$ matrices almost commute. However, this is not a necessary condition, and we give a counter-example here.

\begin{proposition}
    Construct $\bmM_i$ for $i=1,\hdots,k$ as follows for given length-$d$ vectors $\bmv_i$, $i=1,\ldots,k$:
    \begin{alignat*}{3}
        \bmM_1 &= \bmv_1 \bmv_1' + && \bmv_2 \bmv_2' + \cdots + &&\bmv_k \bmv_k'\\
        \bmM_2 &= &&\bmv_2 \bmv_2' + \cdots + &&\bmv_k \bmv_k'\\
        &&\vdots\\
        \bmM_k &= && &&\bmv_k \bmv_k'\\
    \end{alignat*}
\sloppypar{\noindent such that $\bmM_1 \succeq \bmM_2 \succeq \cdots \succeq \bmM_1 \succeq 0$. Let $\{\bmu_1,\hdots, \bmu_k \}$ be an orthonormal basis for $\mathrm{Span}\{\bmv_1,\hdots, \bmv_k \}$ such that, for all $i=1,\ldots,k$, $\mathrm{Span}\{\bmu_1,\hdots,\bmu_i\} = \mathrm{Span}\{\bmv_1,\hdots,\bmv_i\}$.
    Then
    $\bmM_i$ for $i=1,\hdots,k$ need not be almost commuting, and $(\Xb_i, \hdots, \Xb_k) = (\bmu_1\bmu_1', \hdots, \bmu_k \bmu_k')$ is the optimal SDP solution with optimal value $p = \tr(\bmM_1)$.}
    
    \begin{proof}
        $\Xb_i$ are clearly feasible with objective value
        \begin{align}
           p &= \langle \bmM_1 , \bmu_1 \bmu_1' \rangle + \langle \bmM_2 , \bmu_2 \bmu_2'\rangle + \cdots + \langle \bmM_k , \bmu_k \bmu_k' \rangle \\
           &= \sum_{i=1}^k (\bmv_i'\bmu_1)^2 + \sum_{i=2}^k (\bmv_i'\bmu_2)^2 + \cdots + \sum_{i=k-1}^k (\bmv_i' \bmu_{k-1})^2 + (\bmv_k'\bmu_k)^2 \\
           &= \sum_{i=1}^k \|\bmv_i\|^2_2 = \tr(\bmM_1).
        \end{align}
        For any feasible solution, we have $$\sum_{i=1}^k \langle \bmM_i , \bmX_i \rangle \leq \sum_{i=1}^k \langle \bmM_1 , \bmX_i \rangle = \langle \bmM_1 , \sum_{i=1}^k \bmX_i \rangle \leq \langle \bmM_1 , \bmI \rangle = \tr(\bmM_1),$$
        since $\bmM_1 \succcurlyeq \bmM_i$ for all $i$ and $\sum_{i=1}^k \bmX_i \preccurlyeq \bmI$. So $\Xb_i$ are optimal.
        
        We next consider a rank-2 case to show the $\bmM_i$ need not be almost commuting. From the construction above, represent $\bmM_1 = \bmv_1\bmv_1' + \bmv_2 \bmv_2'$ and $\bmM_2 = \bmv_2 \bmv_2'$. \textcolor{black}{Suppose that $\|\bmv_1\| \leq 1$ and $\|\bmv_2\| \leq 1$. It is easy to show $\|\bmM_1 \bmM_2 - \bmM_2 \bmM_1\|_2 = |\bmv_2'\bmv_1|\|\bmv_1 \bmv_2' - \bmv_2 \bmv_1'\|_2 \leq \|\bmv_1\| \|\bmv_2\| \|\bmv_1 \bmv_2' - \bmv_2 \bmv_1'\|_\mathrm{F} = \sqrt{2} \|\bmv_1\|^2 \|\bmv_2\|^2 \sin(\theta) \leq \sqrt{2} \sin(\theta)$, where $\theta$ is the angle between the vectors $\bmv_1$ and $\bmv_2$, and this bound could be as large as $\sqrt{2}$. Thus, $\bmM_1$ and $\bmM_2$ need not be almost commuting.}
    \end{proof}
\end{proposition}
\subsection{Extension to the sum of Brocketts with linear terms}
\label{appendix:sums_brocketts_linear_terms}

Given coefficient matrices and vectors $\{(\bmM_i, \bmc_i)\}_{i=1}^k$, where $\bmc_i \in \bbR^d$ for all $i \in [k]$, suppose the problem in \cref{eq:sum_heterogeneous_quadratics} is augmented with linear terms giving the following optimization problem that appears in \cite{supp:breloyMMStiefel2021}:
\begin{align}
    \max_{\bmU: \bmU'\bmU = \bmI} \sum_{i=1}^k \bmu_i' \bmM_i \bmu_i + \bmc_i' \bmu_i.
\end{align}
  
It is then easy to see for the matrices 
\begin{align}
    \tilde{\M}_i &:= \begin{bmatrix} \bmM_i & \bmc_i \\ \bmc_i' & 0 \end{bmatrix}, \qquad \tilde{\bmX}_i := \begin{bmatrix} \bmX_i & \bmu_i \\ \bmu_i' & 1 \end{bmatrix}, \qquad \bmX_i := \bmu_i \bmu_i'
\end{align}
that $\sum_{i=1}^k \bmu_i' \bmM_i \bmu_i + \bmc_i' \bmu_i = \langle \tilde{\M}_i, \tilde{\bmX}_i \rangle.$ Define $\A := [\bmI_d~~ \bm{0}_d']' \in \bbR^{(d+1) \times d}$ and $\bme_{d+1}$ to be the $d+1$-standard basis vector in $\bbR^{d+1}$. Extending \cref{eq:primal_problem} to the case with linear terms, we obtain a generalized relaxation for the problem:
\begin{align}
    \max_{\tilde{\bmX}_i} &\sum_{i=1}^k \langle \tilde{\M}_i, \tilde{\bmX}_i \rangle \\
    &\text{s.t.}~~\A' \sum_{i=1}^k \tilde{\bmX}_i \A \preccurlyeq \bmI\\
    &\quad ~~ \langle \A \A', \tilde{\bmX}_i \rangle = 1, \quad~~ \bme_{d+1}' \tilde{\bmX}_i \bme_{d+1} = 1 \quad~~ \tilde{\bmX}_i \succeq 0.
\end{align}

By the Schur complement, the constraint $\tilde{\bmX}_i \succeq 0$ guarantees that $\bmX_i - \bmu_i \bmu_i' \succeq 0$ and therefore also $\bmX_i \succeq 0$. The linear operator $\A$ acts to impose  the relevant Fantope-like constraints onto the top-left $d \times d$-size submatrices of the primal variables, and the added constraint on the $(d+1,d+1)^{\text{th}}$ element of each $\tilde{\bmX}_i$ forces it to be 1. For dual variables $\tilde{\bmZ}_i \in \mathbb{S}_+^{d+1}$, $\bmY \in \mathbb{S}_+^d$, $\bmnu \in \bbR^k$, and $\xi \in \bbR$, the KKT conditions are
\begin{align}
        &\tilde{\bmX}_i \succeq 0, \quad \A'\sum_{i=1}^k \bmX_i\A \preceq \bmI, \quad \langle \A\A', \tilde{\bmX}_i \rangle = 1, \quad \bme_{d+1}' \tilde{\bmX}_i \bme_{d+1} = 1\\
        &\A \bmY \A' = \tilde{\M}_i + \tilde{\bmZ}_i - \nu_i \A \A' - \xi \bme_{d+1} \bme_{d+1}',  \qquad \bmY \succeq 0 \label{eq:kkt_linear_term:Y}\\
        & \langle \bmI - \A'\sum_{i=1}^k \bmX_i \A, \bmY \rangle = 0 \\
        & \langle \tilde{\bmZ}_i, \tilde{\bmX}_i \rangle = 0\\
        &\tilde{\bmZ}_i \succeq 0,
\end{align}
which, in fact, are the same KKT conditions as before. If we denote $\bmZ_i := \A' \tilde{\bmZ}_i \A$ to be the top $d+1 \times d+1$ positions of $\tilde{\bmZ}_i$, multiplying \cref{eq:kkt_linear_term:Y} by $\A'$ on the left and $\A$ on the right gives back exactly \cref{KKT:b} for the relaxation in \cref{eq:primal_problem}.

\section{Extended experiments}
\label{appendix:experiments}

\subsection{Assessing the ROP: random PSD \texorpdfstring{$\bmM_i$}{Mi}}

For $\bmM_i$ that are random PSD matrices of rank $k$, we generate the matrix $\bmA \in \bbR^{d \times k}$ with i.i.d. Gaussian samples and compute $\bmM_i = \bmA \bmA'$.

\edit{The table shows the fraction of trials that resulted in rank-one $\bmX_i$ for all $i=1,\hdots,k$. We computed the average error of the sorted eigenvalues of each optimal solution $\Xb_i$ to $\bme_1$, i.e. $\frac{1}{k}\sum_{i=1}^k \|\diag(\bmSigma_i) - \bme_1\|_2^2$ where $\Xb_i = \bmV_i \bmSigma_i \bmV_i'$, and counted any trial with error greater than $10^{-5}$ as not tight. }

\edit{For rank $k=3$, the SDP solutions possessed the ROP in the vast majority of trials. As the rank or dimension increased, the fraction of trials with ROP declined.}

\begin{table}[htbp]

  \centering
  \begin{tabular}{|c|c|c|c|c|c|}
    \hline
    \multicolumn{2}{|c|}{} &\multicolumn{4}{|c|}{\textbf{Fraction of 100 trials with ROP}} \\
     \cline{3-6}
     \multicolumn{2}{|c|}{}  & $k=3$ & $k=5$ & $k=7$ & $k=10$ \\\hline
   \multirow{5}{*}{\rotatebox{90}{\textbf{RandPSD}}}  & $d=10$ & 0.97 &  0.61 &  0.3 &  0.14 \\\cline{2-6}    
    &$d=20$ & 0.92 & 0.48 & 0.13 & 0  \\\cline{2-6}
   & $d=30$ & 0.93 & 0.53 & 0.14 & 0 \\\cline{2-6}
   & $d=40$ & 0.92 & 0.45 & 0.04 & 0  \\ \cline{2-6}
   & $d=50$ & 0.95 & 0.53 & 0.05 & 0 \\
   \hline
   
  \end{tabular}
  \caption{Numerical experiments showing the percentage of trials where the SDP was tight for random synthetic PSD $\bmM_i$ of rank $k$.}
  \label{tbl:trial_counts:rand_psd}
\end{table}

\subsection{Assessing the ROP: HPPCA}

\cref{tbl:trial_counts_hppca:a-supp} and \cref{tbl:trial_counts:hppca:b-supp} display the full experiment results \edit{for $L=2$ related to the} 
abbreviated versions--\cref{tbl:trial_counts_hppca:a} and \cref{tbl:trial_counts:hppca:b}--in \Cref{s:experiments} of the main paper.

\begin{table}[htbp]
\parbox{.45\linewidth}{
\centering
\resizebox{0.45\textwidth}{!}{
\begin{tabular}{|c|c|c|c|c|c|}
    \hline
    \multicolumn{2}{|c|}{} &\multicolumn{4}{|c|}{\textbf{Fraction of 100 trials with ROP}} \\
     \cline{3-6}
     \multicolumn{2}{|c|}{}  & $k=3$ & $k=5$ & $k=7$ & $k=10$ \\\hline
  \multirow{5}{*}{\rotatebox{90}{\textbf{$\bmn = [5,20]$}}}  
 
    &$d=10$&   1  &  0.99   &    1   &    1\\\cline{2-6} 
    &$d=20$&   1  &  0.98  &  0.98 &   0.99\\\cline{2-6} 
    &$d=30$&    0.99  &  0.93  &  0.98  &  0.97\\\cline{2-6} 
    &$d=40$&    0.98  &  0.91  &  0.99 &   0.98\\\cline{2-6} 
    &$d=50$&    0.97 &   0.95 &   0.96  &  0.98\\\cline{2-6} 
    
   \hline  \hline
   \multirow{5}{*}{\rotatebox{90}{\textbf{$\bmn = [10,40]$}}}  &$d=10$&       1   &    1   & 0.99  &     1\\\cline{2-6} 
    &$d=20$&       1   &    1  &  0.98  &  0.99\\\cline{2-6} 
    &$d=30$&       1  &  0.99   & 0.99  &  0.96\\\cline{2-6} 
    &$d=40$&    0.98  &  0.97  &  0.92  &  0.96\\\cline{2-6} 
    &$d=50$&    0.99  &  0.96  &  0.98  &  0.88\\\cline{2-6} 
   \hline \hline
   
   \multirow{5}{*}{\rotatebox{90}{\textbf{$\bmn = [20,80]$}}}   &$d=10$&    1    &    1    &   1    &   1\\\cline{2-6} 
    &$d=20$&    1    &    1    &   1   &    1\\\cline{2-6} 
    &$d=30$&    1   &     1    &   1  &  0.98\\\cline{2-6} 
    &$d=40$&    1    &    1  &  0.97  &  0.95\\\cline{2-6} 
    &$d=50$&    1  &   0.98  &  0.98  &  0.97\\\cline{2-6} 
   \hline \hline
   
   \multirow{5}{*}{\rotatebox{90}{\scriptsize{\textbf{$ \bmn =[50,200]$}}}}   &$d=10$&    1  &   1   &     1   &  1 \\\cline{2-6} 
    &$d=20$&    1  &   1   &     1   &  1 \\\cline{2-6} 
    &$d=30$&    1  &   1   &     1   &  1 \\\cline{2-6} 
    &$d=40$&    1  &   1   &  0.99   &  1 \\\cline{2-6} 
    &$d=50$&    1  &   1   &  0.98   &  1 \\\cline{2-6} 
   \hline \hline
   
   \multirow{5}{*}{\rotatebox{90}{\scriptsize{\textbf{$ \bmn =[100,400]$}}}}  &$d=10$&    1  &   1  &   1   &   1 \\\cline{2-6}
    &$d=20$&    1  &   1  &   1   &   1 \\\cline{2-6}
    &$d=30$&    1  &   1  &   1   &   1 \\\cline{2-6}
    &$d=40$&    1  &   1  &   1   &   1 \\\cline{2-6}
    &$d=50$&    1  &   1  &   1   &   1 \\\cline{2-6}
   \hline
   
  \end{tabular}
  }
\caption{\textbf{(HPPCA)} Numerical experiments showing the percentage of trials where the SDP was tight for instances of the HPPCA problem as we vary $d$, $k$, and $\bmn$ using $L=2$ groups with noise variances $\bmv = [1,4]$.}
  \label{tbl:trial_counts_hppca:a-supp}
}
\hspace{8mm}
\parbox{.45\linewidth}{
\centering
\resizebox{0.45\textwidth}{!}{
\begin{tabular}{|c|c|c|c|c|c|}
    \hline
    \multicolumn{2}{|c|}{} &\multicolumn{4}{|c|}{\textbf{Fraction of 100 trials with ROP}} \\
     \cline{3-6}
     \multicolumn{2}{|c|}{}  & $k=3$ & $k=5$ & $k=7$ & $k=10$ \\\hline
  \multirow{5}{*}{\rotatebox{90}{\textbf{$\bmv = [1,1]$}}}  
 
   &$d=10$&    1  &   1   &     1   &  1 \\\cline{2-6} 
    &$d=20$&    1  &   1   &     1   &  1 \\\cline{2-6} 
    &$d=30$&    1  &   1   &     1   &  1 \\\cline{2-6} 
    &$d=40$&    1  &   1   &  1   &  1 \\\cline{2-6} 
    &$d=50$&    1  &   1   &  1   &  1 \\\cline{2-6} 
   \hline  \hline
   
   \multirow{5}{*}{\rotatebox{90}{\textbf{$\bmv = [1,2]$}}}
   &$d=10$&    1  &   1   &     1   &  1 \\\cline{2-6} 
    &$d=20$&    1  &   1   &     1   &  1 \\\cline{2-6} 
    &$d=30$&    1  &   0.98   &     1   &  1 \\\cline{2-6} 
    &$d=40$&    1  &   1   &  0.99   &  1 \\\cline{2-6} 
    &$d=50$&    1  &   1   &  1   &  0.99 \\\cline{2-6} 
   \hline  \hline
   
   \multirow{5}{*}{\rotatebox{90}{\textbf{$\bmv = [1,3]$}}}   &$d=10$&    1  &   1   &     1   &  1 \\\cline{2-6} 
    &$d=20$&    1  &   1   &     1   &  1 \\\cline{2-6} 
    &$d=30$&    0.99  &   0.99   &    0.97   &  0.99 \\\cline{2-6} 
    &$d=40$&    1  &   0.98   &  0.97   &  0.99 \\\cline{2-6} 
    &$d=50$&    1  &   0.97   &  0.96   &  0.98 \\\cline{2-6} 
   \hline  \hline
   
   \multirow{5}{*}{\rotatebox{90}{{\textbf{$\bmv = [1,4]$}}}}  
   &$d=10$&       1   &    1  &  0.99   &    1\\\cline{2-6} 
    &$d=20$&       1   &    1  &  0.98  &  0.99\\\cline{2-6} 
    &$d=30$&       1  &  0.99  &  0.99  &  0.96\\\cline{2-6} 
    &$d=40$&    0.98  &  0.97  &  0.92  &  0.96\\\cline{2-6} 
    &$d=50$&    0.99  &  0.96   & 0.98  &  0.88\\\cline{2-6} 
   \hline 
   
  \end{tabular}
  }
\caption{\textbf{(HPPCA)} Numerical experiments showing the percentage of trials where the SDP was tight for instances of the HPPCA problem as we vary $d$, $k$, and $\bmv$ using $L=2$ groups with samples $\bmn = [10,40]$.}
  \label{tbl:trial_counts:hppca:b-supp}
}
\end{table}

\newpage
\subsection{Assessing global optimality of local solutions}

\paragraph{Further experiment details} For 100 random experiments of each choice of $\sigma$, we obtain candidate solutions $\Xb_i$ from the SDP and perform a rank-one SVD of each to form $\Ub_{\text{SDP}}$, i.e. {$$\Ub_{\text{SDP}} = [\ub_1 \cdots \ub_k], \quad \ub_i = \argmax_{\bmu : \|\bmu\|_2 = 1} \bmu' \Xb_i \bmu, $$}while measuring how close the solutions are to being rank-1. In the case the SDP is not tight, the rank-1 directions of the $\bmX_i$ will not be orthonormal, so as a heuristic, we project $\Ub_{\text{SDP}}$ onto the Stiefel manifold by its QR decomposition. For comparison, we use the Stiefel majorization-minimization (StMM) solver with a linear majorizer \cite{supp:breloyMMStiefel2021} to obtain a candidate solution $\Ub_{\text{MM}}$ and use \cref{thm:dual_certificate} to certify it either as globally optimal or as a stationary point.

When executing each algorithm in practice, we remark that the results may vary with the choice of user specified numerical tolerances and other settings. For the StMM algorithm, we choose a random initialization of $\bmU$ each trial and run the algorithm either for specified maximum number of iterations or until the gradient on the Stiefel manifold is less than some tolerance threshold; here we set $\texttt{tol} = 10^{-10}$. Using MATLAB's CVX implementation to solve \cref{eq:primal_problem} and \cref{eq:dual_certificate}, we found setting \texttt{cvx\_precision} to \texttt{high} guarantees the best results for returning tight solutions and verifying global optimality. However, iterates of the StMM algorithm that converge close to a tight SDP solution may still not be sufficient for the feasibility LMI to return a positive certificate if the solution is not numerically optimal to a high level of precision.

\printbibliography[heading=subbibliography]
\end{refsection}
\end{document}